\documentclass[12pt]{article}
\usepackage{amsmath,amsthm,amsfonts,amssymb}
\usepackage[all]{xy}
\usepackage[usenames]{color} 

\newcounter{enumitemp}
\newenvironment{enumeratecontinue}{
 \setcounter{enumitemp}{\value{enumi}}
 \begin{enumerate}
 \setcounter{enumi}{\value{enumitemp}}
}
{
 \end{enumerate}
}

\newcommand\pref[1]{(\ref{#1})}

\newtheorem{thm}{Theorem}[section]

\newtheorem{theorem}[thm]{Theorem}

\newtheorem{lemma}[thm]{Lemma}

\newtheorem*{theoremA}{Theorem A}
\newtheorem*{theoremB}{Theorem B}

\newtheorem*{theoremC}{Theorem C}
\newtheorem*{theoremE}{WWPD Construction Theorem}
\newtheorem*{theoremWWPD}{Global WWPD Theorem}
\newtheorem*{corollary*}{Corollary}

\newtheorem{corollary}[thm]{Corollary}
\newtheorem{proposition}[thm]{Proposition}
\newtheorem*{proposition*}{Proposition}

\newtheorem{prop}[thm]{Proposition}

\newtheorem{fact}[thm]{Fact}

\theoremstyle{definition}

\newtheorem{definition}[thm]{Definition} 

\newtheorem*{defn*}{Definition}

\newtheorem{remark}[thm]{Remark}
\theoremstyle{remark}

\newcounter{remarks}
{\paragraph*{Remarks}\smallskip
 \begin{list}{\arabic{remarks}. }{\usecounter{remarks}%
 \setlength{\leftmargin}{0in}%
 \setlength{\rightmargin}{0in}%
 \setlength{\labelsep}{0pt}%
 \setlength{\labelwidth}{0pt}%
 \setlength{\listparindent}{0pt}%
 }
}
{
\end{list}
}

\newcommand\cbdy\delta
\newcommand\from\colon
\newcommand\inv{{-1}}
\newcommand\subgroup{<}
\newcommand\normal\triangleleft
\newcommand\infinity\infty
\newcommand\na{\text{na}}
\newcommand\supp{\text{\!\!\tiny{supp}}}
\newcommand\disjunion\coprod
\newcommand\act\curvearrowright

\DeclareMathOperator{\Fix}{Fix}

\DeclareMathOperator\image{Image}
\DeclareMathOperator\kernel{Ker}

\DeclareMathOperator\Ker{Ker}

\DeclareMathOperator\Isom{Isom}

\newcommand{\R}{{\mathbb R}}
\newcommand\reals{\R}

\newcommand{\Z}{{\mathbb Z}}
\newcommand{\C}{{\mathcal C}}

\renewcommand{\S}{{\mathcal S}}

\newcommand{\M}{{\mathcal M}}
\newcommand{\K}{{\mathcal K}}

\DeclareMathOperator{\Out}{\mathsf{Out}}
\DeclareMathOperator{\Aut}{\mathsf{Aut}}
\DeclareMathOperator{\Inn}{\mathsf{Inn}}
\DeclareMathOperator{\mcg}{\mathsf{MCG}}
\DeclareMathOperator\MCG{\mcg}
\DeclareMathOperator{\Stab}{\mathsf{Stab}}

\newcommand{\F}{\mathcal F}
\newcommand{\rtt}{relative train track map}

\renewcommand\L{\mathcal L}

\def\B{\mathcal B}
\newcommand{\A}{\mathcal A}
\newcommand{\h}{\mathcal H}

\newcommand{\fG} {f : G \to G}
\newcommand{\ti} {\tilde}
\newcommand{\iNp} {indivisible Nielsen path}

\newcommand{\eg}{EG}
\newcommand{\noneg}{NEG}
\renewcommand\neg\noneg

\newcommand{\wt}{\widetilde}

\newcommand{\ct}{CT}
\newcommand{\cts}{CTs}

\newcommand{\comment}[1]{}

\newcommand\BeFujiTag{BestvinaFujiwara:bounded}
\newcommand\BeFuji{\cite{BestvinaFujiwara:bounded}}
\newcommand\BeFeighn{\cite{BestvinaFeighn:HypComplex}}
\newcommand\BHTag{BestvinaHandel:tt}
\newcommand\BH{\cite{\BHTag}}
\newcommand\BookZeroTag{BFH:laminations}

\newcommand\BookOneTag{BFH:TitsOne}
\newcommand\BookOne{\cite{\BookOneTag}}
\newcommand\BookTwoTag{BFH:TitsTwo}

\newcommand\BookThreeTag{BFH:Solvable}

\newcommand\recognitionTag{FeighnHandel:recognition}
\newcommand\recognition{\cite{\recognitionTag}}
\newcommand\abelianTag{FeighnHandel:abelian}

\newcommand\SubgroupsTag{HandelMosher:Subgroups}
\newcommand\Subgroups{\cite{\SubgroupsTag}}

\newcommand\SubgroupsOne{\cite[Part I]{\SubgroupsTag}}
\newcommand\SubgroupsTwo{\cite[Part II]{\SubgroupsTag}}
\newcommand\SubgroupsThree{\cite[Part III]{\SubgroupsTag}}
\newcommand\SubgroupsFour{\cite[Part IV]{\SubgroupsTag}}\newcommand\FSHypTag{HandelMosher:FreeSplittingHyperbolic}
\newcommand\FSHyp{\cite{\FSHypTag}}
\newcommand\FSLoxTag{HandelMosher:FreeSplittingLox}
\newcommand\FSLox{\cite{\FSLoxTag}}
\newcommand\MargulisDiscreteTag{Margulis:DiscreteSubgroups}
\newcommand\GhoshWeakTag{Ghosh:WeakAttraction}
\newcommand\GhoshWeak{\cite{\GhoshWeakTag}}

\newcommand\bdy\partial

\newcommand\intersect\cap
\newcommand\union\cup
\newcommand\<\langle
\renewcommand\>\rangle
\newcommand\meet\wedge
\newcommand\composed{\circ}
\newcommand\cross\times
\newcommand\restrict{\bigm |}
\newcommand\wh{\widehat}
\newcommand\inject\hookrightarrow

\newcommand\abs[1]{\left|#1\right|}
\newcommand\Id{\text{Id}}

\newcommand\injectto\hookrightarrow
\newcommand\injectfrom\hookleftarrow
\newcommand\surjectto\twoheadrightarrow
\newcommand\surjectfrom\twoheadleftarrow

 \newcommand\surjection\twoheadrightarrow
 
\newcommand\suchthat{\bigm|}

\DeclareMathOperator\IA{IA}

\newcommand\IAThree{\IA_n(\Z/3)}

\newcommand\cH{{\cal H}}

\newcommand{\cS}{{\mathcal S}}
\newcommand{\cM}{{\mathcal M}}

\DeclareMathOperator\FS{{\cal FS}}

\newcommand\fscn{{\cal FS}(F_n)}
\DeclareMathOperator\FF{{\cal FF}}

\newcommand\wwpd{WWPD}

\newcommand\shsh{{\#\!\#}}
\newcommand\lsm{\ell^-_s}

\newcommand\itemH{\item$\!\!\!{}_\Gamma$\ }
\newcommand\prefH[1]{\pref{#1}$_{\Gamma}$}

\title{Hyperbolic actions and 2nd bounded cohomology \\of subgroups of $\Out(F_n)$ \\ Part I: Infinite lamination subgroups}

\author{Michael Handel and Lee Mosher \thanks{The first author  was supported by the National Science Foundation under Grant No.~DMS-1308710 and by PSC-CUNY under grants in Program Years 46 and 47. The second author was supported by the National Science Foundation under Grant No.~DMS-1406376.}}

\begin{document}

\maketitle

\begin{abstract}
In this two part work (with sequel  \cite{HandelMosher:BddCohomologyII})
we prove that for every finitely generated subgroup $\Gamma \subgroup \Out(F_n)$, either $\Gamma$ is virtually abelian or $H^2_b(\Gamma;\reals)$ contains an embedding of $\ell^1$. The method uses actions on hyperbolic spaces. Here in Part I we focus on the case of infinite lamination subgroups $\Gamma$ --- those for which the set of all attracting laminations of all elements of $\Gamma$ is an infinite set --- using actions on free splitting complexes of free groups.
\end{abstract}

\section{Introduction}
\label{SectionIntro}

The study of hyperbolic actions --- group actions on Gromov hyperbolic spaces --- has co-evolved with the study of the 2nd bounded cohomology $H^2_b(\Gamma;\reals)$ of a group~$\Gamma$. This started with Brooks' theorem, using the action of a free group $\Gamma$ of rank~$\ge 2$ on its Cayley tree to prove that there is an embedding $\ell^1 \inject H^2_b(\Gamma;\reals)$ \cite{Brooks:H2bRemarks}. In works to follow, this proof was developed in increasing generality for certain proper(ly discontinuous) hyperbolic actions \cite{BrooksSeries:H2bSurface}, \cite{EpsteinFujiwara}, \cite{Fujiwara:H2BHyp}. Fujiwara extended the method to work for certain nonproper actions \cite{Fujiwara:H2bFreeProduct}. 
Bestvina and Fujiwara \BeFuji\ used hyperbolicity of the curve complex $\C(S)$ of a finite type surface $S$ \cite{MasurMinsky:complex1} to prove an ``$H^2_b$-alternative'' for subgroups $\Gamma \subgroup \MCG(S)$ of the mapping class group: either $\Gamma$ is virtually abelian, or there is an embedding $\ell^1 \inject H^2_b(\Gamma;\reals)$. While the action on $\C(S)$ is not proper, nonetheless Bestvina and Fujiwara distilled enough proper discontinuity to extend $H^2_b$-methods further, by introducing the WPD or ``weak proper discontinuity'' property. Later, the \emph{really}~weak but still useful WWPD property was introduced by Bestvina, Bromberg, and Fujiwara \cite{BBF:MCGquasitrees} who used it to study stable commutator length in~$\MCG(S)$ \cite{BBF:SCLonMCG}. For later steps of this co-evolution, using WPD to further the study of hyperbolic actions, see \cite{Bowditch:tight}, \cite{Osin:AcylHyp}, and \cite{BHS:Hierarchy}.

For subgroups $\Gamma \subgroup \Out(F_n)$, Bestvina and Feighn \cite{BestvinaFeighn:HypComplex} produced enough WPD hyperbolic actions to prove the $H^2_b$-alternative if $\Gamma$ contains a fully irreducible outer automorphism; for another proof by Hamenst\"adt see \cite{Hamenstadt:LinesOfMinima}. From our early subgroup decomposition theory \cite{HandelMosher:SubgroupOutF_n}, the same conclusion holds if $\Gamma$ is finitely generated and fully irreducible (see Section~\ref{SectionPPResults}); and by more recent work of Horbez using other methods \cite{Horbez:HandelMosher} one may eliminate the finite generation hypothesis. 

Here is our main result, to be proved over Parts~I and~II of this work:

\begin{theoremA}
Finitely generated subgroups of $\Out(F_n)$ satisfy the $H^2_b$-alternative.
\end{theoremA}
\noindent
Theorem~A is proved using Theorems~B and~C regarding certain hyperbolic actions with WWPD elements. Theorem~B is proved here, and Theorem~C is proved in Part~II \cite{HandelMosher:BddCohomologyII}. Those proofs also depend on WWPD methods found in \cite{HandelMosher:WWPD}\footnote{\cite{HandelMosher:WWPD} was split off from an old version of this work (\texttt{arXiv:1511.06913v4}) to decrease length.} (and see Proposition~\ref{PropWWPDProps}). To state Theorems~B and~C, we first review some definitions.

\smallskip\textbf{WWPD.} Given an action $G \act X$ on a hyperbolic space, an element $\gamma \in G$ is \emph{loxodromic} if its action on the Gromov boundary $\bdy X$ has north-south dynamics, with repeller--attractor pair $\bdy_\pm\gamma = (\bdy_-\gamma,\bdy_+\gamma) \in \bdy X \times \bdy X - \Delta$. Two loxodromic elements $\gamma,\delta \in G$ are said to be \emph{independent} if $\{\bdy_-\delta,\bdy_+\delta\} \intersect \{\bdy_-\gamma,\bdy_+\gamma\} = \emptyset$. Following \BeFuji, the action $G \act X$ is \emph{nonelementary} if there exists an independent pair $\delta,\gamma \in G$ of loxodromic elements (see Proposition~\ref{PropActionWithWWPD} and the following remark). The WWPD property for a loxodromic element $\gamma \in G$ is equivalent to saying that under the natural induced action $G \act \bdy X \cross \bdy X - \Delta$, the orbit $G \cdot \bdy_\pm\gamma$ is a discrete subset of the topological space $\bdy X \cross \bdy X - \Delta$ (see \cite{HandelMosher:WWPD} Proposition~\ref{PropWWPDProps} for the proof of equivalence of the property just stated with the version of WWPD in \cite{BBF:SCLonMCG}).

\smallskip\textbf{Finite and infinite lamination subgroups.} Behind the results of \BeFuji\ are Thurston's decomposition theory for elements of $\MCG(S)$ \cite{FLPKM}, Ivanov's decomposition theory for subgroups of $\MCG(S)$ \cite{Ivanov:subgroups}, and the Masur-Minsky results on hyperbolicity of $\C(S)$ \cite{MasurMinsky:complex1}. Given $\phi \in \MCG(S)$, its lamination set $\L(\phi)$ is a finite set consisting of one unstable lamination for each pseudo-Anosov piece of the Thurston decomposition of~$\phi$. Each subgroup $\Gamma \subgroup \MCG(S)$ has its associated lamination set $\L(\Gamma) = \cup_{\phi \in \Gamma} \L(\phi)$ which is useful in applying and reformulating results of subgroup decomposition theory. For example, Ivanov's results imply that $\Gamma \subgroup \MCG(S)$ is virtually abelian if and only if $\L(\Gamma)$ is finite. 

Behind the proof of Theorem~A are the decomposition theory for elements of $\Out(F_n)$ due to Bestvina, Feighn, and Handel \cite{\BookZeroTag,\BookOneTag,\BookTwoTag,\recognitionTag}, a decomposition theory for abelian subgroups of $\Out(F_n)$ due to Feighn and \hbox{Handel} \cite{\abelianTag}, our decomposition theory for finitely generated subgroups of $\Out(F_n)$ \Subgroups, and our results on hyperbolicity and dynamics of the free splitting complex $\FS(F_n)$ \cite{\FSHypTag,\FSLoxTag}. Associated to each $\phi \in \Out(F_n)$ is its finite set $\L(\phi)$ of \emph{attracting laminations} \cite{\BookOneTag}. Each subgroup $\Gamma \subgroup \Out(F_n)$ has its associated lamination set $\L(\Gamma) = \union_{\phi \in \Gamma} \L(\phi)$. If $\L(\Gamma)$ is finite then $\Gamma$ is a \emph{finite lamination subgroup}, otherwise it is an \emph{infinite lamination subgroup}. Every virtually abelian subgroup of $\Out(F_n)$ is a finite lamination subgroup, but the converse does not hold, unlike in $\MCG(S)$.

\smallskip\textbf{Subgroups of $\IAThree$ with (virtually) abelian restrictions.} In \BeFuji, Bestvina and Fujiwara use Ivanov's results \cite{Ivanov:subgroups} to reduce to the case of an ``irreducible'' subgroup $\Gamma \subgroup \MCG(S)$. We use a somewhat different reduction to special subgroups of $\Out(F_n)$. Recall the finite index, characteristic subgroup 

\medskip
\centerline{$\IAThree = \IA(F_n;\Z/3) = \kernel\biggl(\Out(F_n) \mapsto H_1(F_n;\Z/3) \approx GL(n,\Z/3)\biggr)
$}
\medskip

\noindent
$\IAThree$ has useful invariance properties (Section~\ref{SectionPPResults}): it is torsion free (\cite{BaumslagTaylor:Center}, \cite{Vogtmann:OuterSpaceSurvey}); if $\phi \in \IAThree$ then any $\phi$-periodic conjugacy class of an element or free factor of $F_n$ is fixed by~$\phi$ (see \SubgroupsTwo); every virtually abelian subgroup of $\IAThree$ is abelian \cite{HandelMosher:VirtuallyAbelian}, which we invoke by writing ``(virtually) abelian''. To achieve these properties we often restrict a subgroup to its intersection with $\IAThree$.

For any free factor $B \subgroup F_n$ with conjugacy class $[B]$ and stabilizer subgroup $\Stab[B] \subgroup \Out(F_n)$, 
there is a natural \emph{restriction homomorphism} $\Stab[B] \mapsto \Out(B)$ (see e.g.\ \cite[Section 2.6]{\BookOneTag}, and \SubgroupsOne\ Fact 1.4). A subgroup $\Gamma \subgroup \IAThree$ is said to have \emph{(virtually) abelian restrictions} if for any proper free factor $B \subgroup F_n$ such that $\Gamma \subgroup \Stab[B]$, the restriction map $\Gamma \mapsto \Out(B)$ has (virtually) abelian image; see Corollary~\ref{CorollaryVirtuallyAbelian} for the fact that the image of the natural homomorphism $\IA(F_n;\Z/3) \intersect \Stab[B] \mapsto \Out(B)$ is contained in $\IA(B;\Z/3)$. 

Among finitely generated subgroups of $\IAThree$ with (virtually) abelian restrictions, Theorems~B and~C are concerned, respectively, with infinite lamination and finite lamination subgroups, and their actions on hyperbolic spaces. 

\begin{theoremB}
For any infinite lamination subgroup $\Gamma \subgroup \IAThree$ with (virtually) abelian restrictions, and any maximal, $\Gamma$-invariant, proper free factor system $\A$,
\begin{enumerate}
\item\label{ItemLoxExistsB} 
The action $\Gamma \act \FS(F_n)$ is nonelementary.
\vspace{-.5em}
\item\label{ItemWWPDExistsB}
There exists $\phi \in \Gamma$ such that $\phi \in [\Gamma,\Gamma]$, $\phi$ is fully irreducible rel~$\A$, and $\phi$ is a loxodromic element for the action $\Gamma \act \FS(F_n)$. Furthermore, any such $\phi$ is a WWPD element for that action.
\end{enumerate}
\end{theoremB}

\begin{theoremC} For any finitely generated, finite lamination subgroup $\Gamma \subgroup \IAThree$ which is not (virtually) abelian and has (virtually) abelian restrictions, there is a finite index normal subgroup $N \subgroup \Gamma$ and an action $N \act X$ on a hyperbolic space, such that:
\begin{enumerate}
\item\label{ThmCLoxOrEll} Every element of $N$ acts elliptically or loxodromically on~$X$.
\vspace{-.5em}
\item\label{ThmCNonelementary}
The action $N \act X$ is nonelementary.
\vspace{-.5em}
\item\label{ThmCWWPD}
Every loxodromic element of $[N,N]$ is a strongly axial, WWPD element with respect to the action $N \act X$.
\end{enumerate}
\end{theoremC}
\noindent
To say that a loxodromic element $\phi \in G$ of a hyperbolic action $G \act X$ is \emph{strongly axial} means that there exists a quasi-isometric embedding $\ell \from \reals \to X$ and a homomorphism $\tau \from \Stab(\bdy_\pm\phi) \to \reals$ such that for all $\psi \in \Stab(\bdy_\pm\phi)$ and all $s \in \reals$ we have $\psi(\ell(s)) = \ell(s+\tau(\psi))$. The ``strongly axial'' conclusion in item~\pref{ThmCWWPD} will be applied in Section~\ref{SectionBCDImplyA}, Case~1.

The proof of Theorem~B uses hyperbolicity of the free splitting complex $\FS(F_n)$ \FSHyp\ and properties of the action $\Out(F_n) \act \FS(F_n)$ \FSLox: a description of loxodromic elements; and the fact that every element acts elliptically or loxodromically. The proof of Theorem~C, which is found in the sequel Part~II \cite{HandelMosher:BddCohomologyII}, uses some new hyperbolic actions. Here in Part~I, Theorem~C is used as a black box for purposes of proving Theorem~A (the place where this occurs is Section~\ref{SectionBCDImplyA}, Case~1).

\subsection*{Description and motivation of WWPD methods.} 

Reducing Theorem~A to Theorems~B and~C is carried out in Section~\ref{SectionBCDImplyA}. The reduction proof will apply WWPD methods that are summarized in the Global WWPD Theorem of \cite{HandelMosher:WWPD}, as we now briefly describe.

Bestvina and Fujiwara demonstrate in \BeFuji\ that the WPD property holds for loxodromic elements of the action of the mapping class group $\MCG(S)$ on the curve complex $\C(S)$ of a finite type surface $S$, those elements being the pseudo-Anosov mapping classes. Bestvina and Feighn demonstrate in \cite{BestvinaFeighn:HypComplex} that WPD holds for elements of $\Out(F_n)$ that act loxodromically on the free factor complex $\FF(F_n)$, those elements being the fully irreducible outer automorphisms. In both cases those demonstrations suffice for applying the WPD methods of \BeFuji\ to analyze the second bounded cohomology of a subgroup $\Gamma$ of $\MCG(S)$ or of $\Out(F_n)$ that happens to contain the appropriate type of element. And in \BeFuji\ this is used, in combination with Ivanov's subgroup decomposition theory \cite{Ivanov:subgroups}, to prove the $H^2_b$-alternative for an \emph{arbitrary} subgroup of $\MCG(S)$.

However, fully irreducible elements of $\Out(F_n)$ seem insufficient for understanding the $H^2_b$ alternative for a general subgroup of $\Out(F_n)$. We use the strictly broader class of elements $\phi \in \Out(F_n)$ that act loxodromically on the free splitting complex $\FS(F_n)$ --- equivalently, $\phi$ has an attracting lamination that fills the free group \FSLox. Although acting loxodromically on $\FS(F_n)$ does not imply that $\phi \in \Out(F_n)$ is a WWPD element of the action $\Out(F_n) \act \FS(F_n)$ \FSLox, our proof produces enough loxodromic WWPD elements for applications to the $H^2_b$ alternative of finitely generated subgroups of $\Out(F_n)$, applying WWPD methods as follows.

We shall use a ``global WWPD property'' \cite{HandelMosher:WWPD} that generalizes and abstracts various features of the WPD methods in \BeFuji. This property does not require a hyperbolic action of $\Gamma$ itself, only a hyperbolic action $N \act X$ of a finite index normal subgroup~$N \, \normal \, \Gamma$. And it does not require every loxodromic element of $N$ to satisfy WWPD, only that WWPD holds for each nontrivial element of some Schottky subgroup $F \subset [N,N]$ of the commutator subgroup of~$N$. The global WWPD property also imposes similar conditions on the actions of $N$ that are obtained by pre-composing the given action of $N$ with inner automorphisms of $\Gamma$. For the precise statement of the global WWPD property see Definition~\ref{DefWWPDHypothesis}. 

\begin{theoremWWPD}[\cite{HandelMosher:WWPD}] 
If a group $\Gamma$ satisfies the global WWPD property then $H^2_b(\Gamma;\reals)$ contains an embedded $\ell^1$. 
\end{theoremWWPD}

In \cite{HandelMosher:WWPD}, besides the Global WWPD Theorem one will also find a general theory of WWPD and its relation to second bounded cohomology. Portions of that theory relevant for this paper are summarized here in Propositions~\ref{PropWWPDProps} and~\ref{PropActionWithWWPD}. Also, in some cases the rather intricate Global WWPD Theorem can be avoided by applying a less intricate theorem from 
\cite{HandelMosher:WWPD}; see ``Remark: On avoiding the Global WWPD Theorem'' in Section~\ref{SectionBCDImplyA}.

\subsection*{Methods of proof of Theorem B.} Consider a subgroup $\Gamma \subgroup \IAThree$ satisfying the hypotheses of Theorem~B, so $\Gamma$ is an infinite lamination subgroup with (virtually) abelian restrictions. Consider also a maximal, proper, $\Gamma$-invariant free factor system $\A$, and recall that each element or \emph{component} of $\A$ is the conjugacy class $[A]$ of a certain proper, nontrivial free factor $A$ (see Section~\ref{SectionBasicNotions}). In Section~\ref{SectionReducingBToE}, we attack Theorem~B by using lamination ping-pong methods from \SubgroupsFour\ to produce $\phi \in \Gamma$ satisfying the following properties:  $\phi$ is irreducible relative to~$\A$; $\phi$ is in the commutator subgroup of~$\Gamma$; and $\phi$ has a filling attracting lamination and hence acts loxodromically on $\FS(F_n)$  \FSLox. Since the image of the restriction homomorphism $\Gamma \mapsto \Out(A)$ is abelian, and since $\phi$ is in the commutator subgroup of $\Gamma$, the restriction $\phi \restrict \Out(A)$ is trivial. These arguments are carried out in full detail in Section~\ref{SectionReducingBToE}, with the effect of reducing Theorem~B to the following:

\begin{theoremE}
Let $n \ge 3$, let $\Gamma$ be a subgroup of $\IAThree$ that preserves a (possibly empty) proper free factor system $\A$, and let $\phi \in \Gamma$ have the following properties: 

\medskip\noindent{\rm \bfseries Relative irreducibility:} $\phi$ is irreducible rel $\A$; 

\medskip\noindent{\rm \bfseries Trivial restrictions:} $\phi \restrict \Out(A)$ is trivial for each component $[A] \in \A$; 

\medskip\noindent{\rm \bfseries Filling lamination:} $\phi$ has a filling attracting lamination.  

\medskip\noindent
Then $\phi$ is a WWPD element for the action of $\Gamma$ on $\fscn$.
\end{theoremE}

We do not know whether the conclusion of the WWPD Construction Theorem continues to hold if the hypothesis ``Trivial restrictions'' is dropped. But as our proof of Theorem~B shows, this hypothesis can be verified for appropriate elements of the subgroups occuring in Theorem~B.

\medskip

The proof of the WWPD Construction Theorem is carried out in Section~\ref{SectionTheoremEProof}. That proof depends on results about \emph{well functions} in the free splitting complex, developed in \FSLox\ in analogy to the well functions of Algom-Kfir in the context of outer space \cite{AlgomKfir:Contracting}, and further developed in Section~\ref{SectionWells}. In particular, Section~\ref{SectionTileRegularity} describes regularity properties of attracting laminations which we use in place of measure theoretic properties of currents as applied in \BeFuji\ and \BeFeighn. Underlying Section~\ref{SectionTileRegularity} is a study of uniform splitting properties of \eg\ strata of relative train track maps carried out in Section~\ref{SectionUniformSplitting}, strengthening splitting properties derived in earlier works.

In both Theorem B and the WWPD Construction Theorem, the free factor system~$\F$ is empty if and only if the subgroup $\cH$ is fully irreducible, in which case we recover the result of Bestvina and Feighn \cite{BestvinaFeighn:HypComplex} with a different proof. There might be various ways to relax the assumption that $\F$ be empty. The way expressed in the WWPD Construction Theorem---allowing $\F$ to be nonempty but requiring $\cH$ to be abelian when restricted to each component of~$\F$---was chosen because it is sufficient for purposes of proving Theorem~B, and because the laminations involved in the WWPD Construction Theorem have a particularly simple structure that is easier than general laminations. For a description of that structure see the heading \emph{The topological structure of $\Lambda^-_\phi$}, under Case 2b of Section~\ref{SectionThmECaseAnalysis}.

\subsection*{Application to the Bridson-Wade Theorem}

As a corollary to Theorem~A we prove the following theorem of Bridson-Wade, following the same lines as the analogous proof for mapping class groups found in \BeFuji.

\begin{corollary*}[\cite{BridsonWade:ActionsOnFreeGroups}] If $\Gamma$ is an irreducible lattice in a connected, semisimple Lie group of real rank $\ge 2$ having finite center, then every homomorphism $h \from \Gamma \to \Out(F_n)$ has finite image.
\end{corollary*}

\begin{proof} By the Margulis-Kahzdan normal subgroup theorem \cite[Theorem 8.1.2]{Zimmer:book}, $K=\kernel(h)$ is either finite or of finite index in $\Gamma$. Assuming $K$ is finite we derive a contradiction. 

If $\Gamma$ is a nonuniform lattice then it contains a solvable subgroup $H$ which does not surject with finite kernel onto an abelian group, so the group $H / H \intersect K$ is solvable and not virtually abelian, but it injects to $\Out(F_n)$ contradicting \cite{\BookThreeTag,Alibegovic:translation}. 

If $\Gamma$ is a uniform lattice then it has a free subgroup of rank~$\ge 2$ \cite[Theorem 5.6]{\MargulisDiscreteTag}, so the group $\Gamma/K$ is not virtually abelian. Also, $\Gamma$ is finitely generated. But $\Gamma/K$ injects to $\Out(F_n)$ and so, by Theorem~A, $H^2_b(\Gamma/K;\reals)$ contains an embedded $\ell^1$. In particular there is an unbounded quasimorphism $\Gamma/K \mapsto \reals$ (see e.g.\ \cite{HandelMosher:WWPD}), and by composition we get an unbounded quasimorphism $\Gamma \mapsto \Gamma/K \mapsto \reals$, contradicting the theorem of Burger and Monod \cite[Corollary 1.3]{BurgerMonod:BoundedCohomology}.
\end{proof}

\vfill\break

\subsection*{Outline of the logic}
To summarize, we prove Theorem~A by the following outline:

\smallskip\noindent
$\bullet$ Section~\ref{SectionReducingBToE}: Proof that the WWPD Construction Theorem implies Theorem B.

\smallskip\noindent
$\bullet$  Section~\ref{SectionBCDImplyA}: Proof that Theorems B, C imply Theorem A (applying also the Global WWPD Theorem \cite{HandelMosher:WWPD}).

\smallskip\noindent
$\bullet$  Section~\ref{SectionTheoremEProof}: Proof of the WWPD Construction Theorem.

\smallskip\noindent
$\bullet$ Part II \cite{HandelMosher:BddCohomologyII}: Proof of Theorem~C. 

\bigskip

\setcounter{tocdepth}{2}
\tableofcontents

\section{Background material}
\label{SectionFlaringProperties}

Section~\ref{SectionHypActions} reviews concepts of second bounded cohomology and WWPD from \cite{HandelMosher:WWPD}.
Section~\ref{SectionBasicNotions} reviews basic definitions, terminology, and facts regarding $\Out(F_n)$, referring to the literature for full details. Section~\ref{SectionUniformSplitting} contains Lemma~\ref{stabilizes}, a uniform splitting property which will be used for our study of well functions in the proof of Proposition~\ref{PropKTiles}, and in Part~II~\cite{HandelMosher:BddCohomologyII}.

\subsection{$H^2_b(G;\reals)$ and WWPD methods.}
\label{SectionHypActions}

A \emph{metric complex} is a connected simplicial complex equipped with the geodesic metric making each $k$-simplex isometric to a regular $k$-simplex in $\reals^{k+1}$ of edge lengths equal to~$1$, using barycentric coordinates to define the isometry. A \emph{hyperbolic complex} $X$ is a Gromov hyperbolic metric complex; equivalently, the $1$-skeleton is Gromov hyperbolic. The Gromov boundary is denoted $\bdy X$, and we denote $\overline X = X \union \bdy X$ with its Gromov topology. The space of two point subsets is denoted $\bdy^2 X = \{\{\xi,\eta\} \suchthat \xi \ne \eta \in \bdy X\}$. The topology on $\bdy^2 X$ is induced by the 2-1 covering map $(\bdy X \times \bdy X) - \Delta \mapsto \bdy^2 X$; equivalently, it is the subspace topology induced by the Hausdorff topology on compact subsets of~$\bdy X$. 

\smallskip\textbf{Remark.} Using the Rips complex one easily sees that every geodesic metric space is quasi-isometric to some metric complex \cite{Gromov:hyperbolic}. Since the spaces we use in this work are all metric complexes, we couch our results in that language.

\paragraph{Hyperbolic actions.} An action of a group $\Gamma$ on an object $X$ is given by a homomorphism $\A \from \Gamma \mapsto \Isom(X)$ from $\Gamma$ to the self-isomorphism group of~$X$ (we use this definition primarily for metric complexes and occasionally for objects in other categories). We often use ``action notation'' $\A \from \Gamma \act X$, usually suppressing $\A$ and writing simply $\Gamma \act X$. For $g \in \Gamma$ and a subset $Y \in X$, we use expressions like $g(Y)$ or $g \cdot Y$ for $\A(g)(Y)$. Given an action $\Gamma \act X$, the \emph{stabilizer} of a subset $Y \subset X$ is the subgroup denoted $\Stab(Y) = \Stab(Y;\Gamma) = \{g \in \Gamma \suchthat g \cdot Y = Y\} \subgroup \Gamma$.

A \emph{hyperbolic action} of a group $\Gamma$ is defined to be an isometric action \hbox{$\Gamma \act X$} on a hyperbolic complex~$X$. Any such action extends to a topological action $\Gamma \act \overline X = X \union \bdy X$. According to the classification of isometries of Gromov hyperbolic spaces, each isometry $h \from X \to X$ fits into one of three types: $h$ is \emph{elliptic} if each orbit $\{h^i \cdot x\}$ is of bounded diameter in $X$; $h$ is \emph{loxodromic} if each orbit map $i \mapsto h^i \cdot x$ is a quasi-isometric embedding $\Z \mapsto X$; and $h$ is \emph{parabolic} if it has a unique fixed point on $\bdy X$. 

For $h$ to be loxodromic is equivalent to saying that the extension $h \from \overline X \to \overline X$ has \emph{source-sink} or \emph{north-south dynamics}, meaning that there exists a repelling fixed point $\bdy_- h \in \bdy X$ and an attracting fixed point $\bdy_+ h \in \bdy X$ such that for each $x \in \overline X - \{\bdy_- h,\bdy_+ h\}$ the sequence $h^i(x)$ converges to $\bdy_- h$ as $i \to -\infinity$ and to $\bdy_+ h$ as $i \to +\infinity$. As ordered and as unordered pairs we denote these fixed points as $\bdy_\pm h = (\bdy_- h, \bdy_+ h) \in \bdy X \times \bdy X - \Delta$ and $\bdy h = \{\bdy_- h, \bdy_+ h\} \in \bdy^2 X$. Note also that $h$ is loxodromic if and only if the action of $h$ on $X$ has a \emph{quasi-axis} which is a quasi-isometric embedding $\gamma \from \reals \to X$ such that for some $T > 0$ called the \emph{period} we have $\gamma(t+T)=h(\gamma(t))$, $t \in \reals$; we can always take $\gamma$ to be a bi-infinite edge path in the 1-skeleton of $X$. A quasi-axis is determined by its restriction $\gamma \restrict [0,T]$, subject to the requirement that $h(\gamma(0))=\gamma(T)$; that restriction, and any of the restrictions $\gamma \restrict [t,t+T]$, are called \emph{fundamental domains} for the action of $h$ on its quasi-axis.

There is a classification of hyperbolic actions going back to Gromov's original work \cite[Section 8.2]{Gromov:hyperbolic}; see also \cite{Hamann:GroupActions}, and \cite[Section 3]{Osin:AcylHyp}. Because of our focus on actions which possess WWPD elements, we follow \BeFuji\ in adopting the following simplified classification scheme (see Proposition~\ref{PropActionWithWWPD} and the following remark). Two loxodromic isometries $g,h \from X \to X$ are \emph{independent} if $\bdy h \intersect \bdy g = \emptyset$. The action $\Gamma \act X$ on a hyperbolic complex is \emph{nonelementary} if it contains an independent pair of loxodromic elements, and is \emph{elementary} otherwise. In the elementary case we single out two possibilities: the action is \emph{elliptic} if every element is elliptic; and the action is \emph{axial} if there exist two points $\xi \ne \eta \in \bdy X$ such that $\Gamma = \Stab(\{\xi,\eta\};\Gamma)$. 

\paragraph{Second bounded cohomology} 

Consider a group $G$, its cochain complex $(C^*(G;\reals))$, and the subcomplex of bounded cochains $(C^k_b(G;\reals))$, where the cochain group $C^k(G;\reals)$ is the vector space of all functions $f \from G^k \to \reals$, and $C^k_b(G;\reals)$ is the subspace of bounded functions. The coboundary operator $\delta \from C^k(G;\reals) \to C^{k+1}(G;\reals)$ that is used for these cochain complexes has the form
$$\delta(f)(g_0,g_1,\ldots,g_k) = \sum_{i=0}^{k+1} (-1)^i  f(\pi_i(g_0,g_1,\ldots,g_k))
$$ 
where for each $i=0,\ldots,k+1$ the function $\pi_i \from G^{k+1} \to G^{k}$ is defined on a $(k+1)$-tuple $(g_0,g_1,\ldots,g_k)$ as follows: the function $\pi_0$ deletes the $g_0$ coordinate; the function $\pi_{k+1}$ deletes the $g_k$ coordinate; and for $1 \le i \le k$ the function $\pi_i$ deletes the comma between the $g_{i-1}$ and $g_i$ coordinates and multiplies those coordinates together using the group operation.

\paragraph{Invididual and global WWPD properties.} Given a hyperbolic action $G \act X$ and a loxodromic element $\gamma \in G$, we need to define what it means for $\gamma$ to be a WWPD element of the action. We start with a conceptually simple definition, but in Proposition~\ref{PropWWPDProps} we give several equivalent formulations, including the original definition of Bestvina, Bromberg and Fujiwara \cite{BBF:MCGquasitrees}.

Consider the topological space $\bdy X \cross \bdy X$ and its diagonal $\Delta \subset \bdy X \times \bdy X$. Consider also the orbit of the ordered pair $\bdy_\pm\gamma = (\bdy_- \gamma,\bdy_+ \gamma)$, namely the subset
$$G \cdot \bdy_\pm \gamma = \{\bdy_\pm(\delta \gamma \delta^\inv) \suchthat \delta \in G\} = \{(\delta \cdot \bdy_- \gamma,\delta \cdot \bdy_+ \gamma) \suchthat \delta \in G\} \subset \bdy X \times \bdy X - \Delta
$$ 
We say that $\gamma$ is a WWPD element of the action $G \act X$ if $G \cdot \bdy_\pm \gamma$ is a discrete subset of $\bdy X \cross \bdy X - \Delta$. 

Besides defining WWPD for individual loxodromic elements of a hyperbolic action, we must also formulate a ``global'' WWPD property of a hyperbolic action which abstracts and generalizes the ``global'' WPD methods used in \BeFuji. Our formulation of this global version allows for the possibility that it is not the whole group that acts but only a finite index normal subgroup.

Consider a group $\Gamma$, a finite index normal subgroup $N \normal \Gamma$, and a hyperbolic action $\A \from N \act X$. For each element $g \in \Gamma$, the inner automorphism $i_g \in \Inn(\Gamma)$ defined by $i_g(h)=ghg^\inv$ may be restricted to the normal subgroup $N$, and then composed with~$\A$, to obtain an action which is denoted as~$\A_g = \A \circ i_g \from N \to \Isom(X)$ or in shorthand as $N \act_g X$. 


\begin{definition} 
\label{DefWWPDHypothesis} Given a group $\Gamma$, the \emph{global WWPD property} for $\Gamma$ says that there exists a finite index normal subgroup $N \normal \Gamma$, a hyperbolic action $\A \from N \act X$, and a rank~$2$ free subgroup $F \subgroup N$, such that the following hold:
\begin{enumerate}
\item\label{ItemLoxOrEll}
Each element of $N$ acts either loxodromically or elliptically on $X$.
\item\label{ItemFirstWWPD}
The restricted action $F \act X$ is Schottky and each of its nonidentity elements is WWPD with respect to the action $N \act X$.
\item\label{ItemWWPDOrEllRestrict} For each $g \in \Gamma$, the action $F \act_g X$ obtained by restricting $N \act_g X$ satisfies one of the following two properties:
\begin{enumerate}
\item\label{ItemSchottkyRestriction} $F \act_{g} X$ is Schottky and each of its nonidentity elements is WWPD with respect to the action $N \act_{g} X$; or 
\item\label{ItemEllipticRestriction}
$F \act_{g} X$ is elliptic.
\end{enumerate}
\end{enumerate}
In situations where $N$, its action $N \act X$, and/or $F$ are specified, we shall adopt phrases like ``$\Gamma$ satisfies WWPD'' or ``$\gamma$ satisfies WWPD with respect to $N$'', etc.
\end{definition}
\noindent

Our main result relating WWPD to second bounded cohomology is the following theorem, the proof of which is found in \cite{HandelMosher:WWPD}:

\begin{theoremWWPD}
Given a group $\Gamma$, if the global WWPD property holds for $\Gamma$ then $H^2_b(\Gamma;\reals)$ contains an embedded $\ell^1$. 
\end{theoremWWPD}

\paragraph{A shortcut to verifying the global WWPD property.}
The requirement that $N$ has finite index in $\Gamma$ implies that one need only verify \pref{ItemWWPDOrEllRestrict} for finitely many elements of $\Gamma$:

\begin{lemma}\label{LemmaCosetWWPD}
Item~\pref{ItemWWPDOrEllRestrict} of Definition~\ref{DefWWPDHypothesis} holds if and only if it holds for all $g$ in a set of coset representatives of $N$ in $\Gamma$. 
\end{lemma} 

For the proof and later application of this lemma we set up some notation. Choose coset representatives $g_\kappa$ of $N$ in $\Gamma$, where $\kappa \in \{1,\ldots,K\}$ and $K = [\Gamma:N]$, and by convention choose $g_1 \in N \in \Gamma$ to be the identity. We use $i_\kappa$ and $N \act_\kappa X$ as abbreviations for $i_{g_\kappa}$ and $N \act_{g_\kappa} X$ respectively, and we use $F \act_\kappa X$ for the restriction to~$F$ of $N \act_\kappa X$. We refer to the restricted automorphisms 
$$i_1 \restrict N, \,\, \ldots, \,\, i_K \restrict N \in \Aut(N)
$$ 
as \emph{outer representatives of the inner action of $\Gamma$ on $N$}, which refers to the fact that in the commutative diagram
$$\xymatrix{
\Gamma \ar[rr]^{i} \ar[d] && \Aut(N) \ar[d] \\
\Gamma / N \ar[rr] && \Out(N) \ar@{=}[r] &\Aut(N) / \Inn(N)
}$$
the automorphisms $i_\kappa \in \Aut(N)$ represent all of the elements of the image of the homomorphism $\Gamma / N \to \Out(N)$ (that homomorphism need not be injective, and so there may be some duplication of elements of $\Out(N)$ represented by the list $i_1,\ldots,i_K$, but this is inconsequential). 

\begin{proof}[Proof of Lemma~\ref{LemmaCosetWWPD}] For each coset representative $g_\kappa$ consider another element $h = \nu g_\kappa$ in its coset ($\nu \in N$). We have $i_h = i_\nu \composed i_\kappa \in \Aut(N)$, and so the restricted actions $F \act_\kappa X$ and $F \act_{h} X$ are conjugate by an isometry of $X$, namely $\A(\nu) \from X \to X$. But each of properties~\pref{ItemWWPDOrEllRestrict}(a) and~\pref{ItemWWPDOrEllRestrict}(b) is invariant under isometric conjugation, which proves the lemma.
\end{proof}

\paragraph{Properties of WWPD elements.} We state here a result from \cite{HandelMosher:WWPD} which gives several equivalent formulations of the WWPD property, and some additional consequences of that property. These formulations and properties will be used in various places around this paper, in particular item~\pref{ItemBadRayNo} will play a crucial role in the opening passages of Section~\ref{SectionThmESetup} where we set up the proof of the WWPD Construction Theorem.


\begin{proposition}[\protect{\cite[Proposition~2.6]{HandelMosher:WWPD}}]
\label{PropWWPDProps}
For each hyperbolic action $\Gamma \act X$, and for each loxodromic element $h \in \Gamma$, the following are equivalent:
\begin{enumerate}
\item\label{ItemDefWWPDYes}
\emph{(Generalizing the WPD definition of \cite[Section 3]{\BeFujiTag}, \cite[Definition 2.5]{Osin:AcylHyp})} \\ For every $x \in X$ and $R > 0$ there exists an integer $M > 0$ such that any subset $Z \subset \Gamma$ that satisfies the following two properties is finite:
\begin{itemize}
\item For each $g \in Z$ we have $d(x,g(x)) < R)$ and $d(h^M(x),g h^M(x)) < R$.
\item No two elements of $Z$ lie in the same left coset of $\Stab(\bdy_\pm h)$.
\end{itemize}
\item\label{ItemEqDiscYes}
\emph{(The WWPD definition of Section~\ref{SectionHypActions})} The orbit $\Gamma \cdot \bdy_\pm \gamma$ is a discrete subset of the space $\bdy X \times \bdy X - \Delta$.
\item\label{ItemWWPDQuasiAxes}
\emph{(The original WWPD definition of \cite{BBF:MCGquasitrees})}
For any quasi-axis $\ell$ of $h$ there exists $D \ge 0$ such that for any $g \in \Gamma$, if $g \not\in \Stab(\bdy_\pm h)$ then the image of any closest point map $g(\ell) \mapsto \ell$ has diameter~$\le D$.
\item\label{ItemBadRayNo}
\emph{(A variant of \pref{ItemDefWWPDYes})}
In the group $\Gamma$ \emph{there is NO} infinite sequence $g_1,g_2,g_3,\ldots$ satisfying the following properties:
\begin{enumerate}
\item\label{ItemBadRayDiffCosets}
For all $i \ne j$ the elements $g_i,g_j$ lie in different left cosets of $\Stab(\bdy_\pm h)$.
\item\label{ItemBadRayUnmoved}
For every $x \in X$ there exists $R>0$ such that for all integers $M \ge 0$ there exists an integer $I \ge 0$ such that if $0 \le m \le M$ and if $i \ge I$ then \hbox{$d(g_i \, h^m(x), h^m(x)) < R$.}
\end{enumerate}
\end{enumerate}
Furthermore if $h$ satisfies WWPD then:
\begin{enumeratecontinue}
\item\label{ItemWWPDNiceEnds}
$\Stab(\bdy_- h) = \Stab(\bdy_+ h) = \Stab(\bdy_\pm h)$. In particular, for every loxodromic element $\gamma \in H$, the set $\bdy \gamma$ is either equal to or disjoint from $\bdy h$. 
\item\label{ItemWWPDNoStabInd}
If $k \in \Gamma - \Stab(\bdy h)$ then $h$ and $khk^\inv$ are independent.
\end{enumeratecontinue}
\end{proposition}

We also need the following special version of the dynamical classification of hyperbolic actions (\cite[Sections 8.2, 8.3]{Gromov:hyperbolic}; see also \cite[Section 3]{Osin:AcylHyp} and \cite{Hamann:GroupActions}).


\begin{proposition}[\protect{\cite[Corollary 2.6]{HandelMosher:WWPD}}]
\label{PropActionWithWWPD}
For each hyperbolic action $\Gamma \act X$, if $\Gamma$ contains a loxodromic WWPD element then one of the following alternatives holds:
\begin{enumerate}
\item The action is nonelementary, i.e.\ it possesses an independent pair of loxodromic elements; or
\item\label{ItemAxialConclusion} The action is axial, meaning that there exist $\xi \ne \eta \in \bdy X$ such that \break \hbox{$\Gamma = \Stab(\{\xi,\eta\};\Gamma)$.}
\end{enumerate} 
\end{proposition}

\paragraph{Remark:} Because of Proposition~\ref{PropActionWithWWPD}, if a WWPD loxodromic element exists then the definition of ``nonelementary'' that we use following \BeFuji\ is in agreement with the definition in \cite{Gromov:hyperbolic}. For general hyperbolic actions the definition in \cite{Gromov:hyperbolic} also allows for nonelementary groups which are ``quasiparabolic''; see \cite[Corollary 2.6]{HandelMosher:WWPD}. Theorems~B and~C are written so that their meanings are unchanged whether one uses the definition of nonelementary from \BeFuji\ or from \cite{Gromov:hyperbolic}: both theorems assert the existence of a WWPD loxodromic element.

\paragraph{Remark:} When Proposition~\ref{PropActionWithWWPD} is applied in Section~\ref{SectionBCDImplyA}, 
the ``axial'' conclusion~\pref{ItemAxialConclusion} will not be strong enough. In addition we will need the action to be ``strongly axial'', a property which will be supplied by application of Theorem~C.

\subsection{Basic concepts of $\Out(F_n)$}
\label{SectionBasicNotions}

In this subsection we briefly review basic terminology and notation and key definitions, focussing on relative train track maps. The original sources for this material include \cite{CullerVogtmann:moduli}, \BH, \BookOne, and \recognition. See Section 1 of \SubgroupsOne\ for full definitions and citations in a still brief but much more comprehensive overview.

\smallskip\textbf{Marked graphs and topological representatives.} 
\BH. 
A \emph{rank~$n$ marked graph} is a finite graph $G$ without valence~1 vertices equipped with a homotopy equivalence $R_n \mapsto G$, where $R_n$ is the base rose of rank~$n$ whose edges are indexed and oriented so as to identify $F_n \approx \pi_1(R_n)$. The marking induces an isomorphism $F_n \approx \pi_1(G)$, a deck transformation action $F_n \act \wt G$ on the universal cover, and an equivariant homeomorphism $\bdy F_n \approx \bdy\wt G$, all well-defined up to inner automorphism of~$F_n$. A homotopy equivalence $f \from G \to G'$ of marked graphs is always assumed to take vertices to vertices and to be locally injective on each edge, and so $f$ induces a map of directions (initial germs of oriented edges) and a map of turns (pairs of directions at the same vertex). Given a marked graph $G$ there is an induced isomorphism between the group of self-homotopy equivalences of $G$ modulo homotopy and the group $\Out(F_n)$. A self-homotopy equivalence $f \from G \to G$ that represents $\phi \in \Out(F_n)$ in this way is called a \emph{topological representative} of $\phi$; in such a case, iterating the induced map on turns, a turn is \emph{illegal} if some iterate is degenerate, and it is \emph{legal} otherwise.

\smallskip\textbf{Paths, lines, and circuits.}  \BookOne. A \emph{path} in a marked graph $G$ is a concatenation of edges and partial edges without backtracking, such that a partial edge may occur only incident to an endpoint of the path; in the degenerate case, a \emph{trivial path} consists of just a vertex. An infinite path in $G$ can be a line or a ray depending on whether the concatenation is doubly infinite or singly infinite. In some situations we must restrict to paths with endpoints at vertices, e.g.\ for defining complete splitting of paths; in other situations we must allow for more general paths with arbitrary endpoints, e.g.\ when taking pre-images of paths. In other situations we will use more general \emph{edge-paths} in which backtracking is allowed. We will use appropriate language to indicate such situations.

A \emph{circuit} is a cyclic concatenation of edges without backtracking, i.e.\ a circle immersion. 

Two paths or circuits are \emph{equivalent} if they differ only by reparameterizaiton and inversion. The space of paths and circuits up to equivalence is denoted $\wh\B(G)$, and is equipped with the \emph{weak topology} having a basis element for each finite path~$\alpha$, consisting of all paths and circuits having $\alpha$ as a subpath. The subspace of lines is denoted $\B(G) \subset \wh\B(G)$. The abstract space of lines of the group~$F_n$ is the quotient space $\B=\B(F_n) = \wt\B(F_n) / F_n$ where $\wt \B(F_n)$ is the space of two point subsets of $\bdy F_n$, endowed with the \emph{weak topology} naturally induced by the topology of $\bdy F_n$. The identification $\bdy F_n \approx \bdy \wt G$ induces a canonical homeomorphism $\B(F_n) \approx \B(G)$. The element of $\B(G)$ corresponding to a given $\ell \in \B(F_n)$ is called the \emph{realization} of $\ell$ in~$G$. The natural action $\Aut(F_n) \act \bdy F_n$ induces an action \hbox{$\Out(F_n) \act \B$}. 

\smallskip\textbf{Path maps.}  \SubgroupsOne, Section 1.1.6. Consider a homotopy equivalence of possibly distinct marked graphs $f \from G \to G'$ that represents $\phi \in \Out(F_n)$ in a manner similar to a topological representative: the composition $R_n \mapsto G \xrightarrow{f} G' \mapsto R_n$ defined using marking maps or their homotopy inverses is homotopic to a topological representative of $\phi$. There is an induced path map $f_\# \from \wh\B(G) \to \wh\B(G')$, where $f_\#(\gamma)$ is obtained from the composition $f \circ\gamma$ by straightening rel endpoints, lifting first to the universal cover when $\gamma$ is infinite in order to make use of endpoints at infinity. For each line in $\B$ realized as $\ell \in \B(G)$ its image $\phi(\ell) \in \B$ is realized as $f_\#(\ell) \in \B(G')$. 

There is another induced path map $f_{\shsh} \from \wh\B(G) \to \wh\B(G')$ defined as follows: for any path $\gamma \in \wh\B(G)$, lift to a path $\ti\gamma$ in the universal cover~$\wt G$, choose any lift $\ti f \from \wt G \to \wt G'$, let $\ti f_\shsh(\ti\gamma)$ be the intersection of all paths in $\wt G'$ that contain~$\ti \gamma$ as a subpath, and let $f_\shsh(\gamma)$ be the projection to $G'$ of $\ti f_\shsh(\ti\gamma)$, which is well-defined independent of choices. This path $f_\shsh(\gamma)$ is obtained from $f_\#(\gamma)$ by truncating a subpath incident to each endpoint of length at most equal to the bounded cancellation constant of $f$. 

Assuming now that $G=G'$, a finite path $\gamma$ is a \emph{Nielsen path} if $f_\#(\gamma) = \gamma$. If a Nielsen path can not be written as a concatenation of two  non-trivial Nielsen paths then it is \emph{indivisible}.

\smallskip\textbf{Splittings of paths.} \BookOne.
A decomposition $\gamma = \gamma_1 \cdot \ldots \cdot \gamma_K$ of a finite path into subpaths is a \emph{splitting} if $f^i_\#(\gamma) = f^i_\#(\gamma_1) \ldots f^i_\#(\gamma_K)$ for all integers $i \ge 1$ (the single dot $\cdot$ in expressions like $\gamma_1 \cdot \ldots \cdot \gamma_K$ \emph{always} refers to a splitting).  If $P$ is a set of finite paths and if each $\gamma_i \in P$ then we say that  $\gamma = \gamma_1 \cdot \ldots \cdot \gamma_K$ is a \emph{splitting of $\gamma$ with terms in~$P$}. This concept of ``splitting with terms in~$P$'' is most useful if for each $\alpha \in P$ the path $f_\#(\alpha)$ splits with terms in~$P$; see for example the definition of a \ct\ \recognition.

\smallskip\textbf{Free factor systems and free factor supports.} \cite{\BookOneTag}. Conjugacy classes in $F_n$ are denoted $[\cdot]$. A \emph{subgroup system} in $F_n$ is a set $\A = \{[A_1],\ldots,[A_K]\}$ such that $A_1,\ldots, A_K \subgroup F_n$ are nontrivial finitely generated subgroups (\SubgroupsOne\ Section~1.1.2). The elements of a subgroup system are called its \emph{components}. The partial ordering of \emph{containment} $\A \sqsubset \A'$ means that for all $[A] \in \A$ and $[A'] \in \A'$ there exists $g \in F_n$ such that $A \subgroup gA'g^\inv$. 

A subgroup system $\A$ is a \emph{free factor system} if there exists a free factorization $F_n = A_1 * \cdots * A_K * B$ such that $\A = \{[A_1],\ldots,[A_K]\}$. Associated to each subgraph $H \subset G$ of a marked graph is the free factor system $[H] = [\pi_1 H] = \{[\pi_1 H_1],\ldots,[\pi_1 H_K]\}$, where $H_1,\ldots,H_K$ are the noncontractible components of $H$ and $[\pi_1 H_k]$ denotes the conjugacy class of the image of the inclusion induced monomorphism $\pi_1(H_k) \inject \pi_1(G) \approx F_n$. Every chain $\A_1 \sqsubset \cdots \sqsubset \A_L$ of free factor systems can be realized in this manner as a chain $G_1 \subset \cdots \subset G_L \subset G$ of subgraphs of some marked graph~$G$.

A line $\ell \in \B(F_n)$ is \emph{supported by} or \emph{carried by} a subgroup system $\A$ if there exists $\ti\ell \in \wt\B(F_n)$ covering $\ell$ and $A \subgroup F_n$ such that $[A] \in \A$ and such that $\bdy\ti\ell \subset \bdy A$; let $\B(\A) = \{\ell \in \B(F_n) \suchthat \text{$\ell$ is carried by $\A$}\}$; equivalently, if $\A$ is realized by a subgraph $H \subset G$ then the realization of $\ell$ in $G$ is contained in~$H$. 

The \emph{free factor support} of a subset $L\subset\B(F_n)$, denoted $\F_\supp(L)$, is the unique minimal free factor system $\F_\supp(L)$ which supports each element of~$L$. Also, given a subgroup system $\A$ we use the shorthand notation $\F_\supp(L;\A) = \F_\supp(L \union \B(\A))$, and we call this the \emph{free factor support of $L$ relative to $\A$}. Also, we say that $L$ \emph{fills relative to~$\A$} if $\F_\supp(L;\A) = \{[F_n]\}$.

\smallskip\textbf{Relative train track maps.} \BH. A topological representative $f \from G \to G$ is \emph{filtered} if it comes equipped with an $f$-invariant filtration of subgraphs $\emptyset = G_0 \subset G_1 \subset \cdots \subset G_R = G$, with corresponding subgraphs $H_r = G_k \setminus G_{r-1}$ called \emph{strata}, such that if $H_r$ has $m$ edges with corresponding $m \times m$ transition matrix $M$ then either $H_r$ is an \emph{irreducible stratum} meaning that $M$ is irreducible or $H_r$ is a \emph{zero stratum} meaning that $M$ is a zero matrix. Irreducible strata $H_r$ are further classified by the value of the Perron-Frobenius eigenvalue $\lambda \ge 1$ of $M$: $H_r$ is an \emph{\eg-stratum} if $\lambda>1$; and $H_r$ is an \emph{\neg-stratum} if $\lambda=1$. An \eg\ stratum is \emph{aperiodic} if $M$ is Perron-Frobenius (some power of $M$ is positive).

For $f \from G \to G$ to be a \emph{relative train track} representative requires some conditions to be imposed on its \eg\ strata. The definition and some further properties are found in \BH\ Section 5. We will use only the following properties of an \eg\ stratum $H_r$:
\begin{description}
\item[(RTT-(i))] Each initial direction of an edge in $H_r$ maps to some initial direction of an edge in $H_r$.
\end{description}
Paths in $G_r$ whose only turns in $H_r$ are legal turns are called \emph{$H_r$-legal paths} or simply \emph{$r$-legal paths}. An edge in $H_r$, having no turns at all, is by default $H_r$-legal.
\begin{description}
\item[(\BH\ Lemma 5.8)]  If $\sigma$ is an $r$-legal path in $G_r$ with endpoints in $H_r$ then $\sigma$ splits with terms in 
$$\{\text{edges of $H_r$}\} \union \{\text{nontrivial subpaths of $G_{r-1}$ with endpoints in $H_r$}\}
$$
and $f_\#(\sigma)$ is $r$-legal.
\end{description}
Paths of the form $f^k_\#(E)$, for $k \ge 0$ and edges $E \subset H_r$, are called \emph{$k$-tiles} of $H_r$ (note that when $f^k(E)$ is straightened to form $f^k_\#(E)$, no edges of $H_r$ are cancelled). 

Every $\phi \in \Out(F_n)$ has a relative train track representative \BH. Furthermore, some positive power $\phi^k$ has a relative train track representative $f \from G \to G$ which is \emph{\eg-aperiodic}, meaning that the transition matrix of each \eg\ stratum is a Perron-Frobenius matrix; one may take $f$ to be the straightened $k^{\text{th}}$ power of any relative train track representative of $\phi$ itself, with a refined filtration.

\smallskip\textbf{Attracting Laminations.} \BookOne. Consider $\phi \in \Out(F_n)$ and an \eg-aperiodic relative train track representative $f \from G \to G$ of some positive power $\phi^i$. Associated to each \eg\ stratum $H_r \subset G$ is its \emph{attracting lamination} $\Lambda_r$, a closed subset of $\B(F_n)$ whose realization in $G$ consists of all lines $\ell$ such that each subpath of $\ell$ is contained in some $k$-tile of~$H_r$; equivalently, the lines of $\Lambda_r$ are the weak limits in $\wh\B(G)$ of sequences of $k$-tiles. The lamination $\Lambda_r$ is also characterized as the closure of some birecurrent, nonperiodic line $\ell \in \B \approx \B(G)$ of height $r$ such that for some weak neighborhood $U \subset \B$ of $\ell$ we have $\phi(U)\subset U $ and $\{\phi_\#^i(U) \suchthat i \ge 0\}$ is a neighborhood basis of~$\ell$; each such line $\ell$ is in $\Lambda_r$, $\ell$ is called a \emph{generic leaf} of $\Lambda_r$, and each such neighborhood $U$ is an \emph{attracting neighborhood} of a generic leaf. The set $\L(\phi) = \{\Lambda_r \suchthat \text{$H_r$ is an \eg\ stratum}\}$ is well-defined independent of the choice of $f \from G \to G$, and it is a finite set.

The elements of the set $\L(\phi)$ are distinguished by their free factor supports: if $\Lambda_1 \ne \Lambda_2 \in \L(\phi)$ then $\F_\supp(\Lambda_1) \ne \F_\supp(\Lambda_2)$. The elements of the two sets $\L(\phi)$ and $\L(\phi^\inv)$ correspond one-to-one according to their free factor supports: the relation $\F_\supp(\Lambda^+) = \F_\supp(\Lambda^-)$ between $\Lambda^+ \in \L(\phi)$ and $\Lambda^- \in \L(\phi^\inv)$ is a bijection. We will refer to the ordered pair $\Lambda^\pm = (\Lambda^+,\Lambda^-)$ as a \emph{(dual) lamination pair} of $\phi$, and we write their common support as $\F_\supp(\Lambda^\pm)$. The set of all lamination pairs of $\phi$ is denoted~$\L^\pm(\phi)$.

Given a lamination pair $\Lambda^\pm \in \L^\pm(\phi)$ which is fixed by~$\phi$, a line $\ell \in \B$ is \emph{weakly attracted} to $\Lambda^+$ (by iteration of $\phi$) if $\phi^k(\ell)$ weakly converges to a generic leaf of $\Lambda^+$. More details of weak attraction are found in Section~\ref{SectionPPResults}.

\smallskip\textbf{\eg\ properties of \cts.}  \recognition. There is a particularly nice kind of \rtt\ called a \emph{CT}. The necessary and sufficient condition for $\phi$ to be represented by a \ct\ is that $\phi$ be \emph{rotationless} \cite[Definition 3.13]{\recognitionTag}.  Every $\phi$ has a rotationless iterate with uniformly bounded exponent \cite[Theorem 4.28]{\recognitionTag} so one can often work exclusively with \ct s. The defining conditions for $f$ to be a \ct\ are found in \cite[Definition 4.7]{\recognitionTag}, a sequence of nine conditions with parenthesized names (some mentioned here). Most of our applications will be to a \ct\ $f \from G \to G$ with an $f$-invariant filtration $\emptyset = G_0 \subset G_1 \subset \cdots \subset G_s=G$ satisfying the following:

\smallskip$\bullet$ \emph{(Special Assumption)} The top stratum $H_s$ is \eg-aperiodic.

\smallskip\noindent
For now, instead of a full and general definition, we only list certain properties of a \ct\ $f \from G \to G$ which hold under the special assumption above:
\begin{enumerate}
\item \cite[Definition 4.7 (Filtration)]{\recognitionTag} There is no $\phi$-periodic free factor system strictly contained between the free factor systems $[\pi_1 G_{s-1}]$ and $[\pi_1 G_s] = \{[F_n]\}$.
\item\label{ItemCTiNP}
 \cite[Corollary 4.19 eg-(i)]{\recognitionTag} There exists (up to reversal) at most one indivisible Nielsen path $\rho$ of height~$s$. If $\rho$ exists then it decomposes as $\rho = \alpha \beta$ where $\alpha$ and $\beta $ are $s$-legal paths with endpoints at vertices and  $(\bar \alpha, \beta)$ is an illegal turn in $H_s$.   In particular, $f(\alpha) = \alpha \gamma$ and $f(\beta) =  \bar \gamma \beta$ for some path  $\gamma$.  Moreover:
\begin{itemize}
\item At least one endpoint of $\rho$ is disjoint from $G_{s-1}$ (see also \SubgroupsOne\ Fact 1.42)
\item The initial and terminal directions of $\rho$ are distinct fixed directions in $H_s$.
\end{itemize}
\item (\cite[Lemma 5.10]{\BHTag})  Any assignment $\ell(E)$ of lengths to the edges $e$ of $H_s$ extends to an assignment of lengths to all paths in $G$ with endpoints at vertices by adding up the lengths of the edges of $H_s$ in the edge-path description of $\sigma$. There is an \emph{eigenvector assignment of lengths} $\ell(e)$ so that  $\ell(f(e)) = \lambda \ell(e)$ for all edges $e$ of $H_s$ where $\lambda> 1$ is the  Perron-Frobenius eigenvalue of the transition matrix for $H_s$. If $\rho=\alpha\beta$ exists as in \pref{ItemCTiNP} it follows that $\ell(\alpha) = \ell(\beta)$ because $\ell(\alpha)$ and $\ell(\beta)$ both satisfy $\lambda L = L + \ell(\gamma)$.
\item\label{ItemCTSplitSimple}
(\cite[Part I, Fact 1.35]{HandelMosher:Subgroups}) If $\gamma$ is a path with endpoints at vertices $H_s$ or a circuit crossing an edge of $H_s$ then for all sufficiently large $i$ the path $f^i_\#(\gamma)$ has a splitting with terms in the set $\{$edges of~$H_s\} \union \{$height $s$ indivisible Nielsen paths$\}\union\{$paths in $G_{s-1}$ with endpoints in~$H_s\}$. 
\end{enumerate}
\emph{Remark on item \pref{ItemCTSplitSimple}:} This is implicit in \BookOne, e.g.\ Lemma 4.2.6, Remark 5.1.2, and Step 2 of the proof of Proposition 6.0.4 of \BookOne.

\smallskip\textbf{Geometricity of \eg\ strata and attracting laminations.} 
\eg\ strata are further classified into geometric and nongeometric strata, as defined and studied in \cite[Sections 5.1, 5.3]{\BookOneTag} and further studied in \SubgroupsOne. Very roughly speaking, an \eg\ stratum $H_s \subset G$ is geometric if the restriction of $f$ to $G_s$ may be modeled up to homotopy by a self-homotopy-equivalence of a 2-complex that is obtained by attaching a surface $S$ to $G_{s-1}$, such that the map restricted to $S$ is pseudo-Anosov. For now we need only the following facts:
\begin{enumeratecontinue}
\item\label{ItemCTGeomNielsen}
 (\SubgroupsOne, Fact 2.3) $H_s$ is geometric if and only if there is a closed height~$s$ individual Nielsen path.
\item (\SubgroupsOne, Section 2.4) Geometricity is a well-defined invariant on a lamination pair $\Lambda^\pm \in \L(\phi)$, independent of the choice of a \ct\ representative of a rotationless power of $\phi$. To be precise, if $f \from G \to G$ and $f' \from G' \to G'$ are \ct\ representatives of rotationless positive or negative powers of $\phi$, with \eg\ strata $H_i \subset G$, $H'_j \subset G'$ corresponding to either $\Lambda^+$ or $\Lambda^-$ depending on the sign of the power,  then $H_i$, $H'_j$ are either both geometric or both nongeometric.
\end{enumeratecontinue}

\subsection{A uniform splitting lemma}
\label{SectionUniformSplitting}

In this section we consider a rotationless $\phi \in \Out(F_n)$ represented by a \ct\ $f \from G \to G$ with an \eg\ stratum $H_s$ and a corresponding attracting/repelling lamination pair $\Lambda^\pm_s$ each realized in~$G$. If there exists an indivisible Nielsen path of height~$s$ (which must be unique up to reversal), denote that path as $\rho$; otherwise, ignore~$\rho$.

Lemma~\ref{stabilizes} will uniformize the splitting property that was reviewed in Section~\ref{SectionBasicNotions}, under item~\pref{ItemCTSplitSimple} of the heading ``\eg\ properties of \cts''. That property said that for each finite path $\sigma$ in $G$ of height~$\le s$ having endpoints at vertices of~$H_s$, some iterate $f^d_\#(\sigma)$ splits into terms each of which is either an edge of~$H_s$, an indivisible Nielsen path of height~$s$, or a path in $G_{s-1}$ with endpoints at vertices of~$H_s$. The exponent $d$ needed for this splitting is unbounded in general, as one sees for example by letting $\sigma$ be a leaf segment of $\Lambda^-_s$ that crosses a large number of edges of $H_s$. The following lemma says that such examples are the \emph{only} reason for the splitting exponent $d$ to be unbounded. 

For any path $\sigma \subset G$ of height~$\le s$, let $\abs{\sigma}_s$ denote the number of times that $\sigma$ crosses edges of~$H_s$. Also, let $\lsm(\sigma)$ denote the maximum of $\abs{\tau}_s$ over all paths $\tau$ in $G$ such that $\tau$ is a subpath both of $\sigma$ and of a generic leaf of $\Lambda^-_s$.

\begin{lemma} \label{stabilizes} 
With notation as above, for all $L > 0$ there is a positive integer $d$ so that if $\sigma \subset G$ is a circuit or a finite path with endpoints at vertices of $H_s$, and if $\lsm(\sigma) < L$, then $f^d_\#(\sigma)$ splits into terms each of which is an edge of $H_s$, an indivisible Nielsen path of height $s$, or a subpath of $G_{s-1}$ with endpoints at vertices; equivalently, $f^d_\#(\sigma)$ splits into terms each of which is $s$-legal or an indivisible Nielsen path of height~$s$.
\end{lemma} 

\begin{proof} We give the argument when $\sigma$ is a finite path with endpoints at vertices of~$H_s$; the argument for circuits is almost the same.

Arguing by contradiction, if the lemma fails then there exists $L > 0$ so that for all even integers $2i$ there is a finite path $\sigma_i \subset G$ with endpoints at vertices and with $\lsm(\sigma_i) < L_1$ such that $f^{2i}_\#(\sigma_i)$ does not split into terms each of which is either $s$-legal or is an indivisible Nielsen path of height~$s$. Let $\beta_i = f^i_\#(\sigma_i)$, and let $\beta_i = \beta_{i,1} \cdot \ldots \cdot\beta_{i,J_i}$ be a maximal splitting of $\beta_i$ into subpaths with endpoints at vertices. There is a splitting $\sigma_i = \sigma_{i,1} \cdot \ldots \cdot \sigma_{i,J_i}$ with $f^i_\#(\sigma_{i,j}) = \beta_{i,j}$; the paths $\sigma_{i,j}$ need not have endpoints at vertices.

If $\abs{\beta_{i,j}}_s$ is bounded independently of $i$ and $j$ then by \SubgroupsOne\ Lemma 1.54, there exists a positive integer $d$ independent of $i,j$ such that each $f^d_\#(\beta_{i,j})$ splits into subpaths that are either $s$-legal or \iNp s of height $s$. The same is then true for $f^i_\#(\beta_{i,j})$ if $i \ge d$, in contradiction to the fact that $f^{2i}_\#(\sigma_i) = f^i_\#(\beta_{i,1})\cdot \ldots \cdot f^i_\#(\beta_{i,J_i})$ has no such splitting. 

After passing to a subsequence we may therefore choose for each $i$ an integer $j(i) \in \{1,\ldots,J_i\}$ such that $\abs{\beta_{i,j(i)}}_s \to \infinity$ as $i \to \infinity$. Since $\beta_{i,j(i)}$ does not split at any interior vertex, there is a uniform bound to the number of edges of $H_s$ in an $s$-legal subpath $\beta_{i,j(i)}$. It follows that the number of illegal turns of height $s$ in $\beta_{i,j(i)}$ goes to infinity with~$i$. This in turn implies that $\abs{\sigma_{i,j(i)}}_s$ goes to infinity with~$i$, because $\sigma_{i,j(i)}$ has at least as many height $s$ illegal turns as $\beta_{i,j(i)}$.

It follows that some weak limit of the sequence $\sigma_{i,j(i)}$ is a line of height~$s$, contained in $G_s$ and containing at least one edge of~$H_s$. Consider a height~$s$ line $\ell$ which is a weak limit of a subsequence of $\sigma_{i,j(i)}$. If $H_s$ is non-geometric then by the weak attraction theory of \SubgroupsThree, the line $\ell$ has one of two options: either the weak closure of $\ell$ contains $\Lambda^-$; or $\ell$ is weakly attracted to $\Lambda^+$ (\SubgroupsThree\ Lemma 2.18). If $H_s$ is geometric --- and so $\rho$ exists and is closed --- then there is a third option, namely that $\ell$ is a bi-infinite iterate of $\rho$ or $\bar \rho$ (\SubgroupsThree\ Lemma 2.19)

Since no $\sigma_i$ contains a subpath of $\Lambda^-$ that crosses $L$ edges of $H_s$, neither does~$\ell$, and so the weak closure of $\ell$ does not contain~$\Lambda^-$, ruling out the first option. Suppose $\ell$ is weakly attracted to $\Lambda^+$. There exists $m > 0$ such that $f^m_\#(\ell)$, and hence $f^m_\#(\sigma_{i,j(i)})$ for arbitrarily large $i$, contains an $s$-legal subpath that crosses $> C$ edges of $H_s$, for any choice of $C$. By \cite[Lemma 4.2.2]{\BookOneTag}, by choosing $C$ to be sufficiently large we may conclude that $f^m_\#(\sigma_{i,j(i)})$ splits at an interior vertex. But then for $i \ge m$ it follows that $\beta_{i,j(i)} = f_\#^{i-m}(f^m_\#(\sigma_{i,j(i)}))$ splits at an interior vertex, contradicting maximality of the splitting of $\beta_i$, and thus ruling out the second option for~$\ell$. This concludes the proof if $H_s$ is non-geometric. 

Assuming now that $H_s$ is geometric, it remains to show that the third option can be avoided, and hence by the previous paragraph the desired contradiction is achieved. That is, if $H_s$ is geometric we show that some weak limit of a subsequence of $\sigma_{i,j(i)}$ contains at least one edge of $H_s$ and is not a bi-infinite iterate of the closed path $\rho$ or $\bar \rho$. This may be done by setting up an application of Lemma 1.11 of \hbox{\SubgroupsThree,} but it is just as simple to give a direct proof. Lift $\sigma_{i,j(i)}$ to the universal cover of $G$ and write it as an edge path $\ti \sigma_{i,j(i)} = \wt E_{i1} \wt E_{i2} \ldots \wt E_{iM_i} \subset \wt G$; the first and last terms are allowed to be partial edges. Let $b$ equal  twice the number of edges in $\rho_s$.  Given $m \in \{2+b,\ldots,M_i-b-1\}$, we say that $\wt E_{im}$ is \emph{well covered} if there is a periodic line $\ti \rho_{im} \subset \wt G$ that projects to $\rho$ or to $\bar \rho$ and that contains $\wt E_{i,m-b}\ldots \wt E_{im}\ \ldots \wt E_{i,m+b}$ as a subpath. Since the intersection of distinct periodic lines  cannot contain two fundamental domains of both lines, $\ti \rho_{im}$ is unique if it exists. Moreover, if both $\wt E_{im}$ and $\wt E_{i,m+1}$ are well covered then  $\ti \rho_{im}=\ti \rho_{i,m+1}$. It follows that if $\wt E_{im}$ is well covered then we can inductively move forward and backward past other well covered edges of $\tilde\sigma_{i,j(i)}$ all in the same lift of $\rho$, until either encountering an edge that is not well covered, or encountering initial and terminal subsegments of $\ti\sigma_{i,j(i)}$ of uniform length. After passing to a subsequence, one of the following is therefore satisfied:
\begin{enumerate}
 \item \label{good weak limit} There exists a sequence of integers $K_i$ such that $2+b < K_i < M_i-b-1$, and $K_i \to \infinity$, and $M_i - K_i \to \infty$, and such that $\wt E_{iK_i} \subset \wt H_s$ is not well covered.   
 \item \label{just rho}  $\sigma_{i,j(i)} = \alpha_i \rho^{p_i} \beta_i$ where the number of edges crossed by $\alpha_i$ and $\beta_i$ is bounded independently of $i$ and $\abs{p_i} \to \infty$.
\end{enumerate}  
If~\pref{good weak limit} holds then the existence of a subsequential weak limit of $\sigma_{i,j(i)}$ that crosses an edge of $H_s$ and is not a bi-infinite iterate of $\rho$ or $\bar \rho$ follows immediately. If~\pref{just rho} holds then $\beta_{i,j(i)}=f^i_\#(\sigma_{i,j(i)})$ decomposes as $\mu_i \, \rho^{q_i} \, \nu_i$ where $\abs{q_i} \to \infty$ and  the number of illegal turns of height $s$ in $\mu_i$ and $\nu_i$ is uniformly bounded. But then for any $p \ge 0$ the number of edges  of $\rho^{q_i}$  that are cancelled when  $f^p_\#(\mu_i) \rho^{q_i}f^p_\#(\nu_i)$ is tightened to $f^p_\#(\beta_{i,j(i)})$ is bounded independently of $i$ and $p$ and therefore $\beta_{i,j(i)}$ can be split at an interior vertex, which is a contradiction.
 \end{proof}

\section{Reducing Theorem A to Theorem C and the WWPD Construction Theorem}
\label{SectionBToE}  

\subsection{Weak attraction and subgroup decomposition theory}
\label{SectionPPResults}
In this section we review some terminology and notation regarding proofs and applications of lamination ping-pong. The origins of lamination ping-pong are found in the proof of the Tits alternative for $\Out(F_n)$ by Bestvina, Feighn, and Handel \BookOne. We focus particularly on results from work on subgroup classification theory \Subgroups, and from Ghosh's work on free subgroups of $\Out(F_n)$ \cite{Ghosh:WeakAttraction}.

\medskip\textbf{The group $\IAThree$.} Recall the finite index normal subgroup $\IAThree \subgroup \Out(F_n)$ defined by
$$\IAThree = \kernel\bigl(\Out(F_n) \mapsto \Aut(H_1(F_n;\Z/3) \approx GL(n,\Z/3)\bigr)
$$
This subgroup has several important invariance properties:
\begin{itemize}
\item  $\IAThree$ is torsion free (see \cite{BaumslagTaylor:Center}; and see \cite{Vogtmann:OuterSpaceSurvey}).
\item For every $\psi \in \IAThree$ and every free factor $A \subgroup F_n$, if its conjugacy class $[A]$ is $\psi$-periodic then it is fixed by $\psi$ (see \SubgroupsTwo, Theorem 3.1).
\item For every $\psi \in \IAThree$ and every $\gamma \in F_n$, if its conjugacy class $[\gamma]$ is $\psi$-periodic then it is fixed by $\psi$ (see \SubgroupsTwo, Theorem 4.1).
\item Every virtually abelian subgroup of $\IAThree$ is abelian (see \cite{HandelMosher:VirtuallyAbelian}).
\end{itemize}
As said in the introduction, we sometimes emphasize the last result by using the terminology ``(virtually) abelian'' in the context of a subgroup of $\IAThree$. For example, here is a consequence of \cite{HandelMosher:VirtuallyAbelian} which justifies that terminology for the property of virtually abelian restrictions:

\begin{corollary}
\label{CorollaryVirtuallyAbelian}
A subgroup $\Gamma \subgroup \IAThree$ has virtually abelian restrictions if and only if it has abelian restrictions.
\end{corollary}

\begin{proof}
For any free factor $A \subgroup F_n$ the inclusion $A \inject F_n$ induces an injection $H_1(A;\Z/3) \mapsto H_1(F_n;\Z/3)$ which is equivariant with respect to the natural actions $\Stab[A] \inject \Out(F_n) \act H_1(F_n;\Z/3)$ and $\Stab[A] \mapsto \Out(A) \act H_1(A;\Z/3)$. It follows that the image of $\Gamma$ under the induced homomorphism $\Gamma \mapsto \Out(A)$ is contained in $\IA_A(\Z/3)$, and so by \cite{HandelMosher:VirtuallyAbelian} that image is virtually abelian if and only if it is abelian.
\end{proof}

\smallskip\textbf{Relative full irreducibility.} An \emph{extension of free factor systems} is simply an instance $\F \sqsubset \F'$ of the usual partial ordering such that $\F \ne \F'$. Given an extension $\F \sqsubset \F'$, its \emph{co-edge number} is the minimum, over all marked graphs $G$ and subgraphs $H \subset H' \subset G$ realizing $\F \sqsubset \F'$, of the number of edges of $H' \setminus H$. If the co-edge number equals~$1$ then $\F \sqsubset \F'$ is a \emph{one-edge extension}, and this holds if and only if the free factor systems $\F$ and $\F'$ are related in one of two ways: either two components of $\F$ having ranks $i,j$ are replaced by a single component of $\F'$ of rank $i+j$ that contains the given two components of $\F$; or a single component of $\F$ having rank~$i$ is replaced by a single component of $\F'$ of rank $i+1$ that contains the given component of $\F$. Otherwise $\F \sqsubset \F'$ is a \emph{multi-edge extension}. If $\F'=\{[F_n]\}$ then we drop $\F'$ from the terminology and speak of the co-edge number of $\F$, and so $\F$ has co-edge number~$1$ if and only if $\F$ consists either of a single component of rank $n-1$ or of two components whose ranks sum to~$n$.

For any subgroup $\cH \subgroup \Out(F_n)$ and any $\cH$-invariant extension $\F \sqsubset \F'$, recall that $\cH$ is \emph{fully irreducible} relative to the extension if there does not exist a free factor system contained strictly between $\F$ and $\F'$ which is invariant under a finite index subgroup of~$\cH$. If in addition $\F'=\{[F_n]\}$ then it is dropped from the terminology and we say that $\cH$ is fully irreducible rel~$\F$. When $\cH=\<\phi\>$ is cyclic then we extend this terminology to $\phi$, and so $\phi$ is fully irreducible relative to a $\phi$-invariant extension $\F \sqsubset \F'$ if and only if no free factor system contained strictly between $\F$ and $\F'$ is $\phi$-periodic. In the context of a subgroup or element of $\IAThree$ we may drop the adverb ``fully'', because in $\IAThree$ periodic free factor systems are fixed \hbox{(\SubgroupsTwo\ Theorem 3.1).}

\smallskip\textbf{Weak attraction; nonattracting subgroup systems.}
Weak attraction theory was first developed in \BookOne\ for attracting laminations associated to top \eg\ strata of improved relative train track maps. We review here the generalization to arbitrary attracting laminations developed in Section 3 of \SubgroupsOne. We also prove Fact~\ref{FactMutualAttraction}, a ``mutual weak attraction'' result which will be applied later in a ping-pong argument. 

Consider a subgroup system~$\A$ in~$F_n$. We say that $\A$ is a \emph{vertex group system} if there exists a minimal action $F_n \act T$ on an $\reals$-tree $T$ with trivial arc stabilizers such that $\A$ is the set of conjugacy classes of nontrivial point stabilizers of the action $F_n \act T$.

Consider $\phi \in \Out(F_n)$ and a lamination pair $\Lambda^\pm_\phi \in \L^\pm(\phi)$ which is fixed by~$\phi$. Earlier we defined weak attraction of a line to the lamination $\Lambda^+_\phi$ under iteration of $\phi$. A conjugacy class is weakly attracted to $\Lambda^+_\phi$ if the periodic line representing that conjugacy class is weakly attracted. The \emph{nonattracting subgroup system} of $\Lambda^\pm_\phi$ is denoted $\A_\na(\Lambda^\pm_\phi)$, it is concretely described in terms of any \ct\ representative of a rotationless power of~$\phi$ in Definitions 1.2 of \SubgroupsThree, and it is more abstractly characterized by item~\pref{FactAnaChar} of the following compilation of results:


\begin{fact} \label{FactANA}
For any rotationless $\phi \in \Out(F_n)$ and $\Lambda^\pm_\phi \in \L^\pm(\phi)$ we have:
\begin{enumerate}
\item \label{FactANAGeom}
$\A_\na\Lambda^\pm_\phi$ is a vertex group system, and it is a free factor system if and only if the pair $\Lambda^\pm_\phi$ is nongeometric (\Subgroups\ Part I Section 3, and Part III Proposition 1.4). 
\item\label{FactAnaChar}
$\A_\na(\Lambda^\pm_\phi)$ is the unique vertex group system such that: a conjugacy class $c$ is carried by $\A_\na(\Lambda^+_\phi)$ $\iff$ $c$ is not weakly attracted to $\Lambda^+_\phi$ by iteration of $\phi$ $\iff$ $c$ is not weakly attracted to $\Lambda^-_\phi$ by iteration of $\phi^\inv$ (\SubgroupsThree\ Corollaries 1.9 and 1.10).
\item \label{ItemANATop}
If $\F$ is a proper, $\phi$-invariant free factor system, and if there exists a lamination pair $\Lambda^\pm_\phi \in \L^\pm(\phi)$ not carried by~$\F$, then: 
\begin{enumerate}
\item \label{FactAnaCoEdgeTwo}
$\F$ has co-edge number~$\ge 2$.
\item \label{FactAnaFullIrr}
If $\phi$ is fully irreducible rel~$\F$ then $\Lambda^\pm_\phi$ is the unique lamination pair in $\L^\pm(\phi)$ not carried by~$\F$, and furthermore:
\begin{enumerate}
\item $\Lambda^\pm_\phi$ fills relative to~$\F$, meaning that $\F_\supp(\Lambda^\pm_\phi;\F) = \{[F_n]\}$.
\item If the lamination pair $\Lambda^\pm_\phi$ is nongeometric then $\A_\na(\Lambda^\pm_\phi)=\F$;
\item If the pair $\Lambda^\pm_\phi$ is geometric then there exists a primitive cyclic subgroup $C \subset F_n$ such that $\A_\na(\Lambda^\pm_{\phi}) = \F \union \{[C]\}$ and $\F_\supp(\F,[C]) = \{[F_n]\}$; in particular $[C] \not\sqsubset \F$.
\end{enumerate}
\end{enumerate}
\end{enumerate}
\end{fact}

\begin{proof} Only item~\pref{ItemANATop} needs some justification. 

We prove~\pref{FactAnaCoEdgeTwo} by contradiction, so suppose that $\F$ has co-edge number~$1$. Let $f \from G \to G$ be a relative train track representative of $\phi$ in which $\F$ is represented by a filtration element $G_r$. There are two cases to consider, regarding how $G$ is obtained topologically from~$G_r$: by attaching either an arc or a ``sewing needle'' to $G_r$.

In the first case $G$ is obtained by attaching an arc $\alpha$ to $G_r$, identifying each endpoint of $\alpha$ with a vertex of~$G_r$. Let $f' \from G \to G$ be obtained from $G$ by tightening the restriction~$f \restrict \alpha$. Applying \cite[Corollary 3.2.2]{\BookOneTag}, the path $f'(\alpha)$ crosses $\alpha$ at most once, and so~$f'$ has no \eg\ stratum above~$G_r$. Since $f' \restrict G_r = f \restrict G_r$, the map $f'$ is also a relative train track representative of~$\phi$, because the definition of a relative train track map imposes no condition on non-\eg\ strata. Applying~\cite[Lemma 3.1.10]{\BookOneTag}, the highest stratum of $f'$ crossed by leaves of $\Lambda^+_\phi$ is \eg\ and so all such leaves are contained in $G_r$, implying that $\Lambda^+_\phi$ is supported by~$\F$, a contradiction.

In the second case $G$ is obtained from $G_r$ by attaching a ``sewing needle'', consisting of a vertex $v$ disjoint from $G_r$, an arc $\alpha$ with both endpoints identified to $v$, and another arc $\alpha'$ with one endpoint identified to $v$ and opposite endpoint identified to a vertex on $G_r$. Let $f_1 \from G_1 \to G_1$ be the topological representative obtained from $f \from G \to G$ by first collapsing $\alpha'$ to a point to get $G_1$ and then tightening the induced map on $\alpha$. The analysis of the first case now applies to this map $f_1$, reaching the same contradiction. 

We prove~\pref{FactAnaFullIrr}. For item~(i), the free factor system $\F_\supp(\Lambda^\pm_\phi;\F)$ equals $\{[F_n]\}$ because it contains $\F$ and is $\phi$-invariant, and $\phi$ is fully irreducible rel~$\F$. It also follows that $H_s$ is the top stratum. For uniqueness of $\Lambda^\pm_\phi$, if $\phi$ had a different lamination pair  not supported by $\F$ then there would exist a $\phi$-invariant free factor system strictly between $\F$ and $\{[F_n]\}$ (by Fact 1.55 of \SubgroupsOne), contradicting that $\phi$ is fully irreducible rel~$\F$. The further conclusions of~\pref{FactAnaFullIrr} then follow from Definition 1.2, ``Remark: The case of a top stratum'', in \SubgroupsThree.
\end{proof}

The conclusions of the following ``mutual weak attraction'' fact match key hypotheses of lamination ping-pong results, for example Proposition 1.3 of \SubgroupsFour, as well as Ghosh's theorem \cite[Theorem 7.3]{\GhoshWeakTag} discussed below, enabling us to apply these results below in Propositions~\ref{PropGhosh} and~\ref{multiple filling}.


\begin{fact}\label{FactMutualAttraction}
Given $\phi$, $\psi \in \Out(F_n)$ and lamination pairs $\Lambda^\pm_\phi \in \L^\pm(\phi)$ and $\Lambda^\pm_\psi \in \L^\pm(\psi)$, suppose that the following hypotheses hold:
\begin{enumerate}
\item\label{ItemPhiPsiDiffLams}
$\{\Lambda^-_\phi,\Lambda^+_\phi\} \intersect \{\Lambda^-_\psi,\Lambda^+_\psi\} = \emptyset$;
\item\label{ItemPhiPsiANA}
No generic leaves of $\Lambda^-_\phi$ or $\Lambda^+_\phi$ are carried by $\A_\na(\Lambda^\pm_\psi)$, and similarly with $\phi,\psi$ switched;
\end{enumerate} 
It follows that generic leaves of $\Lambda^+_\phi$ and of $\Lambda^-_\phi$ are weakly attracted to $\Lambda^+_\psi$ by iteration of $\psi$ and to $\Lambda^-_\psi$ by iteration of $\psi^\inv$, and similarly with $\phi,\psi$ switched.
\end{fact}

\begin{proof}
By symmetry we need only check that a generic leaf $\ell^+_\phi$ of $\Lambda^+_\phi$ is weakly attracted to $\Lambda^+_\psi$ by iteration of $\psi$, which we do by applying Theorem H of \SubgroupsThree. 

Consider a generic leaf $\ell^-_\psi$ of $\Lambda^-_\psi$. If $\ell^-_\psi$ is a leaf of $\Lambda^+_\phi$ then by combining \cite[Lemma 3.1.15]{\BookOneTag} with birecurrence of $\ell^-_\psi$ it follows that one of two cases holds. In the first case, $\ell^-_\psi$ is a generic leaf of $\Lambda^+_\phi$ and so $\Lambda^-_\psi = \Lambda^+_\phi$ which contradicts~\pref{ItemPhiPsiDiffLams}. In the second case, choosing any \ct\ representative $f \from G \to G$ of a rotationless power of $\phi$ with core filtration element $G_r$ and \eg\ stratum $H_r$ corresponding to $\Lambda^\pm_\phi$, the realization in $G$ of $\ell^-_\psi$ is contained in $G_{r-1}$ and so $\ell^-_\psi$ is carried by $\A_\na(\Lambda^\pm_\phi)$, which contradicts~\pref{ItemPhiPsiANA}. 

We conclude that $\ell^-_\psi$ is not a leaf of $\Lambda^+_\phi$ and hence there exists a weak neighborhood $V^-_\psi$ of $\ell^-_\psi$ such that $\ell^+_\phi \not\in V^-_\psi$. Also, from~\pref{ItemPhiPsiANA} we conclude that $\ell^+_\phi$ is not carried by $\A_\na(\Lambda^\pm_\psi)$. Applying Theorem H~(2) of \SubgroupsThree\ it follows that for every weak neighborhood $V^+_\psi$ of $\ell^+_\psi$ there exists $m$ such that $\psi^m(\ell^+_\phi) \in V^+_\psi$, that is, $\ell^+_\phi$ is weakly attracted to $\ell^+_\psi$ by iteration of $\psi$.
\end{proof}

\medskip\textbf{A lamination ping-pong result.}  The next result is a pared down version of \hbox{\SubgroupsFour\ Proposition 1.7,} omitting several conclusions of that proposition labelled $(2^\pm)$, $(3^\pm)$,  $(4^\pm)$ that we do not need here. 

\begin{proposition}[Proposition 1.7 of \SubgroupsFour]
\label{PropPingPongVer}
Given a free factor system $\F$, rotationless $\phi,\psi \in \Out(F_n)$ both leaving $\F$ invariant, and lamination pairs $\Lambda^\pm_\phi \in \L^\pm(\phi)$, $\Lambda^\pm_\psi \in \L^\pm(\psi)$ each having a generic leaf fixed by $\phi^{\pm 1}, \psi^{\pm 1}$ (resp.) with fixed orientation, assume that the following hypotheses hold:

\smallskip

(a) $\F \sqsubset \A_\na\Lambda^\pm_\phi$ and $\F \sqsubset \A_\na\Lambda^\pm_\psi$.

(i) Generic leaves of $\Lambda^+_\psi$ are weakly attracted to $\Lambda^+_\phi$ under iteration by $\phi$. 

(ii) Generic leaves of $\Lambda^-_\psi$ are weakly attracted to $\Lambda^-_\phi$ under iteration by $\phi^{-1}$. 

(iii) Generic leaves of $\Lambda^+_\phi$ are weakly attracted to $\Lambda^+_\psi$ under iteration by $\psi$. 

(iv) Generic leaves of $\Lambda^-_\phi$ are weakly attracted to $\Lambda^-_\psi$ under iteration by $\psi^{-1}$. 

\smallskip\noindent
Under these hypotheses, 
there exists an integer $M$, such that for any $m,n \ge M$ the outer automorphism $\xi = \psi^m \phi^n$ has an invariant attracting lamination $\Lambda^+_\xi \in \L(\xi)$ and an invariant repelling lamination $\Lambda^-_\xi \in \L(\xi^\inv)$ such that each is nongeometric if $\Lambda^\pm_\phi$ and $\Lambda^\pm_\psi$ are nongeometric, and the following hold: 
\begin{itemize}
\item[(1)] $\F$ is carried by both $\A_\na\Lambda^-_\xi$ and $\A_\na\Lambda^+_\xi$, and so neither $\Lambda^+_\xi$ nor $\Lambda^-_\xi$ is carried by $\F$. Also, both $\A_\na\Lambda^-_\xi$ and $\A_\na\Lambda^+_\xi$ are carried by $\A_\na\Lambda^\pm_\phi$ and by $\A_\na\Lambda^\pm_\psi$.
\item[(5+)] For any weak neighborhood $U^+ \subset \B$ of a generic leaf of $\Lambda^+_\psi$ there exists an integer $M^+$ such that if $m,n \ge M^+$ then a generic leaf of $\Lambda^+_\xi$ is in $U^+$.
\item[(5-)] For any weak neighborhood $U^- \subset \B$ of a generic leaf of $\Lambda^-_\phi$ there exists an integer $M^-$ such that if $m,n \ge M^-$ then a generic leaf of $\Lambda^-_\xi$ is in $U^-$.
\item[(6)] $\Lambda^\pm_\xi$ is a dual lamination pair of $\xi$ under either of the following conditions:

--- Both pairs $\Lambda^\pm_\phi$ and $\Lambda^\pm_\psi$ are nongeometric; or

--- Both laminations $\Lambda^+_\xi$ and $\Lambda^-_\xi$ are geometric. \qed
\end{itemize}
\end{proposition}

\medskip\textbf{Subgroup decomposition theory.} We review Theorem~I of \SubgroupsFour\ which is the main result of subgroup decomposition theory. This theorem and its proof will be applied in Proposition~\ref{PropInfIndexStab} to follow. 

Consider $\phi,\psi \in \Out(F_n)$ and lamination pairs $\Lambda^\pm_\phi \in \L^\pm(\phi)$, $\Lambda^\pm_\phi \in \L^\pm(\psi)$. The Independence Theorem \cite[Theorem 1.2]{\FSLoxTag} says that the sets $\{\Lambda^+_\phi,\Lambda^-_\phi\}$ and $\{\Lambda^+_\psi,\Lambda^-_\psi\}$ are either equal or disjoint; we refer to this by saying that the \emph{duality} relation amongst attracting laminations is well-defined. Furthermore $\Stab(\Lambda^\pm_\phi) =  \Stab(\Lambda^+_\phi) = \Stab(\Lambda^-_\phi)$, in other words if $\theta \in \Out(F_n)$ stabilizes either of $\Lambda^+_\phi$ or $\Lambda^-_\phi$ then $\theta$ stabilizes both $\Lambda^+_\phi$ and $\Lambda^-_\phi$ \cite[Corollary 1.3]{\FSLoxTag}.

Given $\phi \in \Out(F_n)$ and a $\phi$-invariant free factor system~$\F$ define 
$$\L(\phi;\F) = \{\Lambda \in \L(\phi) \suchthat \text{$\Lambda$ is not carried by~$\F$}\}
$$ 
and let $\L^\pm(\phi;\F) \subset \L^\pm(\phi)$ be similarly defined.  Given $\Gamma \subgroup \Out(F_n)$ and an $\Gamma$-invariant free factor system $\F$, let $\L(\Gamma;\F) = \union_{\phi \in \Gamma} \L(\phi;\F)$. Since free factor support is well-defined for a dual lamination pair, and since duality is well-defined for attracting laminations, the set $\L(\Gamma;\F)$ decomposes as a disjoint union of dual lamination pairs. 

Following Definition 1.2 of \SubgroupsFour, we say that $\Gamma$ is \emph{geometric above $\F$}  if every element of $\L(\Gamma;\F)$ is geometric.


\begin{theorem}[\SubgroupsFour, Theorem I]  \label{Theorem_I} 
Let   $\h \subgroup \IA_n(\Z/3)$ be a subgroup and  $\F \sqsubset \{[F_n]\}$   an $\h$-invariant multi-edge extension of free factor systems    such that $\h$ is irreducible relative to $\F$.  If  either $\h$ is finitely generated or      $\L(\h,\F) \ne \emptyset$   then there exists $\phi \in \h$ which is fully irreducible relative to $\F$.   Moreover,  for any  $\theta \in \h$  and     $\Lambda^-_\theta \in \L(\theta^{-1}, \F)$, if either   $\h$ is geometric above $\F$ or $\Lambda^-_\theta$ is non-geometric then for any weak neighborhood $U \subset \B$ of a generic leaf of  $\Lambda^-_\theta$ we may choose $\phi$ so that generic leaves of $\Lambda^-_\phi$, the unique element of $\L(\phi^{-1}, \F))$,   are contained in $U$. 
\end{theorem}
  
\medskip\textbf{Laminations with infinite orbit.} The following result will be used in Section~\ref{SectionReducingBToE}. It is an application of Theorem I of \SubgroupsFour, and of the lamination ping-pong tools underlying its proof.

\begin{proposition} \label{PropInfIndexStab} 
Consider a  subgroup $\Gamma \subgroup \IA_n(\Z/3)$ and a $\Gamma$-invariant free factor system $\F$ such that $\Gamma$ is irreducible relative to $\F$. If $\L(\Gamma;\F)$ contains more than one dual lamination pair then for each $\Lambda \in \L(\Gamma;\F)$, the stabilizer of $\Lambda$ in $\Gamma$ has infinite index, equivalently the $\Gamma$-orbit of $\Lambda$ is infinite. 
\end{proposition}

\break
 
\begin{proof} We begin with:

\begin{description}
\item[Claim 1:] For any $\alpha,\beta \in \Gamma$ and lamination pairs $\Lambda^\pm_\alpha \in \L^\pm(\alpha;\F)$, $\Lambda^\pm_\beta \in \L^\pm(\beta;\F)$, if $\alpha$ is fully irreducible rel~$\F$, and if $\{\Lambda^-_\alpha,\Lambda^+_\alpha\} \ne \{\Lambda^-_\beta,\Lambda^-_\beta\}$, then neither of $\Lambda^\pm_\beta$ is a sublamination of either of $\Lambda^\pm_\alpha$, equivalently if $\gamma$ is a generic leaf of one of $\Lambda^-_\beta$ or $\Lambda^+_\beta$ then $\gamma$ is a leaf of neither $\Lambda^-_\alpha$ nor $\Lambda^+_\alpha$.
\end{description}
To prove this, note first that the sets $\{\Lambda^-_\alpha,\Lambda^+_\alpha\}$ and $\{\Lambda^-_\beta,\Lambda^+_\beta\}$ are disjoint (by the Independence Theorem \cite[Theorem 1.2]{\FSLoxTag}), so $\gamma$ is a generic leaf of neither $\Lambda^+_\beta$ nor~$\Lambda^-_\beta$. Also, by applying \cite[Lemma 3.1.15]{\BookOneTag} to a relative train track representative of $\alpha$ in which $\F$ is realized by a filtration element, it follows that every birecurrent, nongeneric leaf of $\Lambda^\pm_\alpha$ is carried by $\F$ and so no such leaf equals $\gamma$, completing the proof of Claim~1.

\medskip

Since $\L(\Gamma;\F) \ne \emptyset$ it follows by Fact~\ref{FactANA}~\pref{ItemANATop} that $\F \sqsubset \{[F_n]\}$ is a multi-edge extension. Applying Theorem~\ref{Theorem_I}, choose $\eta \in \Gamma$ which is fully irreducible rel~$\F$. The set $\L^\pm(\eta;\F)$ consists of a unique lamination pair~$\Lambda^\pm_\eta$. 

Let $\Theta_\L$   be the set of pairs $(\theta, \Lambda^\pm_\theta)$ where $\theta \in \Gamma$, $\Lambda^\pm_\theta \in \L(\theta; \F)$ and the pairs $\{\Lambda^-_\theta,\Lambda^+_\theta\}$ and $ \{\Lambda^-_\eta,\Lambda^+_\eta\}$  are unequal and hence disjoint.   By hypothesis, $\Theta_\L \ne \emptyset$.   For each $(\theta, \Lambda^\pm_\theta) \in \Theta_\L$,    it follows by Claim 1 that  generic leaves $\ell^\pm_\theta$ of $\Lambda^\pm_\theta$ are neither leaves of $\Lambda^-_\eta$ nor leaves of $\Lambda^+_\eta$. Since $\Lambda^-_\eta$ and $\Lambda^+_\eta$ are weakly closed, there are weak neighborhoods $U^\pm_\theta \subset \B$ of $\ell^\pm_\theta$ disjoint from $\Lambda^-_\eta$ and $\Lambda^+_\eta$. 

  If $\Lambda \in \L(\Gamma; \F)$ is neither $\Lambda^-_\eta$ nor $\Lambda^+_\eta$, choose $(\theta, \Lambda^\pm_\theta) \in \Theta_\L$ such that $\Lambda = \Lambda^+_\theta$.  From Claim~1,  $\Lambda$  is not a  sublamination of either of~$\Lambda^\pm_\eta$. Applying \cite[Proposition~6.0.8]{\BookOneTag},  $\Lambda$ is weakly attracted to $\Lambda^+_\eta$ under iteration of $\xi$. The $\eta$-orbit of $\Lambda$ is therefore infinite, and so its $\Gamma$-orbit is infinite.  This completes the proof if $\Lambda \not \in  \{\Lambda^-_\eta,\Lambda^+_\eta\}$.     The same argument applies with $\eta$ replaced by any other element of $\Gamma$ that is irreducible rel $\F$.  We are therefore reduced to

\begin{description}
\item[Claim 2:]  There exists $(\mu, \Lambda_\mu^\pm) \in \Theta_\L$  such that $\mu$ is irreducible rel~$\F$. 
\end{description}

As a first case, suppose that there exists   $(\theta, \Lambda^\pm_\theta)  \in \Theta_\L$ with non-geometric $\Lambda^\pm_\theta$.   By Theorem~\ref{Theorem_I} there exists $\mu \in \Gamma$ which is fully irreducible rel~$\F$ and such that $\Lambda^-_\mu$ is contained in $U_\theta^-$; in particular,  $\Lambda^-_\mu \not \in \{\Lambda^-_\eta,\Lambda^+_\eta\}$ and we have verified claim 2.  We are now reduced to the case that each  $\Lambda^\pm_\theta$ is geometric.  If $\Lambda^\pm_\eta$ is geometric then the same  argument  applies without change.  

We may therefore assume that $\Lambda^\pm_\eta$ is non-geometric and that each $\Lambda^\pm_\theta$ is geometric. 

Choose $(\mu, \Lambda^\pm_\mu) \in \Theta_\L$ so that $\A_\na\Lambda^\pm_\mu$ is minimal with respect to $\sqsubset$. In other words, if $(\theta,  \Lambda^\pm_\theta) \in \Theta_\L$ and $\A_\na\Lambda^\pm_{\theta} \sqsubset \A_\na\Lambda^\pm_\mu$  then $\A_\na\Lambda^\pm_{\theta} = \A_\na\Lambda^\pm_\mu$. Such a $(\mu, \Lambda^\pm_\mu) $ exists because there is a bound to the length of a chain of proper inclusions of non-attracting subgroup systems (\SubgroupsOne\  Proposition 3.2). We will complete the proof by showing that $\mu$ is irreducible rel $\F$. Since $\Lambda^\pm_\mu$ is geometric, it suffices by Lemma 2.3 (3) of \SubgroupsFour\ to show that $\Stab_{\Gamma}(A_\na\Lambda^\pm_\mu)$ has finite index in $\Gamma$.

Assuming that $\Stab_{\Gamma}(A_\na\Lambda^\pm_\mu)$ has infinite index we will arrive at a contradiction by applying  Proposition~\ref{PropPingPongVer}.  Passing to a power of $\mu$ we may assume that $\Lambda_\mu^\pm$   have generic leaves $\ell^\pm_\mu$ that are fixed by $\mu$  with fixed orientations, which is one of the hypotheses of Proposition~\ref{PropPingPongVer}. Applying Lemma 2.1 of \SubgroupsFour, there exists $\zeta \in \Gamma$ such that the following hold. 
None of the lines $\zeta(\ell^+_\mu)$, $\zeta(\ell^-_\mu)$, $\zeta^\inv(\ell^+_\mu)$, $\zeta^\inv(\ell^-_\mu)$ is carried by $\A_\na\Lambda^\pm_\mu$ (\SubgroupsFour\ Lemma 2.1 (1)). Applying Theorem H of \SubgroupsThree, it follows that each of these lines is weakly attracted to $\Lambda^\pm_\mu$ under iteration by~$\mu$. Letting $\beta = \zeta \mu \zeta^\inv$ with geometric lamination pair $\Lambda^\pm_\beta = \zeta(\Lambda^\pm_\mu) \in \L^\pm(\beta;\F)$, it follows that $\mu$ and $\beta$ with their lamination pairs $\Lambda^\pm_\mu$ and $\Lambda^\pm_\beta$ satisfy hypotheses (i)--(iv) of Proposition~\ref{PropPingPongVer}. Also, $\A_\na\Lambda^\pm_\beta = \zeta(\A_\na\Lambda^\pm_\mu) \ne \A_\na\Lambda^\pm_\mu$ (\SubgroupsFour\ Lemma 2.1~(4)).  Hypothesis (a) of Proposition~\ref{PropPingPongVer} is that $\F \sqsubset \A_\na\Lambda^\pm_\mu$ and $\F \sqsubset \A_\na\Lambda^\pm_\beta$, which holds because $\Lambda^\pm_\mu,\Lambda^\pm_\beta$ are not supported by~$\F$.  We may now apply Proposition~\ref{PropPingPongVer}  with $U^+  =  U^+_\beta$ and $U^- = U^-_\mu$.

 
 Applying the conclusions of Proposition~\ref{PropPingPongVer}, for some $m,n \ge 1$ we have: an outer automorphism $\nu = \beta^m \mu^n$; a $\nu$-invariant attracting lamination $\Lambda^+_{\nu} \in \L(\nu)$; and a $\nu$-invariant repelling lamination $\Lambda^-_\nu \in \L(\nu)$. By Conclusion (1), $\Lambda^-_\nu, \Lambda^+_\nu \in \L(\Gamma; \F)$.  By Conclusion~(5${}^\pm$), generic leaves of $\Lambda^-_\nu$ are contained in~$U^-$, and generic leaves of $\Lambda^+_\nu$ are contained in $U^+$. Since $U^-$ and $U^+$ are disjoint from $\Lambda^\pm_\eta$, neither $\Lambda^-_\nu$ nor $\Lambda^+_\nu$ is in $\{\Lambda^-_\eta,\Lambda^+_\eta\}$ and so  $\Lambda^-_\nu$ and $\Lambda^+_\nu$ are geometric. By Conclusion~(6), $\Lambda^\pm_\nu$ is a dual lamination pair for $\nu$, and so $\A_\na\Lambda^\pm_\nu$ is a well-defined vertex group system. By Conclusion (1), 
 $\A_\na\Lambda^\pm_\nu \sqsubset \A_\na\Lambda^\pm_\mu$ and $\A_\na\Lambda^\pm_\nu \sqsubset \A_\na\Lambda^\pm_\beta$ and $(\nu, \Lambda^\pm_\nu) \in \Theta_\L$. The containment $\A_\na\Lambda^\pm_\mu \sqsupset \A_\na\Lambda^\pm_\nu $ is therefore proper, in contradiction to our choice of~$\mu$.   
\end{proof}

\subparagraph{Free subgroups.} We will need the following Proposition~\ref{PropGhosh} which, inside certain subgroups of $\Out(F_n)$, produces useful free subgroups. In the nongeometric case this is a consequence of a theorem of Ghosh \GhoshWeak. The geometric case combines tools from our subgroup decomposition theory in \SubgroupsFour\ with a result of Farb and Mosher \cite{FarbMosher:quasiconvex} that produces useful free subgroups of mapping class groups.


\begin{proposition}\label{PropGhosh}
Given $\phi,\psi \in \Out(F_n)$ and a proper free factor system $\F$ that is preserved by $\phi$ and $\psi$, and given lamination pairs $\Lambda^\pm_\phi \in \L^\pm(\phi)$ and $\Lambda^\pm_\psi \in \L^\pm(\phi)$ that fill $F_n$, suppose that the following hold:
\begin{enumerate}
\item\label{ItemGhoshFullyIrr}
$\phi$, $\psi$ are both fully irreducible rel~$\F$.
\item\label{ItemGhoshDiffLams}
$\{\Lambda^-_\phi,\Lambda^+_\phi\} \intersect \{\Lambda^-_\psi,\Lambda^+_\psi\} = \emptyset$;
\item\label{ItemGhoshGeomAlt}
Either both of the lamination pairs $\Lambda^\pm_\phi$, $\Lambda^\pm_\psi$ are nongeometric, or the group $\<\phi,\psi\>$ is geometric above~$\F$.
\end{enumerate}
Then there exists $M \ge 1$ such that for any integers $m,n \ge M$ the outer automorphisms $\phi^m$ and $\psi^n$ freely generate a rank two free subgroup $\<\phi^m,\phi^n\> \subgroup \Out(F_n)$ such that any nontrivial $\xi \in \<\phi^m,\psi^n\>$ is fully irreducible rel~$\F$ and has a lamination pair $\Lambda^\pm_\xi$ that fills rel~$\F$, and if both of $\Lambda^\pm_\phi$, $\Lambda^\pm_\psi$ are nongeometric then $\Lambda^\pm_\xi$ is nongeometric.
\end{proposition}

\smallskip\textbf{Remark.} In a more restricted context contained in the proof of Proposition~\ref{multiple filling}, we shall prove a stronger conclusion saying that, after further increasing $M$, each pair $\Lambda^\pm_\xi$ fills~$F_n$ in the absolute sense.

\begin{proof} The result breaks naturally into two cases.

\smallskip\textbf{Case 1: $\Lambda^\pm_\phi$, $\Lambda^\pm_\psi$ are both nongeometric.} The conclusion in this case exactly matches the conclusion of \cite[Theorem 7.3]{\GhoshWeakTag}. Our work here is therefore just to verify the hypothesis of that theorem, which is that $(\phi,\Lambda^\pm_\phi)$ and $(\psi,\Lambda^\pm_\psi)$ are \emph{independent rel~$\F$} \cite[Definition~7.2]{\GhoshWeakTag}. We verify each of the six clauses of independence rel~$\F$.

\begin{description}
\item[Independence rel $\F$, clause (1):] Neither of $\Lambda^\pm_\phi,\Lambda^\pm_\psi$ is carried by~$\F$.
\item[Independence rel $\F$,  clause (2):] $\{\Lambda^\pm_\phi\} \union \{\Lambda^\pm_\psi\}$ fill rel~$\F$.
\end{description}
These hold because each pair $\Lambda^\pm_\phi$ and $\Lambda^\pm_\psi$ individually fills. 

\begin{description}
\item[Independence rel $\F$, clauses (3,4):] Generic leaves of $\Lambda^\pm_\phi$ are weakly attracted to $\Lambda^-_\psi$ by iteration of $\psi^\inv$ and to $\Lambda^+_\psi$ by iteration of $\psi$, and similarly with $\phi,\psi$ switched. 
\end{description}
This follows from Fact~\ref{FactMutualAttraction} after checking its hypotheses. Hypothesis~\pref{ItemPhiPsiDiffLams} of Fact~\ref{FactMutualAttraction} is identical to Proposition~\ref{PropGhosh}~\pref{ItemGhoshDiffLams}. 
Hypothesis~\pref{ItemPhiPsiANA} of Fact~\ref{FactMutualAttraction} follows from Fact~\ref{FactANA}~\pref{ItemANATop} using that a filling generic leaf $\ell$ of an attracting lamination cannot be carried by a component of a proper free factor system, nor by the conjugacy class of a cyclic subgroup because $\ell$ is nonperiodic.

\begin{description}
\item[Independence rel $\F$, clause (5):]  The free factor systems $\A_\na\Lambda^\pm_\phi=\A_\na\Lambda^\pm_\psi$ are mutually malnormal rel~$\F$. 
\end{description}
The meaning of mutual malnormality of two free factor systems $\A_1$ and $\A_2$ rel~$\F$ is that for each subgroup $C \subgroup F_n$, if $[C] \sqsubset \A_1$ and $[C] \sqsubset \A_2$ then $[C] \sqsubset \F$. In the current situation is is obvious from the fact that $\A_\na\Lambda^\pm_\phi = \A_\na\Lambda^\pm_\psi = \F$ (Fact~\ref{FactANA}~\pref{ItemANATop}). 

\begin{description}
\item[Independence rel $\F$, clause (6):] Both lamination pairs $\Lambda^\pm_\phi$, $\Lambda^\pm_\psi$ are nongeometric.
\end{description}
This holds by assumption of Case~1, and so Case~1 is complete.

\medskip
\textbf{Case 2: $\Lambda^\pm_\phi$ and $\Lambda^\pm_\psi$ are not both nongeometric.} Denote $\Gamma = \<\phi,\psi\> \subgroup \Out(F_n)$. Combining hypothesis~\pref{ItemGhoshGeomAlt} of the proposition with the hypothesis of Case~2 it follows that $\Gamma$ is geometric above~$\F$. Also, the group $\Gamma$ is fully irreducible rel~$\F$, because it contains an element which is fully irreducible rel~$\F$, namely~$\phi$. The hypotheses of Theorem~J of \SubgroupsFour\ for the subgroup $\Gamma$ are therefore satisfied. 

From the conclusions of Theorem~J of \SubgroupsFour, we obtain a compact surface $S$ with nonempty boundary and an injection $\pi_1 S \inject F_n$ whose image is its own normalizer, such that $\Gamma$ stabilizes the subgroup system $[\pi_1 S]$. It follows that there is a well-defined homomorphism $\Gamma \to \Out(\pi_1 S)$ which is defined, for each $\xi \in \Gamma$, by choosing an automorphism $\Xi \in \Aut(F_n)$ that represents $\xi$ and preserves $\pi_1 S$, restricting to $\Xi \restrict \pi_1 S \in \Aut(\pi_1 S)$, and passing to the quotient in $\Out(\pi_1 S)$ (see e.g.\ \SubgroupsOne\ Fact 1.4). Furthermore the conclusions of Theorem~J also say that the image of this homomorphism is contained in the natural $\MCG(S)$ subgroup of $\Out(\pi_1 S)$, thereby giving a homomorphism $dj^\# \from \Gamma \to \MCG(S)$. Also, $dj^\#(\xi)$ is pseudo-Anosov if and only if $\xi$ is fully irreducible rel~$\F$. Also, the induced map $dj_\B \from \B(\pi_1(S)) \to \B(F_n)=\B$ induces a bijection between the following two sets: the set of all geodesic laminations on $S$ which are unstable laminations of pseudo-Anosov elements of $dj^\#(\Gamma)$ (here we pick a hyperbolic structure on $S$ with totally geodesic boundary); and the set of all attracting laminations not supported by~$\F$ of elements of $\Gamma$ that are fully irreducible rel~$\F$.

\newcommand\un{\text{un}}
\newcommand\st{\text{st}}

Consider $dj^\#(\phi),dj^\#(\psi) \in dj^\#(\Gamma) \subgroup \MCG(S)$. Their lamination pairs $\Lambda^\pm_\phi$, $\Lambda^\pm_\psi$ form four distinct closed subsets of $\B$, and therefore the unstable/stable geodesic laminations pairs $\Lambda^\un_\phi$, $\Lambda^\st_\phi$, $\Lambda^\un_\psi$, $\Lambda^\st_\psi$ form four distinct laminations on $S$. The hypotheses of \cite[Theorem 1.4]{FarbMosher:quasiconvex} are therefore satisfied, the conclusions of which give the existence of $M \ge 1$ such that for all $m,n \ge M$ the mapping classes $dj^\#(\phi^m)$, $dj^\#(\psi^n)$ freely generate a rank~$2$ subgroup of $dj^\#(\Gamma)$ such that any nontrivial element is pseudo-Anosov. It follows that $\phi^m,\psi^n$ freely generate a rank~$2$ subgroup of $\Gamma$ such that any nontrivial element $\xi$ is fully irreducible rel~$\F$, and has a unique attracting/repelling lamination pair $\Lambda^\pm_\xi$ not supported by~$\F$ which maps via $dj_\B$ to the unique unstable/stable lamination pair $\Lambda^\un_\xi$, $\Lambda^\st_\xi$ of $dj^\#(\xi)$ in $S$. Since $\xi$ is fully irreducible rel~$\F$, this pair $\Lambda^\pm_\xi$ fills rel~$\F$. Case~2 is now complete.
\end{proof}

\subsection{Proof that the WWPD Construction Theorem implies Theorem B.}
\label{SectionReducingBToE}

Throughout this section we fix the following notation, taken from the hypotheses of Theorem~B:
\begin{enumerate}
\itemH\label{ItemInfLamFulAbe}
 \, $\Gamma \subgroup \IAThree$ is an infinite lamination group with (virtually) abelian restrictions.
\itemH \, $\A$ is a maximal, $\Gamma$-invariant, proper free factor system of $F_n$.
\end{enumerate}

\break

In addition we know that
\begin{enumeratecontinue}
\itemH \label{ItemLowerAbelian}
For each component $[A]$ of $\A$,
\begin{enumerate}
\item\label{ItemAIsFixed} $[A]$ is fixed by~$\Gamma$.
\item\label{ItemGroupIsAbelian} The group $\Gamma_A = \image(\Gamma \mapsto \Out(A))$ is abelian.
\end{enumerate}
\end{enumeratecontinue}
Item~\pref{ItemAIsFixed} uses Theorem 3.1 of \SubgroupsTwo, which says that if $\eta \in \IAThree$ then every free factor conjugacy class that is $\eta$-periodic is fixed by~$\eta$; we use this fact without further reference in what follows. Item~\pref{ItemGroupIsAbelian} is an application of Corollary~\ref{CorollaryVirtuallyAbelian} combined with the fact that $\Gamma \subgroup \IAThree$.

We will need the following minor extension of \cite[Lemma~4.4]{FeighnHandel:abelian}.
  
\begin{lemma} \label{FinitelyManyLams} Each virtually abelian subgroup $A$ of $\Out(F_n)$ is a finite lamination subgroup. Furthemore if $A \subgroup \IAThree$ then each element of $\L(A)$ is $A$-invariant.
\end{lemma}
  
\begin{proof} By \cite[Corollary 3.14]{FeighnHandel:abelian} there is a finite index abelian subgroup $A'< A$ that is generated by rotationless elements.  Applying \cite[Lemma~4.4]{FeighnHandel:abelian} it follows that $\L(A')$ is a finite collection of $A'$-invariant laminations. Since each $\phi \in A$ has a power $\phi^k \in A'$, $k \ne 0$, and since $\L(\phi)=\L(\phi^k)$, it follows that $\L(A') = \L(A)$ and so $\L(A)$ is~finite. 

Suppose now that $A \subgroup \IAThree$, that $\Lambda \in \L(A)$, and that $\psi \in A$; we prove that $\psi(\Lambda)=\Lambda$. Choose $\phi \in A$ so that $\Lambda \in \L(\phi) = \{\Lambda_1,\ldots, \Lambda_m\}$. Since $A$ is virtually abelian there exists $k \ge 1$ such that $\psi^k$ commutes with $\phi^k$. It follows that $\psi^k$ permutes~$\L(\phi)$. The free factor supports of the $\Lambda_i$'s are distinct \cite[Lemma 3.2.4]{\BookOneTag} and are permuted by $\psi^k$. Since $\psi \in \IAThree$, it preserves the free factor support of each $\Lambda_i$ and so $\psi$ preserves each $\Lambda_i$.
\end{proof}

Returning now to the context of the subgroup $\Gamma$ in property $\pref{ItemInfLamFulAbe}_\Gamma$ above, by Lemma~\ref{FinitelyManyLams} it follows that $\Gamma$ is not virtually abelian (and hence that hypothesis is not needed in the statement of Theorem~B, as it was in Theorem~C).

Applying Lemma~\ref{FinitelyManyLams} together with \prefH{ItemLowerAbelian} we have:
\begin{enumeratecontinue}
\itemH \label{item:restriction is abelian}
For each proper free factor $A \subgroup F_n$ such that $\Gamma \subgroup \Stab[A]$, the image of the homomorphism $\Gamma \inject \Stab[A] \mapsto \Out(A)$ is a finite lamination group. In particular, for each component $[A]$ of $\A$ the subgroup $\Gamma_A \subgroup \Out(A)$ is a finite lamination group.
\end{enumeratecontinue}
Given $\xi \in \Gamma$ consider the set $\L(\xi;\A)$ consisting of those laminations in $\L(\xi)$ that are not carried by $\A$. Let $\L(\Gamma;\A) = \union_{\xi \in \Gamma} \L(\xi;\A)$, which is an infinite set, because $\L(\Gamma)$ is infinite by hypothesis, but by \prefH{item:restriction is abelian} only finitely many elements of $\L(\Gamma)$ are carried by $\A$. We have:
\begin{enumeratecontinue}
\itemH
\label{item:finite lam} 
\, $\A \sqsubset \{F_n\}$ is a multi-edge extension (by Fact~\ref{FactANA}~\pref{FactAnaCoEdgeTwo}), and every element of $\L(\Gamma;\A)$ has infinite orbit under $\Gamma$ (by Proposition~\ref{PropInfIndexStab}).
\end{enumeratecontinue}

For each $\Lambda \in \L(\Gamma;\A)$ consider the free factor system $\F_\supp(\Gamma \cdot \Lambda)$, that is, the smallest free factor system carrying every lamination in the $\Gamma$-orbit of~$\Lambda$. Clearly $\F_\supp(\Gamma \cdot \Lambda)$ is $\Gamma$-invariant. By \prefH{ItemInfLamFulAbe} we have $\Gamma \subgroup \IA_n(\Z/3)$, and so each component of $\F_\supp(\Gamma \cdot \Lambda)$ is $\Gamma$ invariant. By item \prefH{item:finite lam}, one of the finitely many components of $\F_\supp(\Gamma \cdot \Lambda)$ supports infinitely many elements of $\L(\Gamma;\A)$ and so the restriction of $\Gamma$ to that component is not a finite lamination group; by \pref{item:restriction is abelian}$_{\Gamma}$ that component must be $\{[F_n]\}$. This shows:
\begin{enumeratecontinue} 
\itemH 
\label{item:fills} 
The $\Gamma$-orbit of each element of $\L(\Gamma;\A)$ fills~$F_n$.
\end{enumeratecontinue}
Using this we next verify most of conclusion~\pref{ItemWWPDExistsB} of Theorem~B:
\begin{enumeratecontinue}
\itemH \label{filling laminations} 
There exists $\eta  \in [\Gamma,\Gamma]$ such that $\eta$ is fully irreducible rel~$\A$ and such that $\eta$ acts loxodromically on $\FS(F_n)$.
\end{enumeratecontinue}
We prove \prefH{filling laminations} in two cases depending on whether $\Gamma$ is geometric above~$\A$ (see just before Theorem~\ref{Theorem_I} to recall the definition).

\smallskip

\textbf{Case 1: $\Gamma$ is not geometric above~$\A$.} Applying Proposition 2.2 (1) of \SubgroupsFour, we obtain $\eta \in \Gamma$ and a non-geometric lamination pair $\Lambda^\pm_\eta \in \L(\eta;\A)$ such that $\A_\na(\Lambda^\pm_{\eta}) = \A$. We may assume in addition that $\eta$ is chosen so that $\F_\supp(\Lambda^\pm_\eta)$ is maximal with respect to~$\sqsubset$. 

We claim that the free factor system $\F_\supp(\Lambda^\pm_\eta)$ is $\Gamma$-invariant. If not, then there exists $\zeta \in \Gamma$ such that $\F_\supp(\Lambda^\pm_\eta) \ne \zeta(\F_\supp(\Lambda^\pm_\eta))
$. Since $\F_\supp(\Lambda^\pm_\eta) = \F_\supp(\Lambda^+_\eta)) = \F_\supp(\Lambda^-_\eta)$, and similarly with $\zeta$ applied, it follows that $\{\Lambda^+_\eta,\Lambda^-_\eta\} \intersect \{\zeta(\Lambda^+_{\eta}),\zeta(\Lambda^-_{\eta})\} = \emptyset$. Applying the ``Inductive Step of Proposition 2.4 (2)'' from \SubgroupsFour, it follows that if the integer $m > 0$ is sufficiently large then $\eta' = \zeta \eta^m \zeta^\inv \eta^{-m}$ has a nongeometric lamination pair $\Lambda^\pm_{\eta'}$ such that $\A_\na(\Lambda^\pm_{\eta'})=\A$ and such that $\F_\supp(\Lambda^\pm_\eta)$ is strictly contained in $\F_\supp(\Lambda^\pm_{\eta'})$, contradicting maximality and therefore proving the claim. 

Applying $\Gamma$-invariance of $\F_\supp(\Lambda^\pm_\eta)$, for each $\theta \in \Gamma$ we have $\F_\supp(\theta \cdot \Lambda^\pm_\eta) = \theta(\F_\supp(\Lambda^\pm_\eta)) = \F_\supp(\Lambda^\pm_\eta)$
and so $\F_\supp(\Lambda^\pm_\eta) = \F_\supp(\Gamma \cdot \Lambda^\pm_\eta) = \{[F_n]\}$
where the latter equation follows from~\pref{item:fills}$_{\Gamma}$. Thus $\Lambda^\pm_\eta$ fills $F_n$ and so $\eta$ acts loxodromically on $\FS(F_n)$, by \FSLox. 

If $\eta$ were not fully irreducible rel~$\A$, that is if there existed an $\eta$-invariant free factor system $\A'$ contained strictly between $\A$ and $\{[F_n]\} = \F_\supp(\Lambda^\pm_\eta)$, then any conjugacy class carried by $\A'$ but not by $\A$ would not be weakly attracted to $\Lambda^\pm_\eta$, contradicting Fact~\ref{FactANA}~\pref{FactAnaChar}. 

Lastly, it remains to arrange that $\eta \in [\Gamma,\Gamma]$ (if it is not already true). Applying \prefH{item:finite lam} we may choose $\zeta \in \Gamma$ such that $\{\zeta(\Lambda^+_\eta),\zeta(\Lambda^-_\zeta)\} \intersect \{\Lambda^+_\eta,\Lambda^-_\eta\} = \emptyset$. Again applying the ``Inductive Step of Proposition 2.4'' from \SubgroupsFour, if the integer $m>0$ is sufficiently large then $\zeta \eta^m \zeta^\inv \eta^{-m} \in [\Gamma,\Gamma]$ satisfies all the portions of \prefH{filling laminations} already established for~$\eta$.

\smallskip\textbf{Case 2: $\Gamma$ is geometric above~$\A$.} The proof is similar to Case 1 but cites different results from \SubgroupsFour. Applying Proposition 2.2 (b) of \hbox{\SubgroupsFour,} there exists $\eta \in \Gamma$ which is fully irreducible rel~$\A$ and which has a geometric lamination pair $\Lambda^\pm_\eta \in \L(\eta;\A)$ whose nonattracting subgroup system $\A_\na(\Lambda^\pm_\eta)$ is $\Gamma$-invariant, that is, $\Stab(\A_\na(\Lambda^\pm_\eta)) = \Gamma$. Also, the nonattracting subgroup system has the form $\A_\na(\Lambda^\pm_\eta) = \A \union \{[C]\}$ for some rank~1 subgroup $C \subgroup F_n$ (\SubgroupsFour, Proposition 2.2 (b)(iii)). Note that $[C] \not \in \A$, for otherwise it would follow that $\A_\na(\Lambda^\pm_\eta) = \A$ is a free factor system, contradicting Fact~\ref{FactANA}~\pref{FactANAGeom}. 

We will need that there is no vertex group system $\A'$ strictly contained between $\A$ and $\A \union \{[C]\}$, for suppose that $\A \sqsubset \A' \sqsubset \A \union \{[C]\}$. Using that vertex group systems are malnormal (\SubgroupsOne\ Lemma 3.1), it follows that $\A$ is a subset of $\A'$, i.e. each element of $\A$ is also an element of $\A'$. If $\A'=\A$ we are done. Otherwise, consider any component $[C'] \in \A' - \A$. From malnormality it follows that $[C'] \sqsubset [C]$, and so up to conjugacy we may assume that $C' \subgroup C$ and hence $C,C'$ each have rank~$1$. By Proposition 3.2 of \SubgroupsOne, it follows that $C'=C$ and hence $\A' = \A \union \{[C]\}$.

By applying Proposition 2.3 (3)(b) of \SubgroupsFour\ we conclude that $\F_\supp(\Lambda^\pm_\eta)$ is $\Gamma$-invariant. Using the same argument as in Case~1, it follows that $\Lambda^\pm_\eta$ fills $F_n$, and so $\eta$ acts loxodromically on $\FS(F_n)$ \FSLox. 

To arrange that $\eta$ is in $[\Gamma,\Gamma]$, as before choose $\zeta \in \Gamma$ such that $\{\zeta(\Lambda^+_\eta),\zeta(\Lambda^-_\zeta)\} \intersect \{\Lambda^+_\eta,\Lambda^-_\eta\} = \emptyset$. Applying the ``Induction Step of Proposition 2.2'' from \SubgroupsFour, if the integer $m > 0$ is sufficiently large then $\eta' = \zeta \eta^m \zeta^\inv \eta^{-m} \in [\Gamma,\Gamma]$ has a geometric lamination pair $\Lambda^\pm_{\eta'} \in \L(\eta';\A)$ whose nonattracting subgroup system satisfies the containment relations
$$\A \sqsubset \A_\na(\Lambda^\pm_{\eta'}) \sqsubset \A_\na(\Lambda^\pm_{\eta}) = \A \union \{[C]\}
$$
As shown above, one of these containment relations is an equation, and it cannot be the first because $\Lambda^\pm_{\eta'}$ is geometric and so $\A_\na(\Lambda^\pm_{\eta'})$ is not a free factor system. Thus $\A_\na(\Lambda^\pm_{\eta'}) = \A_\na(\Lambda^\pm_\eta)$ and so $\Stab(\A_\na(\Lambda^\pm_{\eta'})) = \Gamma$. Applying Lemma 2.3~(3) of \SubgroupsFour\ it follows that $\eta'$ is fully irreducible rel~$\A$ and that $\F_\supp(\Lambda^\pm_{\eta'})$ is $\Gamma$-invariant, and so again by the same argument as in Case~1 it follows that $\eta'$ acts loxodromically on $\FS(F_n)$.

\bigskip

To complete the proof of Theorem~B there are just two more properties to verify:
\begin{enumeratecontinue}
\itemH \label{ItemGammaWWPD}
Every $\eta$ as in \prefH{filling laminations} --- that is, every $\eta \in [\Gamma,\Gamma]$ that is loxodromic and fully irreducible rel~$\A$ --- is a WWPD element for the action $\Gamma \act \FS(F_n)$.
\itemH \label{ItemGammaNonelementary}
The action $\Gamma \act \FS(F_n)$ is nonelementary.
\end{enumeratecontinue}
To prove item~\prefH{ItemGammaWWPD}, note that for each component $[A]$ of $\A$ the restriction of $\eta$ to $\Out(A)$ is in the commutator subgroup of the abelian group $\Gamma_A$ and hence is trivial. Applying the WWPD Construction Theorem it follows that $\eta$ is WWPD.

To prove item~\prefH{ItemGammaNonelementary}, take $\eta$ as in~\prefH{ItemGammaWWPD}, let $\Lambda^\pm_\eta$ be its filling lamination pair, and apply \prefH{item:finite lam} to conclude that the $\Gamma$ orbit of the pair $\Lambda^\pm_\eta$ is infinite. In particular, for some $\delta$ the pair $\Lambda^\pm_\eta$ is disjoint from the pair $\delta(\Lambda^\pm_\eta) = \Lambda^\pm_{\delta \eta \delta^\inv}$ which also fills. Thus $\eta^\delta = \delta\eta\delta^\inv$ is also loxodromic. Furthermore $\eta$, $\eta^\delta$ form an independent pair of loxodromics, because by \FSLox\ the set of attracting (repelling) points of loxodromic elements of $\Gamma$ in the Gromov boundary $\bdy\FS(F_n)$ corresponds bijectively and $\Gamma$-equivariantly to the set of repelling (attracting) laminations $\Lambda \in \L(\Gamma)$ such that $\Lambda$ fills~$F_n$.

\subsection{Proof that Theorems B and C imply Theorem A}
\label{SectionBCDImplyA}

Given a finitely generated subgroup $G \subgroup \Out(F_n)$ which is not virtually abelian, to prove Theorem~A it suffices to verify that the group $G$ satisfies the global WWPD hypothesis  (Definition~\ref{DefWWPDHypothesis}), for then we can apply the Global WWPD Theorem \cite{HandelMosher:WWPD} from which we obtain an embedding $\ell^1 \inject H^2_b(G;\reals)$. We shall apply Theorems B and C to verify the global WWPD hypothesis for~$G$. 

\paragraph{Setup:} Throughout this section we denote the following objects satisfying various properties:

\smallskip$\bullet$ Denote $G_0 = G \intersect \IAThree$, a finite index normal subgroup of~$G$. It follows that $G_0 \subgroup \IAThree$ is not abelian.

\smallskip$\bullet$ Choose $F_r \subgroup F_n$ to be a free factor of rank $r \le n$, with corresponding restriction homomorphism $\pi \from \Stab[F_r] \to \Out(F_r)$, such that the following two properties hold \emph{and} the rank $r$ is minimal with respect to these properties:

\smallskip \quad --- $G_0 \subgroup \Stab[F_r]$,

\smallskip \quad --- the group $\Gamma = \pi(G_0) \subgroup \Out(F_r)$ is not virtually abelian.

\smallskip\noindent
By restriction we have a surjective homomorphism $\pi \from G_0 \to \Gamma$. Since $G_0$ is contained in $\IAThree \intersect \Stab[F_r]$, and since the subgroup $\pi\bigl(\IAThree \intersect \Stab[F_r]\bigr) \subgroup \Out(F_r)$ is contained in $\IA_r(\Z/3)$, we have

\smallskip$\bullet$ $\Gamma \subgroup \IA_r(\Z/3)$.

\smallskip\noindent
By choice of the rank $r$ we have:

\smallskip$\bullet$ $\Gamma$ has virtually abelian restrictions.

\smallskip\noindent
To prove this, let $B \subgroup F_r$ be a proper free factor such that $\Gamma \subgroup \Stab_{\Out(F_r)}[B]$, meaning that $\Gamma$ stabilizes the $F_r$-conjugacy class of~$B$. It follows that $G_0$ stabilizes the $F_n$-conjugacy class of $B$, hence $G_0 \subgroup \Stab_{\Out(F_n)}[B]$. By minimality of $r$ it follows that $\image(\Gamma \mapsto \Out(B)) = \image(G_0 \mapsto \Out(B))$ is virtually abelian. Combining the last two bullet points with Corollary~\ref{CorollaryVirtuallyAbelian} we have:

\smallskip$\bullet$ $\Gamma$ has abelian restrictions.

\smallskip\noindent
With this setup, henceforth we work primarily in $\Out(F_r)$, thinking of $G$ and its subgroups somewhat abstractly rather than as subgroups of $\Out(F_n)$. In particular, our strategy for verifying the global WWPD hypothesis of $G$ is to work with a hyperbolic action of (a certain subgroup of) $G_0$ that factors through a hyperbolic action of (a certain subgroup of) $\Gamma \subgroup \Out(F_r)$, the latter action being obtained by applying Theorem B or~C. As a notational side effect, many standard notations should be interpreted in $F_r$, for example the attracting lamination notation $\L(\cdot)$.

The proof now breaks into two cases, depending on whether $\Gamma$ is an infinite lamination subgroup of $\Out(F_r)$.
 
\paragraph{Case 1: $\Gamma \subgroup \Out(F_r)$ is a finite lamination subgroup.} Applying Theorem~C, we obtain a finite index normal subgroup $N \normal \Gamma$, a hyperbolic complex $X$, and an isometric action $N \act X$, satisfying the following properties: 
\begin{enumerate}
\item[(a)] every element of $N$ acts elliptically or loxodromically on $X$;
\item[(b)] the action $N \act X$ is nonelementary;
\item[(c)] every element of $[N,N]$ is either elliptic or WWPD with respect to the action $N \act X$.
\end{enumerate}
We shall produce a certain rank~$2$ free subgroup $E \subgroup [N,N]$ that satisfies Definition~\ref{DefWWPDHypothesis} items~\pref{ItemLoxOrEll}, \pref{ItemFirstWWPD} and~\pref{ItemWWPDOrEllRestrict}, which proves the global WWPD hypothesis for~$\Gamma$. Definition~\ref{DefWWPDHypothesis}~\pref{ItemLoxOrEll} holds because it is identical to property~(a).

First we recall the notation that is found between the statement and proof of Lemma~\ref{LemmaCosetWWPD}. Let $i_1,\ldots,i_K \in \Aut(N)$ be outer representatives of the inner action of $\Gamma$ on $N$, with $i_1 = \Id_{N}$. Let $N \act_\kappa X$ denote the composed action $N \xrightarrow{i_\kappa} N \act X$; so $N \act_1 X$ is another notation for the given action $N \act X$. We also use $\act_\kappa$ to denote restriction of the action $N \act_\kappa X$ to subgroups of $N$. Since $[N,N]$ is a characteristic subgroup of $N$, it follows that property~(c) holds not just for the given action $N \act_1 X$ but for each of the composed actions $N \act_\kappa X$: 
\begin{itemize}
\item[(d)] For every $a \in [N,N]$ and every $\kappa = 1,\ldots,K$, the element $a$ is either elliptic or WWPD with respect to the action $N \act_\kappa X$. 
\end{itemize}
Apply (b) to obtain an independent pair of loxodromic elements for the action \hbox{$N \act_1 X$.} As is well known \cite{Gromov:hyperbolic}, one can then apply hyperbolic ping pong to high powers of these two elements to obtain a rank~$2$ Schottky subgroup $E_0 \subgroup N$. The commutator subgroup $[E_0,E_0] \subgroup E_0$ is free of infinite rank, and $[E_0,E_0] \subgroup [N,N]$, hence we may pick an independent pair of loxodromic elements in $[E_0,E_0]$ that freely generate a rank~$2$ subgroup $E_1 \subgroup [N,N]$. The action $E_1 \act_1 X$ is Schottky because it is a restriction of the Schottky action $E_0 \act_1 X$. With respect to the action $E_1 \act_1 X$, \emph{every} nonidentity element of $E_1$ is loxodromic, and by applying~(c), \emph{every} nonidentity element of $E_1$ is a WWPD element of the action $N \act_1 X$, thus \emph{every} rank~$2$ subgroup $E \subset E_1$ satisfies Definition~\ref{DefWWPDHypothesis}~\pref{ItemFirstWWPD}. 

To complete the verification of the WWPD hypothesis it remains to construct a rank~$2$ subgroup $E \subgroup E_1$ that satisfies Definition~\ref{DefWWPDHypothesis}~\pref{ItemWWPDOrEllRestrict}. And to do this, by applying (d) together with Lemma~\ref{LemmaCosetWWPD}, it suffices to show that $E$ satisfies the following:
\begin{itemize}
\item Each restricted action $E \act_\kappa X$ is either Schottky or elliptic, for~$\kappa=1,\ldots,K$.
\end{itemize}
For the construction of $E$ we adapt an induction argument of Bestvina and Fujiwara which is found in the proof of \cite[Theorem 8]{\BeFujiTag}.

Assume by induction that for some $\kappa = 1,\ldots,K-1$ we have a rank~$2$ subgroup $E_\kappa \subset E_1$ such that for each $i=1,\ldots,\kappa$ the action $E_\kappa \act_i X$ is either Schottky or elliptic, and hence the restriction of that action to any rank~$2$ subgroup of $E_\kappa$ is either Schottky or elliptic. We break into cases depending on the nature of the action $E_\kappa \act_{\kappa+1} X$. If that action is elliptic then, taking $E_{\kappa+1}=E_\kappa$, the induction is complete. If that action is not elliptic then there exists $\gamma \in E_\kappa$ which is loxodromic with respect to the action $E_\kappa \act_{\kappa+1} X$. It follows by (d) that $\gamma$ is WWPD with respect to the action $N \act_{\kappa + 1} X$ and so $\gamma$ is also WWPD with respect to the restricted action $E_\kappa \act_{\kappa+1} X$. Applying \cite[Corollary 2.6]{HandelMosher:WWPD} (see Proposition~\ref{PropActionWithWWPD}), the action $E_\kappa \act_{\kappa+1} X$ is either nonelementary or axial, and we consider those cases separately.

If the action $E_\kappa \act_{\kappa+1} X$ is nonelementary then, picking independent loxodromic elements and using hyperbolic ping-pong, we obtain a rank~$2$ subgroup $E_{\kappa+1} \subgroup E_\kappa$ for which the action $E_{\kappa+1} \act_{\kappa+1} X$ is Schottky, and the induction is complete. 

If the action $E_\kappa \act_{\kappa+1} X$ is axial, it preserves a two point subset $\{\xi,\eta\}$ and contains at least one loxodromic element. After replacing $E_\kappa$ with the kernel of the action on $\{\xi,\eta\}$ --- a subgroup of index at most $2$ in $E_\kappa$ --- we may assume that $E_\kappa$ fixes both $\xi$ and $\eta$. Applying item~(c) of Theorem~C, any loxodromic element of $E_\kappa$ is strongly axial, and so there exists a quasi-isometric embedding $\ell \from \reals \to X$ and a homomorphism $\tau \from E_\kappa \to \reals$ such that for all $\theta \in E_\kappa$ and $s \in \reals$ we have $\theta(\ell(r))=\ell(r + \tau(\theta))$.  Clearly $\Ker(\tau)$ is elliptic, since it fixes the point $\ell(0)$. Since $E_\kappa$ has a loxodromic element, 
%
%
%
it follows that $\tau$ has infinite image in $\R$, and so $\kernel(\tau)$ is a free subgroup of $E_\kappa$ of rank~$\ge 2$. Taking $E_{\kappa+1} \subgroup \kernel(\tau)$ to be any rank~$2$ subgroup, the action $E_{\kappa+1} \act_{\kappa+1} X$ is elliptic and the induction is complete. 

We can now set $E = E_K$ to complete the proof in Case~1.

\subparagraph{Remark.} Our case where $E_\kappa \act_{\kappa+1} X$ is axial corresponds to the second of four bulleted cases of the argument of Bestvina and Fujiwara found in \cite[page 85]{\BeFujiTag}.
In that case they make use of the fact that if a subgroup of the mapping class group fixes the attracting/repelling foliation pair of some pseudo-Anosov element, then that subgroup is virtually cyclic. No similar fact is available in our present context; we instead make use of the ``strongly axial'' conclusion in Theorem~C.

\paragraph{Case 2: $\Gamma \subgroup \Out(F_r)$ is an infinite lamination subgroup.} Letting $\A$ be any maximal, proper, $\Gamma$-invariant free factor system in $F_r$, the hypotheses of Theorem~B are satisfied, and from its conclusions we obtain the following: 
\begin{itemize}
\item[(a)] The action $\Gamma \act \FS(F_r)$ is nonelementary; 
\item[(b)] There exists a loxodromic element of $[\Gamma,\Gamma]$ which is fully irreducible rel~$\A$, and any such element satisfies WWPD for the action $\Gamma \act \FS(F_r)$
\end{itemize}
We may apply all of the numbered properties \prefH{ItemInfLamFulAbe}\,--\,\prefH{ItemGammaNonelementary} from Section~\ref{SectionReducingBToE}, each of which was proved starting only with the assumption that $\Gamma$ and~$\A$ satisfy the hypotheses of Theorem~B. 

We also adopt all the notation from the setup at the beginning of Section~\ref{SectionBCDImplyA}, but to simplify notation and highlight parallels with Case~1 we use the notation 
$$N=G_0
$$
We have an action $N \act \FS(F_r)$, given by the composition $N \xrightarrow{\pi} \Gamma \act \FS(F_r)$. Evidently the following properties corresponding to (a) and~(b) are satisfied:
\begin{itemize}
\item[(a)${}_N$] The action $N \act \FS(F_r)$ is nonelementary;
\item[(b)${}_N$] There exists a loxodromic $\phi \in [N,N]$ such that $\pi(\phi) \in \Out(F_r)$ is fully irreducible rel~$\A$, and any such $\phi$ satisfies WWPD for the action $N \act \FS(F_r)$.
\end{itemize}
The proof in Case~2 will parallel Case~1 for a brief while, before diverging. Choose $i_1,\ldots,i_K \in \Aut(N)$ to be outer representatives of the inner action of $G$ on $N$, with $i_1 = \Id_{N}$. Let $N \act_\kappa \FS(F_r)$ denote the composed action $N \xrightarrow{i_\kappa} N \act \FS(F_r)$, and use the same action symbol $\act_\kappa$ for restrictions of $N \act_\kappa \FS(F_r)$ to subgroups of $N$. Definition~\ref{DefWWPDHypothesis}~\pref{ItemLoxOrEll} is simply property (a)${}_N$ above. Thus to verify the global WWPD hypothesis for $G$ we must produce a rank~2 free subgroup $E \subgroup [N,N]$ and use it to verify Definition~\ref{DefWWPDHypothesis}~(\ref{ItemFirstWWPD}, \ref{ItemWWPDOrEllRestrict}), where Lemma~\ref{LemmaCosetWWPD} has been used to rewrite Definition~\ref{DefWWPDHypothesis}~\pref{ItemWWPDOrEllRestrict}: 
\begin{itemize}
\item[\pref{ItemFirstWWPD}] The restricted action $E \act \FS(F_r)$ is Schottky and its nontrivial elements all satisfy WWPD with respect to the action $N \act \FS(F_r)$.
\item[\pref{ItemWWPDOrEllRestrict}${}'$] For each $\kappa=1,\ldots,K$, the action $E \act_\kappa \FS(F_r)$ is either elliptic, or it is Schottky and its nontrivial elements all satisfy WWPD with respect to the action $N \act_\kappa \FS(F_r)$.
\end{itemize}

Where Cases~1 and~2 diverge is that we do \emph{not} know that every loxodromic $\phi \in [N,N]$ is a WWPD element for the action $N \act \FS(F_r)$; we know this only for those $\phi$ such that $\pi(\phi) \in \Out(F_r)$ is fully irreducible rel~$\A$. Thus, to verify the global WWPD hypothesis for $G$ using the action $N \act \FS(F_r)$ there is still quite a bit of intricate work to do involving subgroup decomposition theory, lamination ping-pong, and Ghosh's theorem, in order to discover the needed WWPD elements.

\medskip

Define $\pi_\kappa \from N \to \Gamma$ to be the composition $N \xrightarrow{i_\kappa} N \xrightarrow{\pi} \Gamma$. The action $N \act_\kappa \FS(F_r)$ is thus the same as the composed action $N \xrightarrow{\pi_\kappa} \Gamma \act \FS(F_r)$. For each component $[A]$ of $\A$, since the restriction of $\Gamma$ to $\Out(A)$ is abelian, it follows that the restriction of $[\Gamma,\Gamma]$ to $\Out(A)$ is trivial. Since the characteristic subgroup $[N,N]$ is preserved by the isomorphism $i_\kappa$, we have $\pi_\kappa[N,N] = \pi[N,N] \subgroup [\Gamma,\Gamma]$ for $\kappa = 1,\ldots,K$, and so
\begin{itemize}
\item[(c)${}_N$] For each component $[A]$ of $\A$ and each $\kappa=1,\ldots,K$, the restriction of $\pi_\kappa[N,N]$ to $\Out(A)$ is trivial.
\end{itemize}

Consider $a \in N$, and denote its images in $\Gamma$ as $\alpha_\kappa = \pi_\kappa(a) \in \Gamma$, $1 \le \kappa \le K$. We~say that $\kappa$ is a \emph{PG index for~$a$} if $\L(\alpha_\kappa; \A) = \emptyset$; otherwise $\kappa$ is an \emph{EG index for~$a$}. Assuming that $\kappa$ is an EG index for $a$, we say that $\kappa$ is a \emph{nongeometric index for $a$} if some element of $\L(\alpha_\kappa; \A)$ is non-geometric; otherwise $\kappa$ is a \emph{geometric index for~$a$}. Let $t_1 \ge 0$ be the maximum number of non-geometric indices that occurs for any element of $N$, and let $\M'$ be the set of all $a \in N$ having $t_1$ non-geometric indices; note that $t_1=0$ if and only if the subgroup $\Gamma\subgroup\Out(F_r)$ is geometric above~$\A$. Let $t_2 \ge t_1$ be the maximum number of EG indices that occur for some $a \in \M'$, and let $\M$ be the set of all $a \in \M'$ having $t_2$ EG indices. There exists $a \in N$ and $\kappa \in \{1,\ldots,K\}$ such that $\kappa$ is an EG index for $a$, because by property~(a) some element of $\Gamma$ is loxodromic and so has a filling lamination. The set~$\M$ is therefore nonempty and $t_2 \ge 1$. After permuting the $\kappa$'s if necessary we may assume that the following subset of $\M$ is nonempty:
\begin{align*}
\M_0 = \{a \in N \suchthat & \,\, \text{$\kappa$ is a non-geometric index of $a$ for $1 \le \kappa \le t_1$, and} \\
&\,\,\text{$\kappa$ is a geometric index of~$a$ for $t_1 <  \kappa \le t_2$}
\}
\end{align*}
That is, $a \in \M_0$ if and only if: $\L(\alpha_\kappa;\A)$ has a nongeometric lamination for each $1 \le \kappa \le t_1$; $\L(\alpha_\kappa;\A)$ is a nonempty set of geometric laminations for $t_1 < \kappa \le t_2$; and $\L(\alpha_\kappa;\A)$ is empty for $t_2 < \kappa \le K$.

Given $a \in \M_0$ with $\alpha_\kappa = \pi_\kappa(a)$, an \emph{assignment of lamination pairs} for~$a$ is a function
$$\kappa \mapsto \Lambda^\pm_{\alpha_\kappa} \in \L(\alpha_\kappa;\A) \quad\text{defined for $1 \le \kappa \le t_2$,}
$$
denoted in shorthand as $\bigl(\Lambda^\pm_{\alpha_\kappa} \bigr)$. We say that this assignment $\bigl(\Lambda^\pm_{\alpha_\kappa}\bigr)$ is \emph{$\M_0$-consistent} if $\Lambda^\pm_{\alpha_\kappa}$ is non-geometric for each $1 \le \kappa \le t_1$. 


\begin{proposition}  \label{multiple filling} 
\quad\hfill
\begin{enumerate}
\item \label{good for all kappa} There exists $a  \in \M_0$ with an $\M_0$-consistent assignment $\bigl(\Lambda^\pm_{\alpha_\kappa}\bigr)$ of filling lamination pairs, such that $\alpha_\kappa = \pi_\kappa(a) \in \Gamma$ is fully irreducible rel $\A$ for each $1 \le \kappa \le t_2$. 
\item \label{no mixed cases} Either $t_1 = 0$ or  $t_2 = t_1 \ge 1$. More precisely: either each $\Lambda^\pm_{\alpha_\kappa}$ is non-geometric for each $\kappa$; or $\Gamma$ is geometric above $\A$ and therefore $\Lambda^\pm_{\alpha_\kappa}$ is geometric for each~$\kappa$.
\end{enumerate}
\end{proposition}

\noindent
Before proving Proposition \ref{multiple filling}, we apply it to the construction of a rank~2 subgroup $E \subgroup [N,N]$ that satisfies (2) and (3)${}'$.

Choose $a \in \M_0$ and an assignment of lamination pairs $\bigl(\Lambda^\pm_{\alpha_\kappa}\bigr)$ satisfying Proposition~\ref{multiple filling}~\pref{good for all kappa}. Consider the set $C_r$ of closed subsets of $\B(F_r)$, a set on which the group $\Out(F_r)$ acts naturally. Each lamination $\Lambda^-_{\alpha_\kappa}$, $\Lambda^+_{\alpha_\kappa}$ for $1 \le \kappa \le t_2$ is an element of $C_r$. By \cite[Corollary 1.3]{\FSLoxTag} we have equality of stabilizer subgroups $\Stab_{\Gamma}(\Lambda^-_{\alpha_\kappa}) = \Stab_{\Gamma}(\Lambda^+_{\alpha_\kappa})$, and by property~\prefH{item:finite lam} this subgroup has infinite index in~$\Gamma$. The subgroup $\Stab_{\Gamma}\{\Lambda^-_{\alpha_\kappa},\Lambda^+_{\alpha_\kappa}\}$, which contains $\Stab_{\Gamma}(\Lambda^-_{\alpha_\kappa}) = \Stab_{\Gamma}(\Lambda^+_{\alpha_\kappa})$ with index at most~$2$, therefore also has infinite index in~$\Gamma$. 

\bigskip
\textbf{Remark.} Since each $\Stab_{\Gamma}\{\Lambda^-_{\alpha_\kappa},\Lambda^+_{\alpha_\kappa}\}$ has infinite index in $\Gamma$, the axial case which came up in the induction argument of Case~1 does not come up at all in Case~2. 
The indices $\kappa$ for which the desired action $E \act_\kappa \FS(F_r)$ is elliptic will be precisely those for which $t_2 < \kappa \le K$.

\bigskip

For each $1 \le \kappa \le t_2$, consider the composed action $N \xrightarrow{\pi_\kappa} \Gamma \subgroup \Out(F_r) \act C_r$, denoted as $N \act_\kappa C_r$. The Second Sublemma of Lemma 2.1 of \SubgroupsFour\ applied to the actions $N \act_\kappa C_r$ produces an element $b \in N$ such that for all $1 \le \kappa \le t_2$, letting $\beta_\kappa : = \pi_\kappa(b)$, we have  
$$\hphantom{(*)} \qquad \beta_\kappa(\{\Lambda^+_{\alpha_\kappa}, \Lambda^-_{\alpha_\kappa}\}) \ne  \{\Lambda^+_{\alpha_\kappa}, \Lambda^-_{\alpha_\kappa}\}
$$ 
and therefore by the Independence Theorem \cite[Theorem 1.2]{\FSLoxTag} we have
$$(*) \qquad \beta_\kappa(\{\Lambda^+_{\alpha_\kappa}, \Lambda^-_{\alpha_\kappa}\}) \intersect  \{\Lambda^+_{\alpha_\kappa}, \Lambda^-_{\alpha_\kappa}\} = \emptyset
$$
Let $c = bab^{-1} \in N$, let $\gamma_\kappa = \pi_\kappa(c) = \beta_\kappa \alpha_\kappa \beta_\kappa^{-1}$, and let $\Lambda^\pm_{\psi_\kappa} =  \beta_\kappa( \Lambda^\pm_{\alpha_\kappa}) \in \L^\pm(\gamma_\kappa; \A)$. Then $\gamma_\kappa $ is fully irreducible rel $\A$ and $ \{\Lambda^+_{\alpha_\kappa}, \Lambda^-_{\alpha_\kappa}\} \cap  \{\Lambda^+_{\gamma_\kappa}, \Lambda^-_{\gamma_\kappa}\}  = \emptyset$  for all $1 \le \kappa \le t_2$; moreover, $\bigl(\Lambda^\pm_{\gamma_\kappa}\bigr)$ is an $\M_0$-consistent assignment of filling lamination pairs for~$c$.

We now apply Proposition~\ref{PropGhosh} to the pair of elements $\alpha_\kappa$ and $\gamma_\kappa$ for each $1 \le \kappa \le t_2$, noting that we have already verified hypotheses~\pref{ItemGhoshFullyIrr} and~\pref{ItemGhoshDiffLams} of Proposition~\ref{PropGhosh} for that pair, and that hypothesis~\pref{ItemGhoshGeomAlt} follows from $\M_0$-consistency of each of the assignments $\bigl(\Lambda^\pm_{\alpha_\kappa}\bigr)$ and $\bigl(\Lambda^\pm_{\gamma_\gamma}\bigr)$ and from Proposition~\ref{multiple filling}~\pref{no mixed cases}. From the conclusion of Proposition~\ref{PropGhosh}, for each $1 \le \kappa \le t_2$ there exists $M_\kappa$ so that for each $m \ge M_\kappa$ the subgroup $E'_\kappa$ of $\Gamma$ generated by $\alpha_\kappa^m$ and $\gamma_\kappa^m$ is free of rank $2$, each non-trivial element $\xi_\kappa \in E'_\kappa$ is irreducible rel $\A$, each $\xi_\kappa$ has a unique lamination pair $\Lambda^\pm_{\xi_\kappa}$ that fills rel $\A$, and $\Lambda^\pm_{\xi_\kappa}$ is nongeometric for $1 \le \kappa \le t_1$ and geometric for $t_1 < \kappa \le t_2$. Increasing each $M_\kappa$ to $M = \max_\kappa M_\kappa$, we obtain a subgroup $E'$ of $N$ freely generated by $a^m$ and $c^m$ such that that $E'_\kappa = \pi_\kappa(E')$, and such that for each nontrivial $x \in E'$, letting $\xi_\kappa = \pi_\kappa(x)$, the assignment $\bigl(\Lambda^\pm_{\xi_\kappa}\bigr)$ is $\M_0$-consistent. 

By \cite[Theorems 1.1 and 1.2]{\FSLoxTag}, the actions of $\alpha_\kappa$ and $\gamma_\kappa$ on $\FS(F_r)$ are loxodromic and independent. 
By standard methods of hyperbolic ping-pong (see for example \cite{HandelMosher:WWPD}), after a further increase of $M$, we may also assume that the action of each $E'_\kappa$ on $\FS(F_r)$ is Schottky. In particular, each non-trivial $\xi_\kappa \in E'_\kappa$ has a filling lamination pair, and that pair must be $\Lambda^\pm_{\xi_\kappa}$ since all other lamination pairs for $\xi_\kappa$ are supported by~$\A$. It follows that for each $1 \le \kappa \le t_2$ the restricted action  $E' \act_\kappa \FS(F_r)$ is Schottky and each non-trivial element of $\pi_\kappa(E')$ is irreducible rel~$\A$. Moreover each non-trivial element of $E'$ is contained in~$\M_0$.

Choose a rank~$2$ free subgroup $E \subgroup [E' \cap N,E' \cap N] \subgroup [N,N]$. Note that each restricted action $E \act_\kappa \FS(F_r)$ is still Schottky for each $1 \le \kappa \le t_2$, and each non-trivial element of $E$ is contained in $\M_0$. Applying property~(c)${}_N$, for each component $[A]$ of $\A$ the restriction of each $\pi_\kappa(E)$ to $\Out(A)$ is trivial. The WWPD Construction Theorem therefore implies that $E\act_\kappa \FS(F_r)$ is a WWPD Schottky group for each $1\le \kappa \le t_2$.    Since each non-trivial element of $E$ is contained in $\M_0$, the actions $E\act_\kappa \FS(F_r)$ must be elliptic for $\kappa > t_2$. This completes the verification of (2) and (3)${}'$ using $E$, thus finishing Case~2 subject to the proof of Proposition \ref{multiple filling}.

\bigskip

Next we present two lemmas needed to set up the ping-pong proof of Proposition~\ref{multiple filling}. Most of the following lemma is cited from Lemma 2.3~(3) of \SubgroupsFour.

\begin{lemma} \label{invariant stabilizer} 
Suppose that $\Gamma \subgroup \IA_r(\Z/3)$ is irreducible relative to a free factor system~$\A$, and suppose that $\phi \in \Gamma$ is rotationless and has a geometric lamination pair $\Lambda^\pm_\phi \in \L(\phi;\A)$ such that $\A_\na(\Lambda^\pm_\phi)$ is $\Gamma$-invariant, equivalently $\Stab_\Gamma(\A_\na(\Lambda^\pm_\phi))=\Gamma$.  Then   
\begin{enumerate}
\item\label{ItemInvStIrr}
$\phi$ is irreducible rel $\A$ (\SubgroupsFour\ Lemma 2.3~(3)(a)).
\item\label{ItemInvStFFS}
The free factor support of $\Lambda^\pm_\phi$ is $\Gamma$-invariant  (\SubgroupsFour\ Lemma 2.3~(3)(b)).
\item\label{ItemInvStMax}
$\A_\na\Lambda^\pm_{\phi} = \A \cup \{[C]\}$ where $C$ is a maximal infinite cyclic subgroup of $F_r$  (\SubgroupsFour\ Lemma 2.3~(3)(c)). 
\item\label{ItemInvStGeom}
$\Gamma$ is geometric above $\A$.
\end{enumerate} 
\end{lemma}
 
\begin{proof} To prove~\pref{ItemInvStGeom}, if $\Gamma$ is not geometric above $\A$ then by Proposition 2.4 of \SubgroupsFour\ there exists a rotationless $\theta \in \Gamma$ that is irreducible rel $\A$ and there exists $\Lambda^\pm_\theta \in \L(\theta;\A)$ such that  $\A_\na(\Lambda^\pm_\theta) = \A$.  But this contradicts the fact that $[C]$ is $\theta$-invariant and not carried by the free factor system $\A$.
\end{proof}

Ping-pong arguments in groups can start with one player --- that is, one group element --- producing a second player by conjugating the first. The conjugating element is chosen carefully, depending on the desired outcome. The proof of Proposition~\ref{multiple filling} is a ping-pong game in the group $N$, carried out inductively: given $a_i \in \M_0$, after choosing a conjugator $b_i \in N$ we produce the second player $c_i = b_i a_i b_i^{-1} \in \M_0$ and then define $a_{i+1} = c_i^m a_i^n$ for sufficiently large $m,n$.  We then iterate this until the resulting $a_i$ satisfies the conclusions of Proposition~\ref{multiple filling}. Each step of this iteration can also be viewed as $t_2$ simultaneous ping-pong games in the quotient group $\Gamma = \pi(N)$, with $t_2$ first players $\alpha_\kappa \, (= \pi_\kappa(a_1))$, requiring a \emph{very} careful and consistent choice of $t_2$ conjugating elements $\beta_\kappa \, (= \pi_\kappa(b_1))$. The following lemma describes the choice of conjugating element $b_1$. It is a straightforward generalization of Lemma 2.1 of \SubgroupsFour, which is used to choose the conjugating maps for the proof of Theorem~I of \SubgroupsFour.


\begin{lemma} \label{simultaneous conjugation}  Suppose that $a \in \M_0$ and $\alpha_\kappa = \pi_\kappa(a) \in \Gamma$, $1 \le \kappa \le K$. Let $\Lambda^\pm_{\alpha_\kappa} \in \L(\alpha_\kappa;\A)$ be an $\M_0$-consistent assignment of lamination pairs, with generic leaves $\ell^\pm_{\alpha_\kappa}$, respectively. There exists $b \in N$ such that $\beta_\kappa = \pi_\kappa(b)$ satisfies the following properties for all $1 \le \kappa \le t_2$.
\begin{enumerate}
\item \label{conjugator1} 
None of the lines $\beta_\kappa(\ell^+_{\alpha_\kappa})$, $\beta_\kappa(\ell^-_{\alpha_\kappa})$, $\beta^\inv_\kappa(\ell^+_{\alpha_\kappa})$, $\beta^\inv_\kappa(\ell^-_{\alpha_\kappa})$ is carried by $\A_\na\Lambda^\pm_{\alpha_\kappa}$.
\item \label{conjugator2}
$\beta_\kappa \{\Lambda^+_{\alpha_\kappa}, \Lambda^-_{\alpha_\kappa}\} \cap \{\Lambda^+_{\alpha_\kappa}, \Lambda^-_{\alpha_\kappa}\} = \emptyset$.  
\item  \label{conjugator3} 
If $\A_\na\Lambda^\pm_{\alpha_\kappa}$ is not $\Gamma$-invariant then  $\beta_\kappa(\A_\na\Lambda^\pm_{\alpha_\kappa}) \ne \A_\na\Lambda^\pm_{\alpha_\kappa}$. 
\item  \label{conjugator4}  
If $\F_\supp\Lambda^\pm_{\alpha_\kappa}$ is not $\Gamma$-invariant then 
$\beta_\kappa(\F_\supp\Lambda^\pm_{\alpha_\kappa}) \ne \F_\supp\Lambda^\pm_{\alpha_\kappa}$.
\end{enumerate}
\end{lemma}

\begin{proof} Following the proof of Lemma 2.1 of \SubgroupsFour, we shall show the following for all $1 \le \kappa \le t_2$:
\begin{itemize}
\item[$(a)_\kappa$] There is a finite index subgroup $\Gamma_\kappa\subgroup \Gamma$ such that for any $\beta \in \Gamma_\kappa$ none of the lines $\beta(\ell^+_{\alpha_\kappa})$, $\beta(\ell^-_{\alpha_\kappa})$, $\beta^\inv(\ell^+_{\alpha_\kappa})$, $\beta^\inv(\ell^-_{\alpha_\kappa})$, is carried by $\A_\na\Lambda^\pm_{\alpha_\kappa}$.    
\item[$(b)_\kappa$] $\Stab_{\Gamma}\{\Lambda^+_{\alpha_\kappa}, \Lambda^-_{\alpha_\kappa}\}$ has infinite index.
\item[$(c)_\kappa$] If $\A_\na\Lambda^\pm_{\alpha_\kappa}$ is not $\Gamma$-invariant then  $\Stab_{\Gamma}(\A_\na\Lambda^\pm_{\alpha_\kappa})$ has infinite index.
\item[$(d)_\kappa$] If $\F_\supp \Lambda^\pm_{\alpha_\kappa}$ is not $\Gamma$-invariant then $\Stab_{\Gamma}(\F_\supp\Lambda^\pm_{\alpha_\kappa})$ has infinite index.
\end{itemize}
Item $(a)_\kappa$ follows from the First Sublemma in the proof of Lemma~2.1 of \SubgroupsFour. Item $(b)_\kappa$ follows from property~\prefH{item:finite lam}.
Item $(c)_\kappa$ follows from Lemma 2.3 of \SubgroupsFour. Item $(d)_\kappa$ follows because if $\Stab_{\Gamma}(\F_\supp\Lambda^\pm_{\alpha_\kappa})$ has finite index then for each $\theta \in \Gamma$ some power of $\theta$ fixes the free factor system $\F_\supp\Lambda^\pm_{\alpha_\kappa}$, but since $\theta \in \IAThree$ it follows that $\theta$ itself fixes $\F_\supp\Lambda^\pm_{\alpha_\kappa}$, by Theorem~3.1 of \SubgroupsTwo.

By pulling back the infinite index subgroups in $(b)_\kappa$, $(c)_\kappa$ and $(d)_\kappa$ under the homomorphism $\pi_\kappa \from N \to \Gamma$, and letting $\kappa=1,\ldots,t_2$ vary, we obtain a finite collection of infinite index subgroups of $N$. Applying the Second Sublemma in the proof of Lemma~2.1 of \SubgroupsFour, there is an infinite subset of $N$ any two elements of which lie in distinct left cosets of each subgroup in this collection. By the pigeonhole principle, this subset must contain two elements $b_1 \ne b_2$ lying in the same left coset of the finite index subgroup 
$$\bigcap_{\kappa=1}^{t_2} j^\inv_\kappa(\Gamma_\kappa)
$$
The element $b = b_1^\inv b^{\vphantom{-1}}_2$ satisfies the conclusions of the~lemma.
\end{proof}

\bigskip
 
To finish the proof that Theorems B and C imply Theorem A, what remains is:
     
\begin{proof}[Proof of Proposition~\ref{multiple filling}] \quad  The construction is inductive, using the fact that $\M_0 \ne \emptyset$ to choose an arbitrary $a(1) \in \M_0$ and an $\M_0$-consistent assignment of lamination pairs $\bigl(\Lambda^\pm_{\pi_\kappa(a_1)}\bigr)$, and producing a sequence $a(2),a(3), \ldots \in \M_0$ with $\M_0$-consistent assignments $\bigl(\Lambda^\pm_{\pi_\kappa(a(i))}\bigr)$. The first part of the proof is a description of a single step in the induction. The remainder of the proof explains how the induction eventually terminates in an element of $\M_0$ satisfying conclusions \pref{good for all kappa} and~\pref{no mixed cases} of the proposition.

Assuming that we have produced $a_i \in \M_0$, we describe how to produce $a_{i+1}$. Let $\alpha_{i,\kappa} := \pi_\kappa(a_i)$ for $1 \le \kappa \le t_2$. Let $\bigl(\Lambda^\pm_{\alpha_{i,\kappa}}\bigr)$ be the corresponding $\M_0$-consistent assignment of lamination pairs.  
For notational convenience we suppress the subscript ``$i$'' for now, writing $a$, $\alpha_\kappa = \pi_\kappa(a)$, and $\bigl(\Lambda^\pm_{\alpha_\kappa}\bigr)$.

Choose $b \in N$ as in Lemma~\ref{simultaneous conjugation}, the conclusions of which we shall apply below. Let $\beta_\kappa = \pi_\kappa(b)$, let $c  = b a b^{-1}$,  let $\gamma_{\kappa} =  \pi_\kappa(c) = \beta_{\kappa} \alpha_{\kappa} \beta^{-1}_{\kappa}$, and let $\Lambda^\pm_{\gamma_{\kappa}} = \beta_{{\kappa}}(\Lambda^\pm_{\alpha_{\kappa}}) \in \L(\gamma_{\kappa};\A)$. Fixing $m,n \ge 1$ subject to lower bounds to be given, consider $x = c^m a^n$ and for each $1 \le \kappa \le t_2$ consider $\xi_\kappa = \pi_\kappa(x) = \gamma_\kappa^m \alpha_\kappa^n$. 
     
We want to simultaneously apply Proposition~\ref{PropPingPongVer} (aka Proposition 1.7 of \SubgroupsFour) to $\alpha_{\kappa}$, $\Lambda^\pm_{\alpha_{\kappa}}$, $\gamma_{\kappa}$ and $\Lambda^\pm_{\gamma_{\kappa}}$ for $1 \le \kappa \le t_2$, and so we must check its hypotheses. After replacing $a$, and hence $c$, with an iterate of itself, we may assume for $1 \le \kappa \le t_2$ that $\Lambda^\pm_{\alpha_{\kappa}}$ and $\Lambda^\pm_{\gamma_{\kappa}}$ have generic leaves $\ell^\pm_{\alpha_{\kappa}}$ and $\ell^\pm_{\gamma_{\kappa}} = \beta_{{\kappa}}(\ell^\pm_{\alpha_{\kappa}})$ that are fixed by $\phi$ and $\psi$ respectively, with fixed orientation. Note that if $1 \le \kappa \le t_1$ then the pair $\Lambda^\pm_{\gamma_\kappa}$ is nongeometric because it is the $\beta_\kappa$ image of the nongeometric pair $\Lambda^\pm_{\alpha_\kappa}$. It remains to check hypotheses (a) and (i)--(iv) of Proposition~\ref{PropPingPongVer}.  
    
Hypothesis~(a) of Proposition~\ref{PropPingPongVer} says that $\A \sqsubset \A_\na\Lambda^\pm_{\alpha_{\kappa}}$ and $\A \sqsubset \A_\na\Lambda^\pm_{\gamma_{\kappa}} = \beta_{\kappa} \A_\na\Lambda^\pm_{\alpha_{\kappa}}$. This follows from the assumption that  $\Lambda^\pm_{\alpha_{\kappa}}$, and hence also $\Lambda^\pm_{\gamma_{\kappa}}$, is not carried by $\A$: every iterate under $\alpha$ and $\gamma$ of every conjugacy class carried by $\A$ is also carried by $\A$, and so such iterates cannot weakly converge to a lamination that is not carried by $\A$.
    
Hypotheses (i)--(iv) of Proposition~\ref{PropPingPongVer} follow from the conclusions of Fact~\ref{FactMutualAttraction}, whose two hypotheses we verify by applying the conclusions of Lemma~\ref{simultaneous conjugation}. Hypothesis~\pref{ItemPhiPsiDiffLams} of Fact~\ref{FactMutualAttraction} is just Lemma~\ref{simultaneous conjugation}~\pref{conjugator2}. Hypothesis~\pref{ItemPhiPsiANA} of Fact~\ref{FactMutualAttraction} requires that: neither of the lines $\ell^\pm_{\gamma_{\kappa}} = \beta_{\kappa}(\ell^\pm_{\alpha_{\kappa}})$ should be carried by $\A_\na\Lambda^\pm_{\alpha_{\kappa}}$; also neither of the lines $\ell^\pm_{\alpha_{\kappa}}$ should be carried by $\A_\na\Lambda^\pm_{\gamma{\kappa}} = \beta_{\kappa} \bigl(\A_\na\Lambda^\pm_{\alpha_{\kappa}})$ which is equivalent to saying that neither of the lines $\beta^\inv_{\kappa}(\ell^\pm_{\alpha_{\kappa}})$ is carried by $\A_\na\Lambda^\pm_{\alpha_{\kappa}}$; but these altogether follow from Lemma~\ref{simultaneous conjugation}~\pref{conjugator1}.

We may now apply the conclusions of Proposition~\ref{PropPingPongVer} for each $1 \le \kappa \le t_2$, producing threshold constants $M^1_{\kappa}$, and if $m,n \ge M^1_{\kappa}$ producing laminations $\Lambda^+_{\xi_\kappa} \in \L(\xi_{\kappa};\A)$ and $\Lambda^-_{\xi_\kappa} \in \L(\xi^\inv_\kappa;\A)$. 

Define a simultaneous threshold constant \hbox{$M^1 = \max_{1 \le \kappa \le t_2} M^1_\kappa$}. Fix $m,n \ge M^1$, and so for each $1 \le \kappa \le t_2$ we have a lamination pair $\Lambda^\pm_{\xi_\kappa} \in \L^\pm(\xi_\kappa)$. 

In the case that $1 \le \kappa \le t_1$, since $\Lambda^\pm_{\alpha_\kappa}$ and $\Lambda^\pm_{\gamma_\kappa}$ are both nongeometric, Proposition~\ref{PropPingPongVer} lets us conclude that $\Lambda^+_{\xi_\kappa}$ and $\Lambda^-_{\xi_\kappa}$ are both nongeometric. In the case that $t_1<\kappa\le t_2$, the laminations $\Lambda^+_{\xi_\kappa}$ and $\Lambda^-_{\xi_\kappa}$ are both geometric by definition of $\M_0$. In either case Proposition~\ref{PropPingPongVer}~(6) lets us conclude that $\Lambda^\pm_{\xi_\kappa}$ form a dual lamination pair. Thus we have shown that $\bigl(\Lambda^\pm_{\xi_\kappa}\bigr)$ is an $\M_0$-consistent assignment of lamination pairs for $x$. 

Next we show for each $1 \le \kappa \le t_2$ that
\begin{description}
\item[$(*)$] $\A_\na\Lambda_{\xi_\kappa} \sqsubset \A_\na\Lambda_{\alpha_\kappa}$. Furthermore, this containment is strict if $\A_\na\Lambda_{\alpha_\kappa}$ is not $\Gamma$-invariant.
\end{description}
The containment $\A_\na\Lambda_{\xi_\kappa} \sqsubset \A_\na\Lambda_{\alpha_\kappa}$ follows from Proposition~\ref{PropPingPongVer}~(1) which moreover gives the containment $\A_\na\Lambda_{\xi_\kappa} \sqsubset \A_\na\Lambda_{\gamma_\kappa}$. For the ``furthermore'' part, if $\A_\na\Lambda_{\alpha_\kappa}$ is not $\cH$-invariant then, by Lemma~\ref{simultaneous conjugation}~\pref{conjugator3}, we have 
$$\A_\na\Lambda^\pm_{\gamma_\kappa} = \beta_{\kappa}(\A_\na\Lambda^\pm_{\alpha_{\kappa}})  \ne \A_\na\Lambda^\pm_{\alpha_{\kappa}},
$$
Arguing by contradiction, if furthermore $\A_\na\Lambda_{\xi_\kappa} = \A_\na\Lambda_{\alpha_\kappa}$ then we have 
$$\A_\na\Lambda_{\alpha_\kappa} = \beta_\kappa^\inv \bigl( \A_\na\Lambda_{\gamma_\kappa} \bigr) \sqsupset \beta_\kappa^\inv \bigl( \A_\na\Lambda_{\xi_\kappa} \bigr) = \beta_\kappa^\inv \bigl( \A_\na\Lambda_{\alpha_\kappa} \bigr)
$$
which must be a strict containment, and by iteration we have an infinite sequence of strict containments
$$\A_\na\Lambda_{\alpha_\kappa} \sqsupset \beta_\kappa^\inv \bigl( \A_\na\Lambda_{\alpha_\kappa} \bigr) \sqsupset \beta_\kappa^{-1} \bigl(\A_\na\Lambda_{\alpha_\kappa}  \bigr) \sqsupset \cdots
$$
But this contradicts Proposition~3.2 of \SubgroupsOne, which says any decreasing sequence of vertex groups systems is eventually constant.

Restoring the ``$i$'' subscript, and so $a=a_i$ and $c=c_i$, now define $a_{i+1} = x = c_i^m a_i^n$, so $\xi_\kappa = \pi_\kappa(a_{i+1})$ and $\Lambda^\pm_{\alpha_{i+1,\kappa}} = \Lambda^\pm_{\xi_{\kappa}}$. This completes the inductive construction of $a_{i+1}$ and its $\M_0$-consistent assignment of lamination pairs $\bigl(\Lambda^\pm_{\alpha_{i+1,\kappa}}\bigr)$. We have a containment $\A_\na\Lambda^\pm_{\alpha_{i,\kappa}} \sqsupset \A_\na\Lambda^\pm_{i+1,\alpha_\kappa}$ which is strict if $\A_\na\Lambda^\pm_{\alpha_{i+1,\kappa}}$ is not $\Gamma$-invariant. 

Consider the whole sequence of containments 
$$\A_\na\Lambda^\pm_{\alpha_{1,\kappa}} \sqsupset \A_\na\Lambda^\pm_{\alpha_{2,\kappa}} \sqsupset \A_\na\Lambda^\pm_{\alpha_{3,\kappa}} \sqsupset \cdots 
$$
Another application of Proposition~3.2 of \SubgroupsOne\ shows this sequence to be eventually constant, and again applying $(*)$ it follows that $\A_\na\Lambda^\pm_{\alpha_{i,\kappa}}$ is eventually $\Gamma$-invariant. It follows that there exists $I \ge 0$ independent of $\kappa$ such that for $i \ge I$ each $\A_\na\Lambda^\pm_{\alpha_{i,\kappa}}$ is $\Gamma$-invariant and is independent of~$i$.

\medskip

We now complete the proof of Proposition~\ref{multiple filling}. 

Conclusion~\pref{no mixed cases} of Proposition~\ref{multiple filling} follows by observing that if $t_2 > t_1$ then, by Lemma~\ref{invariant stabilizer} combined with $\Gamma$-invariance of $\A_\na\Lambda^\pm_{\alpha_{i,\kappa}}$ for each $i \ge I$ and each \hbox{$t_1 < \kappa \le t_2$,} the entire group $\Gamma$ is geometric above $\A$, from which it follows that $t_1=0$. 

Conclusion \pref{good for all kappa} will be proved in two cases. 

\textbf{Case 1: $t_2 > t_1=0$.} As just shown, in this case $\Gamma$ is geometric above~$\A$ and $\A_\na\Lambda^\pm_{\alpha_{i,\kappa}}$ is $\Gamma$-invariant for all $i \ge I$ and all $1 \le \kappa \le t_2$. Applying Lemma~\ref{invariant stabilizer}, it then follows that $\alpha_{i,\kappa}$ is irreducible rel $\A$ and $\Lambda^\pm_{\alpha_{i,\kappa}}$ has $\Gamma$-invariant free factor support and so is filling by property~\prefH{item:fills}. This proves \pref{good for all kappa} and so completes the proof of Proposition~\ref{multiple filling} in Case~1.
  
\textbf{Case 2:} $t_2=t_1$. In this case for each $1 \le \kappa \le t_2$ we know that $\Lambda^\pm_{\alpha_{i,\kappa}}$ is non-geometric for all $i$, and hence $\A_\na\Lambda^\pm_{\alpha_{i,\kappa}}$ is a proper free factor system.  Combining this with the fact that $\A \sqsubset \A_\na\Lambda^\pm_{\alpha_{i,\kappa}}$, that $\Gamma$ is irreducible rel $\A$, and that $\A_\na\Lambda^\pm_{\alpha_{i,\kappa}}$ is $\Gamma$-invariant for $i \ge I$, it follows that $\A_\na\Lambda^\pm_{\alpha_{i,\kappa}} = \A$ for $i \ge I$.  

Considering now each $1 \le \kappa \le t_2$ separately, after imposing stricter threshold constants we wish to prove by induction on $j \ge I$ that:
\begin{itemize}
\item[$(\#)$] We have containments $\displaystyle \F_\supp(\Lambda^\pm_{\alpha_{I,\kappa}}) \sqsubset \F_\supp(\Lambda^\pm_{\alpha_{I+1,\kappa}}) \sqsubset \cdots \sqsubset \F_\supp(\Lambda^\pm_{\alpha_{j,\kappa}})$
\item[$(\#\!\#)$] For $I \le i < j$ the containment $\F_\supp(\Lambda^\pm_{\alpha_{i,\kappa}}) \sqsubset \F_\supp(\Lambda^\pm_{\alpha_{i+1,\kappa}})$ is proper if and only if $\F_\supp(\Lambda^\pm_{\alpha_{i,\kappa}})$ is not $\Gamma$-invariant. 
\end{itemize}
Assuming that this has been proved for a particular $j$, and returning temporarily to the earlier notation where the subscript $j$ has been dropped, we shall apply the ``Inductive Step of Proposition 2.4'' from \SubgroupsFour, the hypotheses of which are already known to be true: the pair $\Lambda^\pm_{\alpha_{\kappa}}$ is nongeometric and its nonattracting subgroup system equals $\A$; by Lemma~\ref{simultaneous conjugation}~\pref{conjugator2} we have $\{\Lambda^-_{\alpha_\kappa},\Lambda^+_{\alpha_\kappa}\} \intersect \{\Lambda^-_{\gamma_\kappa},\Lambda^+_{\gamma_\kappa}\} = \{\Lambda^-_{\alpha_\kappa},\Lambda^+_{\alpha_\kappa}\} \intersect  \{\beta_\kappa(\Lambda^-_{\alpha_\kappa}), \beta_\kappa(\Lambda^+_{\alpha_\kappa})\}  = \emptyset$; and there are generic leaves of $\Lambda^\pm_{\alpha_\kappa}$ that are fixed by $\alpha_\kappa$ with fixed orientations. From the conclusions of the ``Inductive Step of Proposition 2.4'' \SubgroupsFour, we obtain a threshold constant $M^2_\kappa$ such that if $m,n \ge M^2_\kappa$ then $\F_\supp(\Lambda_{\alpha_\kappa}) \sqsubset \F_\supp(\Lambda_{\xi_\kappa})$ with proper inclusion if and only if $\F_\supp(\Lambda_{\alpha_\kappa})$ is not $\Gamma$-invariant. 

Letting $M^2 = \max_{1 \le \kappa \le t_2} M^2_\kappa$, and requiring that $m,n \ge M^2$, for all $1 \le \kappa \le t_2$ we have completed the inductive verification of $(\#)$ and $(\#\!\#)$ for all~$j \ge I$.

Since the length of a proper chain of free factor systems of $F_r$ is uniformly bounded, there exists $j \ge I$ such that $(\#)$ and $(\#\!\#)$ are satisfied and such that the free factor support of $\Lambda^\pm_{\alpha_{j,\kappa}}$ is $\Gamma$-invariant, and so by \prefH{item:fills} the lamination pair $\Lambda^\pm_{\alpha_{j,\kappa}}$ fills~$F_r$. Since $\A_\na\Lambda^\pm_{\alpha_{j,\kappa}}$, it follows that $\alpha_{j,\kappa}$ is irreducible rel~$\A$. This being true for each $1 \le \kappa \le t_2$, we have proved conclusion~(1) of Proposition~\ref{multiple filling} in Case~2, and so the proof of the proposition is complete.
\end{proof}

\paragraph{Remark: On avoiding the Global WWPD Theorem.} For purposes of proving the $H^2_b$ alternative of a general finitely generated subgroup $G \subgroup \Out(F_n)$, our methods produce an appropriate hyperbolic action of some finite index normal subgroup of $G$ to which the Global WWPD Theorem can be applied. However, both the statement and proof of the Global WWPD Theorem \cite{HandelMosher:WWPD} are quite intricate, and one might hope to avoid that theorem by producing an appropriate hyperbolic action of $G$ itself. Although this seems out of reach in general, in the case that $G \subgroup \IAThree$ it is sometimes possible to avoid the Global WWPD Theorem and instead to apply the following much less intricate theorem:

\begin{theorem}[\protect{\cite[Theorem 2.10]{HandelMosher:WWPD}}]
If $\Gamma \act X$ is a nonelementary group action on a hyperbolic space, and if this action has a WWPD element, then $H^2_b(G;\reals)$ contains a copy of $\ell^1$. 
\end{theorem}

For example, if $G \subgroup \IAThree$ is an infinite lamination subgroup with virtually abelian restrictions then Theorem~B guarantees that the action $G \act \FS(F_n)$ satisfies the hypotheses of the previous theorem, and hence $H^2_b(G;\reals)$ contains a copy of $\ell^1$. On the other hand, when $G \subgroup \IAThree$ is a finite lamination subgroup with virtually abelian restrictions, Theorem~C still generally requires one to pass to a further finite index normal subgroup of $G$, in which case we know no other option than to fall back on the Global WWPD Theorem.

\section{Well functions and weak tiling functions}
\label{SectionWells}
In this section and the next we are concerned with the proof of the WWPD Construction Theorem. For this section we assume that $n \ge 3$ and we shall pursue a further study of the ``well functions'' that were defined in \cite[Section~4.4]{\FSLoxTag} in the setting of outer automorphisms $\phi$ possessing a filling lamination pair $\Lambda^\pm_\phi$, and that generalize the well functions defined in \cite{AlgomKfir:Contracting} in the setting of fully irreducible outer automorphisms. 

The idea behind a well function is that for any conjugacy class $c$ that is weakly attracted to $\Lambda^\pm_\phi$ by iteration of $\phi^{\pm 1}$, there is an iterate $\phi^{w(c)}$ with coarsely well-defined exponent $w(c)$ such that the conjugacy class $\phi^{w(c)}$ is not a good approximation of either $\Lambda^-_\phi$ or $\Lambda^+_\phi$. This contrasts with being outside of the well: for exponents $l >\!\!> w(c)$ the class $\phi^{l}(c)$ is a good approximation of $\Lambda^+_\phi$ but a bad approximation of $\Lambda^-_\phi$; whereas for exponents $l <\!\!< w(c)$ the class $\phi^l(c)$ is a good approximation of $\Lambda^-_\phi$ but a bad approximation of $\Lambda^+_\phi$. The vicinity of the integer $w(c)$ within the whole of the integers $\Z$ therefore defines a ``well'' in which approximations to $\Lambda^+_\phi$ and $\Lambda^-_\phi$ are simultaneously bad.

Our main result in this section is Proposition~\ref{PropKTiles}, which allows one to coarsely evaluate the well function $w(c)$ using a \ct\ representative $f \from G \to G$ of~$\phi$, by studying how the circuit in $G$ representing $c$ can be ``weakly tiled'' using natural collections of paths in~$G$ (see Definition~\ref{DefWeakTilings}). As an application we prove Corollaries~\ref{tilings from tiles} and~\ref{bounded intersection} which describe strong combinatorial regularity properties of weak tilings. These regularity properties of weak tilings are what we shall use to verify WWPD in Section~\ref{SectionTheoremEProof}, in lieu of measure theoretic regularity properties of currents used to verify WPD in the arguments of \BeFuji\ and \BeFeighn.

\subsection{Well functions on intermediate conjugacy classes}
\label{SectionConjClassWell}
We review well functions and their properties from \cite[Lemma 4.14]{\FSLoxTag}, although see Remark~\ref{RemarkWellFunctions} to understand some differences of formulation.

\begin{definition}
\label{DefIntermediate}
Given a free factor system~$\F$ and a conjugacy class $c$ of $F_n$, we say that $c$ is \emph{filling relative to $\F$} if there is no proper free factor system that carries both $c$ and $\F$. If $c$ is not carried by $\F$ nor filling relative to $\F$ then $c$ is said to be \emph{intermediate} (relative to~$\F$). 
\end{definition}
We note some properties relating this definition to an outer automorphism $\phi \in \Out(F_n)$ that preserves~$\F$, the first of which is obvious:
\begin{itemize}
\item A conjugacy class $c$ is intermediate rel~$\F$ if and only if $\phi(c)$ is intermediate rel~$\F$. 
\item If $\phi$ is fully irreducible rel~$\F$ and if $\Lambda^\pm$ is a lamination pair for $\phi$ that is not carried by~$\F$ and hence fills rel~$\F$, then a conjugacy class $c$ is intermediate rel~$\F$ if and only if $c$ is weakly attracted to $\Lambda^\pm$ and is not filling relative to $\F$. 
\end{itemize}
To see why the second item is true, recall Theorem F of \SubgroupsThree\ which says that $c$ is weakly attracted to $\Lambda^+$ under iteration by $\phi$ if and only if $c$ is weakly attracted to $\Lambda^-$ under iteration by $\phi^{-1}$ if and only if $c$ is not carried by the nonattracting subgroup system $\A_\na(\Lambda^\pm)$. If $\Lambda^\pm$ is non-geometric then $\F = \A_\na(\Lambda^\pm)$ (see Definition 1.2 ``Remark: The case of a top stratum'' in \SubgroupsThree) and so ``not carried by~$\F$'' is equivalent to ``weakly attracted to $\Lambda^\pm$''. If $\Lambda^\pm$ is  geometric then there are conjugacy classes that are not carried by $\F$ and are not  weakly attracted to  $\Lambda^\pm$ but each of these fills relative to $\F$; each such conjugacy class is represented by a multiple of the peripheral curve in a geometric model (\SubgroupsOne, Proposition 2.18).

For most of this section we focus on the following objects whose notations we fix:

\medskip\noindent\textbf{Notation A:}

\noindent$\bullet$ \, $\F$ is a proper free factor system in $F_n$, possibly trivial. 

\noindent$\bullet$ \, $\phi,\phi^\inv \in \Out(F_n)$ are rotationless, and they fix~$\F$ and are irreducible relative to~$\F$.

\noindent
$\bullet$ \, $\Lambda^\pm \in \L^\pm(\phi)$ is a lamination pair that fills~$F_n$. 

\smallskip
\noindent
With this notation we also have the following property:

\smallskip
\noindent
$\bullet$ \, $\F$ has co-edge number~$\ge 2$ (Fact~\ref{FactANA}~\pref{FactAnaCoEdgeTwo}).

\begin{definition}
Consider the space of lines $\B$ with its weak topology. For any attracting neighborhood $U \subset \B$ of a generic leaf of $\Lambda^+$, and any conjugacy class $c$ that is weakly attracted to $\Lambda^+$ under iteration by $\phi$, there exists a maximal integer $w_{\phi,U}(c)$ such that $c \in \phi^{w_{\phi,U}(c)}(U)$ or equivalently $\phi^{-w_{\phi,U}(c)}(c) \in U$. We refer to $w_{\phi,U}$ as the \emph{well function of $\phi$ with respect to $U$.} It is immediate from the definition  that 
$$w_{\phi,U}(\phi^m(c)) = w_{\phi,U}(c) + m 
$$
for all $m$. 
\end{definition}

The two key properties of well functions are as follows:
\begin{description}
\item[Coarse well-definedness of wells:] For any two $U,V \subset \B$ attracting neighborhoods of a generic leaf of $\Lambda^+$ there exists $K \ge 0$ so that $\phi_\#^K(V) \subset U$ and $\phi^K_\#(U)  \subset V$. It follows that 
$$\abs{w_{\phi,U}(c) - w_{\phi,V}(c)} \le K
$$
for all $c$ that are weakly attracted to $\Lambda^\pm$.  In other words the function $w_{\phi,U}(\cdot)$ is coarsely well-defined independent of $U$. 
\item[Coarse additive inverse property of wells:] For any attracting neighborhoods $U^\pm$ of generic leaves of $\Lambda^\pm$, respectively, there exists $L \ge 0$ so that for all intermediate conjugacy classes $c$ we have 
$$\abs{w_{\phi,U^+}(c) + w_{\phi^{-1},U^-}(c)} \le L
$$  
In other words, the well functions $w_{\phi,U^+}(\cdot)$ and $w_{\phi^\inv,U^-}(\cdot)$ are coarse additive inverses of each other.
\end{description}
The coarse well-definedness property is obvious. The coarse additive inverse property is proved in \cite[Lemma 4.14(1)]{\FSLoxTag} (which is in turn based on Proposition 3.1 of \SubgroupsFour) for a particular choice of $U^\pm$, and so it holds for all choices. Note that coarse well-definedness does not require the hypothesis of ``intermediate''; that hypothesis need only be brought in for results, like the coarse additive inverse property, that ultimately depend on Proposition 3.1 of \SubgroupsFour.

\smallskip

\emph{Henceforth,} when attracting neighborhoods of generic leaves of $\Lambda^+$ and of $\Lambda^-$ are chosen, implicitly or explicitly, we write the corresponding well functions of $\phi$~and~$\phi^\inv$ as $w_\phi(\cdot)$ and $w_{\phi^\inv}(\cdot)$, suppressing the dependence on the choice of attracting neighborhoods.

\smallskip

\begin{remark}\label{RemarkWellFunctions}
Our presentation here of well functions differs from the presentation in \FSLox\ version 1 in a few regards. One difference is that well functions were defined there only for particular choices of $U^-,U^+$; the coarse well-definedness property lets us use any choices. Another difference is that $w_\phi$ as defined here equals the additive inverse of $w_{\phi^\inv}$ as it was defined in Lemma 4.14 of \FSLox\ version 1; by the coarse additive inverse property, this changes the definition of $w_\phi$ by a constant depending only on the choice of $U^-,U^+$.
\end{remark}

\subsection{Weak tiling functions on intermediate conjugacy classes} 

Following up on Notation~A, for much of this section we shall also focus on these additional objects whose notations we also fix:

\medskip\noindent\textbf{Notation B:} 

\noindent$\bullet$ \, $f \from G \to G$ is a \ct\ representative of $\phi$ with core filtration element $G_r$ representing the free factor system~$\F$. 

\noindent$\bullet$ \, The attracting lamination $\Lambda^+$ corresponds to the top stratum $H_s$ of~$G$, an \eg\ stratum.

\noindent$\bullet$ \, $H^z_s$ is the union of $H_s$ with the zero strata, if any, that it envelops, those strata being the contractible components of $G_{s-1}$ \cite[Definition 2.18]{\recognitionTag}.

\noindent$\bullet$ \, $G_u = G - H^z_s$ is the maximal filtration element that deformation retracts to~$G_r$.

\noindent$\bullet$ \, $\rho_s$ is an indivisible Nielsen path of height $s$ \emph{if one exists} (and see items~\pref{ItemCTiNP} and~\pref{ItemCTGeomNielsen} under the heading ``\eg\ properties of \cts'' in Section~\ref{SectionBasicNotions}).

\smallskip\noindent
With this notation we have the following property:

\smallskip\noindent
$\bullet$ \, If $\rho_s$ exists and is closed then the conjugacy class $c$ represented by $\rho_s$ fills relative to~$\F$ (\SubgroupsOne\ Lemma 2.5) and so $c$ is not intermediate.

\medskip

Our immediate goal, formulated in Proposition~\ref{PropKTiles}, is to give quantitative bounds on well functions, expressed in terms of the \ct\ chosen in Notation~B, and more specifically in terms of tiles as defined and studied in \cite[Section 3]{\BookOneTag}. 

Earlier we recalled that a \emph{$k$-tile} of the stratum $H_s$ is a path of the form $f^k_\#(E)$ where $E$ is an edge of $H_s$ and $k \ge 0$. Tiles satisfy a \emph{self-similarity property} saying that for any integers $l > k > 0$, if the difference $l-k$ is sufficiently large then every $l$-tile contains every $k$-tile as a subpath \cite[Lemma 3.1.8 (3)]{\BookOneTag}. 

Given a path or circuit $\sigma$ in $G$, a \emph{$k$-tiling} of $\sigma$ is a splitting of $\sigma$ each term of which is either a $k$-tile or a subpath of $G_{s-1}$ with endpoints on $H_s$. Every generic leaf $\gamma^+$ of $\Lambda^+$ has a $k$-tiling for all $k \ge 1$ \cite[Lemma 3.1.10 (3)]{\BookOneTag}. 

Given a line $\gamma$ in $G$, an \emph{exhaustion by tiles} is an increasing family of finite subpaths $\gamma_1 \subset \gamma_2 \subset \cdots$ whose union is all of $\gamma$ such that each $\gamma_i$ is a $k_i$-tile for some sequence $k_1 < k_2 < \cdots$. Every generic leaf $\gamma^+$ of $\Lambda^+$ has an exhaustion by tiles \cite[Lemma 3.1.10 (4)]{\BookOneTag}. 

\smallskip

\textbf{Remark:} Given a generic leaf $\gamma^+$ with an exhaustion by tiles $\gamma_1 \subset \gamma_2 \subset \cdots$, those tiles define a neighborhood basis $V_1 \supset V_2 \supset \cdots$ of $\gamma^+$ where $V_k \subset \B$ is the set of lines having $\gamma_k$ as a subpath. Once an attracting neighborhood $U^+ \subset \B$ of $\ell$ is chosen, the sequence $\phi^i(U^+)$ also forms neighborhood basis of~$\ell$. It follows that for any sequence of conjugacy classes $c_i$, the sequence of values of a well function $w_\phi(c_i)$ correlates with the value of the maximum integer $k_i \ge 0$ such that $c_i$ contains a $k_i$ tile. Namely, $w_\phi(c_i) \to +\infinity$ if and only if $k_i \to +\infinity$. This gives the first hint to the quasi-comparability of well functions and of the ``weak tiling functions'' to be introduced in Definition~\ref{DefWeakTilingFunctions} below.

%
%
%

We will need a less restrictive kind of tiling, allowing terms that are Nielsen paths of height~$s$.

\begin{definition}[Weak tilings]
\label{DefWeakTilings}
A \emph{weak $k$-tiling} of a path or circuit $\sigma\subset G$ is a splitting of $\sigma$ each term of which is either a $k$-tile, or an \iNp\ of height $s$, or a maximal subpath of $\sigma$ in~$G_{s-1}$. Weak $k$-tilings are most useful when $k$ is large in the following sense: the \emph{weak tiling threshold constant} $k_0 \ge 0$ is defined to be the minimum integer such that for all $k \ge k_0$ the endpoints of each $k$-tile are periodic.
\end{definition}

\noindent
\textbf{Examples of weak tilings:} The following lemma gives general methods for constructing weak tilings. Lemma~\ref{stabilizes} gives a another method, with tighter quantitative control, for producing circuits with weak $0$-tilings; that lemma plays a key role in the proof of Proposition~\ref{PropKTiles}.

\begin{lemma} \label{pullback} Following Notations A~and~B, we have:
\begin{enumerate}   
\item \label{ItemTilingExistence}
For each circuit or finite path $\sigma \subset G$ with endpoints, if any, at vertices there exists $m_\sigma \ge 0$ such that $f^{m_\sigma}_\#(\sigma)$ has a weak $0$-tiling.
\item \label{ItemTilingPushforward}
For all $k\ge 0$, if $\sigma = \mu_1 \cdot \ldots \cdot \mu_I$ is a weak $k$-tiling then $f_\#(\sigma) = f_\#(\mu_1) \cdot \ldots \cdot f_\#(\mu_I)$ is a weak $(k+1)$-tiling.
\item \label{ItemTilingPullback}
Suppose that  $f_\#(\sigma) =\nu_1 \cdot \ldots \cdot \nu_I$ is a weak $(k+1)$-tiling for some $k \ge k_0$ and that any endpoint of $f_\#(\sigma)$ that is contained in $G_{s-1}$ is fixed by $f$ (and hence is contained in $G_u$). Then there is weak $k$-tiling $\sigma = \mu_1 \cdot \ldots \cdot \mu_I$ such that $f_\#(\mu_i) = \nu_i$ for all $i$.
\end{enumerate}
\end{lemma}   

\begin{proof} Item \pref{ItemTilingExistence} follows from \cite[Lemma~4.2.6]{\BookOneTag}. Item \pref{ItemTilingPushforward} follows from the fact that if $\mu_i$ is a path in $G_{s-1}$, an \iNp\ of height $s$, or a $k$-tile, then $f_\#(\mu_i)$ is a path in $G_{s-1}$, an \iNp\ of height $s$ or a $(k+1)$-tile respectively. 

For \pref{ItemTilingPullback} we define $\mu_i$ given $\nu_i$ as follows. If $\nu_i$ is an \iNp\ of height~$s$, let $\mu_i = \nu_i$. Suppose next that $\nu_i$ is a maximal subpath of~$G_{s-1}$. An endpoint of $\nu_i$ is either an endpoint of $f_\#(\sigma)$ or is contained in $H_s$. In the former case it is fixed by hypothesis and in the latter case it is a principal vertex by \cite[Remark~ 4.9]{\recognitionTag} and hence fixed because $f$ is rotationless. Since the components of $G_{s-1} - G_u$ are zero strata, and since zero strata contain no fixed points, it follows that $\nu_i \subset G_u$ and that there is a path $\mu_i\subset G_u$ with the same endpoints as $\nu_i$ and satisfying $f_\#(\mu_i) = \nu_i$. Suppose finally that $\nu_i$ is a $(k+1)$-tile with endpoints say $y_1$ and $y_2$. For $i=1,2$, let $x_i$ be the unique vertex in the $f$-orbit of $y_i$ such that $f(x_i) = y_i$. There exists a $k$-tile $\mu_i$ such that $f_\#(\mu_i) = \nu_i$. Since $k \ge k_0$, the endpoints of $\mu_i$ are periodic and map to $y_1$ and $y_2$ and so must be $x_1$ and $x_2$. To complete the proof we need only verify that $\mu_1 \cdot \ldots \cdot \mu_I$ is a splitting and that adjacent terms are not both contained in $G_{s-1}$. If $\mu_i$ is an \iNp\ of height $s$ or a $k$-tile then $f$ maps the initial and terminal directions of $\mu_i$ to the initial and terminal directions of $f_\#(\mu_i)$, all of which are contained in $H_s$. If $\mu_i \subset G_u$ then the initial and terminal directions of $f_\#(\mu_i)$ are contained in $G_u$. Since $\nu_1 \cdot \ldots \cdot \nu_m$ is a splitting and since adjacent terms are not both contained in $G_u$, the same is true for $\mu_1 \cdot \ldots \cdot \mu_I$.
 \end{proof}

\begin{definition}[Weak tiling functions]
\label{DefWeakTilingFunctions}   
Following notations A and~B, we can now define the integer valued \emph{weak tiling function} $\tau_f$ on the set of intermediate conjugacy classes. Consider an intermediate conjugacy class $c$ that is represented by a circuit $\sigma \subset G$. If $c$ is such that $\sigma$ has a weak $k_0$-tiling, then choose the maximal $k \ge k_0$ such that $\sigma$ has a weak $k$-tiling and define $\tau_f(c) = k-k_0$. By Lemma~\ref{pullback} we can extend this to arbitrary $c$, by defining $\tau_f(c)$ to be the maximal integer such that $f_\#^{-\tau_f(c)}(\sigma)$ has a weak $k_0$-tiling.
\end{definition}

\subsection{Coarse equivalence of well functions and \hfill\break weak tiling functions.}
\label{SectionWellAndTilingFunctions}
The following proposition is the main result of this section.  It states that the weak tiling function determined by $\fG$ is coarsely equivalent to any well function associated to $\Lambda^+$. As corollaries we will state some strong regularity properties of weak $k$-tilings (see Corollaries~\ref{tilings from tiles} and \ref{bounded intersection}).
   
\begin{prop} \label{PropKTiles} Following notations A and~B, let $w_\phi=w_{\phi,U}$ be a well function for $\Lambda^+$ defined with respect to some attracting neighborhood $U$ of a generic leaf of $\Lambda^+$, let $k_0$ be the weak tiling threshold for $f \from G \to G$ (Definition~\ref{DefWeakTilings}), and let $\tau_f$ be the associated weak tiling function (Definition~\ref{DefWeakTilingFunctions}). Then there is a constant $N$ so that for all intermediate $c$ we have
$$\abs{w_\phi(c) - \tau_f(c)} \le N
$$ 
\end{prop} 

The proof of the proposition is given after the statement and proof of the following lemma and a corollary thereof. This lemma is a relative version of \cite[Proposition 3.8]{AlgomKfir:Contracting}.

\begin{lemma}  \label{Whitehead} Suppose that the lamination pair $\Lambda^\pm$ is geometric and consider $\rho_s$, the \iNp\ of height~$s$. There exists an integer $M \ge 1$ so that if $\sigma \subset G$ is a circuit representing an intermediate conjugacy class then $\sigma$ does not   contain a   subpath of the form $\rho_s^M$ or $\rho_s^{-M}$.  
\end{lemma} 

\begin{proof} Let $c_0$ be the   conjugacy class of the circuit in $G$ determined by the closed path~$\rho_s$. By Proposition 2.18 of \SubgroupsOne, the joint free factor support of $c_0$ and $\F$ is equal to the joint free factor support of $\Lambda^\pm$ and $\F$, and hence is equal to $\{[F_n]\}$.  We may therefore choose  conjugacy classes $c_1,\ldots,c_k$ carried by $\F$ whose free factor support equals $\F$, and therefore $\{c_0,c_1,\ldots,c_k\} $ fills. By Whitehead's theorem \cite{Whitehead:CertainSets} there is a marked rose $R$ such that the Whitehead graph for $\{c_0,\ldots,c_k\}$ is connected and has no cut vertices---this is the graph with a vertex $v_e$ for every directed edge $e$ in $R$ and an edge $v_e$---$v_{e'}$ whenever $\bar e e'$ is a subpath of the circuit in $R$ representing one of $c_0,\ldots,c_k$ or their inverses. Let $\tau \subset R$ be a closed path determining the circuit in $R$ that corresponds to $c_0$. It follows that if $\sigma' \subset R$ is any circuit that contains $\tau^2$ as a subpath then the Whitehead graph for $\{[\sigma'], c_1,\ldots,c_k\}$ contains the Whitehead graph for $\{c_0,\ldots,c_k\}$ and so is also connected and has no cut points. A second application of Whitehead's theorem shows that $\{[\sigma'], c_1,\ldots,c_k\}$ is filling. To complete the proof, apply the bounded cancellation lemma and choose $M$ so large that if $\sigma \subset G$ contains a subpath of the form $\rho_s^M$ then the realization of $[\sigma]$ in $R$ contains a subpath of the form $\tau^2$.
\end{proof}

\begin{corollary}\label{endpoint condition} For any circuit $\sigma \subset G$ representing an intermediate conjugacy class~$c$ the following hold:
\begin{enumerate} \item \label{at least one tile} 
If $\sigma$ has a weak $k$-tiling then some term of that tiling is a $k$-tile. 
\item \label{k tile separation} For all $k\ge 0$ there exists $L \ge 1$ so that if $\sigma$ has a weak $k$-tiling and if $\sigma'$ is a subpath of $\sigma$ that crosses  $L$ edges of $H_s$ then $\sigma'$ contains a $k$-tile.
\end{enumerate}  
\end{corollary}

\begin{proof} We prove~\pref{at least one tile} by contradiction, assuming that no term (of~the weak $k$-tiling of~$\sigma$) is a $k$-tile. If $\rho_s$ does not exist, or if it exists and no term is a copy of $\rho_s$ or $\bar\rho_s$, then $\sigma$ is contained in $G_r$ and so $c$ is supported by $\F = [G_r]$, contradicting that $c$ is intermediate. We may therefore assume that $\rho_s$ exists and that at least one term is a copy of $\rho_s$ or $\bar\rho_s$. At least one endpoint $x$ of $\rho_s$ is contained in the interior of $H_s$. If $\rho_s$ is not closed then $\sigma$ has a term incident to $x$ that is not a copy of $\rho_s$ or $\bar\rho_s$, but that term must contain an edge of $H_s$ and so can only be a $k$-tile, contradicting the assumption. If $\rho_s$ is closed then by a similar argument $\sigma$ can only be an iterate of $\rho_s$ or $\bar\rho_s$, and so $c$ fills relative to~$\F$ (see Notation~B), also contradicting that $c$ is intermediate.

For item \pref{k tile separation} suppose that $\sigma = \sigma_1\ldots \sigma_m$ is a weak $k$-tiling. Let $L_0$ be the maximal number of $H_s$-edges in a $k$-tile. If $\rho_S$ exists let $L_s$ be the number of $H_s$ edges in $\rho_s$, otherwise $L_s=0$. If $\Lambda^\pm$ is nongeometric and $\rho_s$ does not exist then the conclusion is evident with $L=2L_0$. If $\Lambda^\pm$ is non-geometric and $\rho_s$ exists then $\rho_s$ has distinct endpoints, so if $2 \le i \le m-1$ and if $\sigma_i = \rho_s^{\pm}$ then either $\sigma_{i-1}$ or $\sigma_{i+1}$ is $k$-tile; the conclusion holds with $L=L_s+2L_0$. If $\Lambda^\pm$ is geometric then $\rho_s$ exists and is closed and, letting $M$ be the constant of Lemma~\ref{Whitehead}, at most $M$ consecutive terms of any cyclic permutation of the given weak $k$-tiling of $\sigma$ can be copies of $\rho_s$ or $\bar\rho_s$. In all cases the conclusion follows with $L =  \max\{2,M\} \cdot L_s + 2L_0$.
\end{proof}

\begin{proof}[Proof of Proposition~\ref{PropKTiles}]   We simplify notation slightly by writing $w$ and $\tau$ for the functions $w_\phi$ and $\tau_f$, suppressing the dependence on $\phi$ and $f$.  Let $\Lambda^- \in \L(\phi^\inv)$ denote the dual repelling lamination of $\Lambda^+$.

Since $w(\phi^m(c)) = w(c) + m$ and $\tau(\phi^m(c)) = \tau(c) + m$ for all $m \in \Z$ it suffices to show that there are uniform upper and lower bounds to $\tau(c)$ as $c$ varies over all intermediate conjugacy classes with $w(c) = 0$. 

An upper bound is easy to find, using properties of $\Lambda^+$. Choose a generic leaf $\gamma^+$ of $\Lambda^+$ realized in~$G$. Since $U$ is a weak neighborhood of $\gamma^+$ there is subpath $\delta \subset G$ of $\gamma^+$ such that every line in $G$ that contains $\delta$ as a subpath is contained in $U$. By \cite[Lemma 3.1.1~(4)]{\BookOneTag} there exists $k_\delta > 0$ and a $k_\delta$-tile $\kappa$ that contains $\delta$ as a subpath. By \cite[Lemma 3.1.8~(4)]{\BookOneTag} there exists $l > \max\{k_0,k_\delta\}$ such that  every $l$-tile contains every $k_\delta$-tile as a subpath. It follows that every $l$-tile contains $\delta$ as a subpath. Corollary~\ref{endpoint condition}~\pref{at least one tile} implies that any circuit that represents an intermediate $c$ and that has a weak $l$-tiling must contain an $l$-tile, and therefore that circuit determine a bi-infinite line contained in $U$. To put it another way, if $\tau(c) \ge l - k_0$ then $w(c) \ge 0$. It follows that $w(c) < 0 \Longrightarrow \tau(c) < l - k_0$ and that $w(c) \le 0 \Longrightarrow \tau(c) \le l-k_0$, and so $l-k_0$ is the desired upper bound to $\tau(c)$ when $w(c)=0$.

It remains to find the lower bound, using properties of $\Lambda^-$. We show that there exists $d > 0$ so that if $w(c) = 0$ and if $\sigma \subset G$ is the circuit representing $c$ then $f_\#^d(\sigma)$ has a weak $0$-tiling, for once that is done it follows that $-(d+k_0)$ is a lower bound for $\tau(c)$. Let $\zeta$ be the line in $G$ that wraps bi-infinitely around~$\sigma$.

\break
\begin{description}
\item[Claim $(*)$:] There exists $L > 0$, independent of $c$, so that no common subpath of $\zeta$ and a generic leaf of $\Lambda^-$ in $G$ crosses $L$ edges of $H_{s}$. 
\end{description}
The conclusion of Claim $(*)$ implies that no common subpath of $\sigma$ and a generic leaf of $\Lambda^-$ crosses $L$ edges of $H_s$, which is equivalent to the inequality $\lsm(\sigma) < L$ in the hypothesis of Lemma~\ref{stabilizes}, and so we may apply the conclusion of that lemma to obtain $d$, depending only on $L$ and so independent of~$c$, such that $f^d_\#(\sigma)$ has a weak $0$-tiling.
 
To prove Claim $(*)$, let $g \from G' \to G'$ be a \ct\ representing $\phi^{-1}$ with top stratum $H'_{s'}$ corresponding to $\Lambda^-$ and with core filtration element $G'_{r'} \subset G'_{s'-1}$ representing~$\F$. Let $\gamma^-$ be a generic leaf for $\Lambda^-$ realized in~$G'$. Let $\sigma' \subset G'$ be the circuit in $G'$ representing $c$, and let $\zeta'$ be the line in $G'$ wrapping bi-infinitely around $\sigma'$. We shall show:
\begin{description}
\item[Claim $(*')$:] There exists $L' > 0$, independent of $c$, so that no common subpath of $\zeta'$ and $\gamma^-$ in $G'$ crosses $L'$ edges of $H'_{s'}$.
\end{description}
This suffices to prove Claim $(*)$, for the following reasons. In $\wt G$ let $\wt H_s$ be the total lift of $H_s$, and let $\wt G_{s-1} = \wt G \setminus \wt H_s$ be the total lift of $G_{s-1}$; in $\wt G'$ define $\wt H'_{s'}$ and $\wt G'_{s'-1}$ similarly. Let $T$ be the simplicial $F_n$-tree obtained from the universal cover $\wt G$ by collapsing each component of $\wt G_{s-1}$ to a point, and let $T'$ be similarly obtained from $\wt G'$ and $\wt G'_{s'-1}$. Using that $[G'_{s'-1}]=[G'_{r'}] = \F = [G_r] = [G_{s-1}]$ it follows that there are equivariant quasi-isometries between $T$ and $T'$ in either direction, and that these are equivariant coarse inverses to each other \cite[Theorem 3.8]{GuirardelLevitt:DefSpaces}. For any lifts $\ti\zeta,\ti\gamma \in \wt\B$, using subscripts to represent their realizations as lines in $\wt G$ and $\wt G'$ and the projections to lines in $T$ and $T'$, it follows that the number of $\wt H_s$ edges of $\ti\zeta_G \intersect \ti\gamma_G$ equals the edge path length of $\ti\zeta_T \intersect \ti\gamma_{T}$, which is quasi-comparable to the edge path length of $\ti\zeta'_{T'} \intersect \ti\gamma^-_{T'}$, which equals the number of $\wt H'_s$ edges of $\ti\zeta_{G'} \intersect \ti\gamma_{G'}$. The existence of the uniform bound $L'$ in Claim~$(*')$ therefore implies the existence of the uniform bound~$L$ in Claim~$(*)$.

Choose an attracting neighborhood $U^- \subset \B$ of $\gamma^-$. Let $w' := w_{\phi^{-1}, U^-}$ be the corresponding well function. By the coarse additive inverse property of well functions we have a uniform upper bound $\abs{w'(c) + w(c)} < W$ independent of~$c$. The set $\phi^{-W}(U^-) \subset \B$ is evidently also an attracting neighborhood of $\gamma^-$. Using an exhaustion by tiles of $\gamma^-$, there exists an integer $l \ge 1$ and an $l$-tile $\delta$ such that every line having $\delta$ as a subpath is contained in $\phi^{-W}(U^-)$. Using the self-similarity property of tiles there exists an integer $m > l$ such that every $m$-tile has $\delta$ as a subpath. Since we are restricting to those $c$ with $w(c) = 0$, we obtain a uniform upper bound $w'(c) < W$. It follows that $c \not\in \phi^{-W}(U^-)$, and so $\zeta'$ does not contain $\delta$ as a subpath. We may therefore choose $L'$ to be twice the maximal number of $H'_{s'}$-edges in an $m$-tile: since $\gamma^-$ has an $m$-tiling, every subpath of $\gamma^-$ that contains $L'$ edges of $H'_{s'}$ contains an $m$-tile and so also contains $\delta$, and that subpath of $\gamma^-$ is therefore not a subpath of $\zeta'$. This completes the proof of Claim~$(*')$, of Claim~$(*)$, and of the lemma.
\end{proof}

\subsection{Regularity properties of weak tilings.} 
\label{SectionTileRegularity}

We continue to follow notations A~and~B. 

A general circuit $\sigma$ in $G$ can have highly irregular behavior with respect to tiles and weak tilings. For example, choose $\alpha$ to be a finite path in $G$ having at least two distinct illegal turns, and choose $j \ge 2$ so that each $k$-tile with $k \ge j$ has more $H_s$ edges than $\alpha$ has. It follows that if $k \ge j$ and if $\sigma$ has a weak $k$-tiling then $\sigma$ does not contain $\alpha$ as a subpath. On the other hand, one can easily construct $\sigma$ so as to contain $\alpha$ as a subpath and to also contain some $k$-tile with arbitrarily large~$k$. This kind of irregularity --- a large gap between the greatest $j$ for which $\sigma$ has a weak $j$-tiling and the greatest $k$ for which $\sigma$ contains a $k$-tile --- is ruled out by the following corollary for those circuits representing intermediate conjugacy classes. As an application of the corollary, the circuit $\sigma$ just constructed must fill rel~$\F$ if $k$ is sufficiently large.

\begin{corollary} \label{tilings from tiles}  Following notations A~and~B, and letting $k_0$ be as in Definition~\ref{DefWeakTilings}, there is a positive constant $M$ so that for all $k \ge  k_0+M$, every circuit $\sigma \subset G$ that realizes an intermediate conjugacy class and that contains a $k$-tile has a weak $k-M$~tiling.  
\end{corollary}

\begin{proof} If the corollary fails then there are sequences $k_i \ge k_0$ and $l_i \ge k_i$ such that $l_i - k_i \to +\infty$, and there are circuits $\sigma_i \subset G$   representing intermediate conjugacy classes, such that $\sigma_i$ contains an $l_i$-tile but has no weak $k_i$-tiling. Assuming this, we argue to a contradiction.

The main step in the proof is to reduce to the case that $k_i = k_0$ for all $i$.   More precisely, we claim that there exist a sequence $l'_i \to +\infty$ and circuits $\alpha_i \subset G$ representing  intermediate conjugacy classes such that $\alpha_i$ contains an $l'_i$-tile  but has no weak $k_0$-tiling. To prove the claim, decompose $\sigma_i = \nu_i \nu'_i$ as a concatenation of two subpaths where $\nu_i$ is an $l_i$-tile. Let $y^-_i$ and $y^+_i$ be the initial and terminal endpoints of $\nu_i$ respectively and let $x^-_i$ and $x^+_i$ be the unique periodic points satisfying $f^{k_i-k_0}(x^-_i) = y^-_i$ and $f^{k_i-k_0}(x^+_i) = y^+_i$.  There are unique paths $\mu_i$ connecting $x^-_i$ to $x^+_i$ and $\mu'_i$  connecting $x^+_i$ to $x^-_i$ such that $f_\#^{k_i-k_0}(\mu_i) = \nu_i$ and $f_\#^{k_i-k_0}(\mu'_i) = \nu'_i$.  As argued in the proof of Lemma~\ref{pullback}, $\mu_i$ is an $(l_i+k_0 - k_i)$-tile; it follows that $\mu_i$ has an $l$ tiling for all $l = 1 \,,\,  \ldots\, , \,l_i+k_0-k_i$  \cite[3.1.8]{\BookOneTag}.    The circuit $\alpha_i$ obtained by tightening $\mu_i \mu_i'$ satisfies $f_\#^{k_i-k_0}(\alpha_i) = \sigma_i$ and so has no weak $k_0$-tiling by Lemma~\ref{pullback} and by the assumption that $\sigma_i$ has no weak $k_i$-tiling.   Since $\mu_i$ is $s$-legal and the initial [resp. terminal] directions of $f_\#^{k_i-k_0}(\bar \mu_i) = \bar\nu_i$ and $f_\#^{k_i-k_0}(\mu'_i)=\nu'_i$ are distinct, there is a uniform upper bound to the number of $H_s$-edges in the maximal common initial [resp. terminal] subpath of $\bar \mu_i$ and $\mu_i'$ 
\cite[Lemma 4.2.2]{\BookOneTag}. Thus only finitely many edges of $H_s$ are cancelled when  $\mu_i \mu_i'$ is tightened to $\alpha_i$.  We may therefore choose $l_i''< l_i$ so that $l_i - l''_i$ is  independent of $i$ and so that at least one of the terms in the $(l_i''+k_0-k_i)$-tiling of the  $(l_i+k_0 - k_i)$-tile $\mu_i$ remains after cancellation and so occurs in $\alpha_i$.  This verifies the claim with $l'_i = l_i''+k_0-k_i$.   

Choose an attracting neighborhood $U$ associated to $\Lambda^+$ and let $w_\phi$ be the corresponding well function. From the claim we have $\tau_f[\alpha_i] < 0$.  Proposition~\ref {PropKTiles} therefore implies that $w_\phi[\alpha_i]$ has a uniform upper bound, and hence that there exists a positive integer $N$ so that for each $i$,   $\alpha_i \not \in f^N_\#(U)$.  This contradicts the fact that $\Lambda^+$ is a weak limit of the $\alpha_i$'s which follows from the assumption that $l'_i \to \infty$.
\end{proof}

\begin{corollary}   \label{bounded intersection}  
Following notations A and~B, let $w_\phi$ be a well function for $\Lambda^+$ defined with respect to some attracting neighborhood of $\Lambda^+$, and let $k_0$ be as in Definition~\ref{DefWeakTilings}. Then there exists  $K > 0$ so that for all $A \ge k_0$ there exist  $L>0$ so that if $\alpha, \beta$ are circuits in $G$    that  represent intermediate conjugacy classes   and that     have a common subpath that crosses at least $L$ edges of $H_s$  and if $w_\phi([\alpha]) \le A$ then $w_\phi([\beta]) \le A+K$.
\end{corollary}
 
\begin{proof}  By Proposition~\ref{PropKTiles} and Corollary~\ref{tilings from tiles} there exists an even integer $K $ so that so that the following are satisfied for all circuits $\alpha , \beta \subset G$  representing intermediate conjugacy classes and all $A \ge k_0$. 
\begin{enumerate}
\item If $w_\phi([\alpha]) \le A$ then $\alpha$ does not contain any $(A+K/2)$-tiles.
\item If $w_\phi([\beta]) >A+ K$ then $\beta$ has a weak $(A+K/2)$-tiling.
\end{enumerate}  
It therefore suffices to show that there exists   $L> 0$ so that if $\beta$ has a   weak $(A+K/2)$-tiling    then every subpath of $\beta$ that crosses at least $L$ edges of $H_s$ contains an  $(A+K/2)$ tile.  The existence of $L$ follows from  Lemma~\ref{endpoint condition}~\pref{k tile separation}. \end{proof}
 
\section{Proof of the WWPD Construction Theorem}
\label{SectionTheoremEProof}

After repeating the theorem for convenience, we shall then review some basic facts and notations regarding the free splitting complex~$\fscn$.

\begin{theoremE}
Suppose that $n \ge 3$, that $\cH$ is a subgroup of $\IAThree$ that preserves a (possibly empty) proper free factor system $\F$ and that $\phi \in \cH$ satisfies the following: 
\begin{description}
\item[Relative irreducibility:] $\phi$ is irreducible rel $\F$; 
\item[Trivial restrictions:] $\phi \restrict \Out(A)$ is trivial for each component $[A] \in \F$; 
\item[Filling lamination:] There exists a filling attracting lamination $\Lambda_\phi^+ \in \L(\phi)$, and hence $\phi$ acts loxodromically on $\fscn$ \FSLox.  
\end{description}
Then $\phi$ satisfies \wwpd\ with respect to the action of $\cH$ on $\fscn$.
\end{theoremE}

Henceforth we fix $\F$, $\phi$, and $\Lambda^+_\phi \in \L(\phi)$ as in the theorem and we let $\Lambda^-_\phi \in \L(\phi^\inv)$ be the dual repelling lamination of $\L(\phi)$, and so $\Lambda^-_\phi$ also fills. The existence of a filling lamination forces $\F$ to have co-edge number~$\ge 2$ (Fact~\ref{FactANA}~\pref{FactAnaCoEdgeTwo}).

\subsection{Free splitting complex and spine of relative outer space.}
\label{SectionFSCRelSpine}
We first set some notation regarding the free splitting complex $\FS(F_n)$. A \emph{free splitting} is a minimal, simplicial action of $F_n \act T$ on a simplicial tree $T$ having trivial edge stabilizers; usually we suppress the action from the notation and write simply $T$ for a free splitting. The conjugacy classes of vertex stabilizers of $T$ forms a free factor system $\F(T)$. Two free splittings are equivalent if there is an $F_n$-equivariant homeomorphism $S \mapsto T$, in which case $\F(S)=\F(T)$. On each free splitting we have the \emph{natural simplicial structure} whose vertices are the points of valence~$\ge 3$, and we have the \emph{natural geodesic metric} which assigns length~$1$ to each natural edge. Fix a set of representatives $\S$ of equivalence classes of free splittings.

The free splitting complex $\FS(F_n)$ has a $k$-simplex for each $T \in \S$ such that the set of natural edges has $k+1$ orbits under the action of $F_n$. The simplex corresponding to $S \in \S$ has a face corresponding to $T \in \S$ if and only if there is a \emph{collapse map $S \mapsto T$} meaning an $F_n$-equivariant map whose point pre-images are connected. Alternatively, we may regard $\S$ as the vertex set of the first barycentric subdivision $\FS'(F_n)$; the 1-skeleton of $\FS'(F_n)$ has a directed edge from $S \in \S$ to $T \in \S$ if and only if there is a collapse map $S \mapsto T$; and $\FS'(F_n)$ is characterized as the directed simplicial complex generated by its 1-skeleton, with a $k$-simplex for every directed path of length $k$. 

A \emph{map} $f \from S \to T$ between $S,T \in \S$ is a function which is simplicial with respect to some simplicial structures on $S$ and $T$ that refine the natural structures. Given $\Phi \in \Aut(F_n)$, a map $f$ is \emph{$\Phi$-twisted equivariant} if $f(\gamma \cdot x) = \Phi(\gamma) \cdot f(x)$; the case $\Phi = \Id$ yields ordinary equivariance. If not otherwise specified, maps are assumed to be equivariant in the ordinary sense.

The canonical right action of $\Out(F_n)$ on $\S$ is defined as follows: for each $T \in\S$ and $\theta \in \Out(F_n)$ the image $T^\theta \in \S$ is well-defined up to equivalence by precomposing the given action $F_n \act T$ with an automorphism $\Theta \in \Aut(F_n)$ representing $\theta$. This extends to a simplicial right action of $\Out(F_n)$ on $\FS'(F_n)$, and to a continuous action on the Gromov bordification $\overline\FS(F_n) = \FS(F_n) \union \bdy \FS(F_n)$. 

For readability, expressions of the right action $T^\theta$ will sometimes be rewritten $T \cdot \theta$.

Many relations between free splittings and conjugacy classes (of group elements or subgroups) satisfy an equivariance property with respect to each $\theta \in \Out(F_n)$ acting from the right on free splittings and $\theta^\inv$ acting from the left on conjugacy classes. For example, a conjugacy class $c$ is elliptic with respect to $S \in \S$ if and only if $\theta^\inv(c)$ is elliptic with respect to $S^\theta$. The same format holds for relations between free splittings and lines, and the following example of this relation will play a key role in the proof of the WWPD Construction Theorem:

\begin{theorem} \cite[Theorem 1.2]{\FSLoxTag}
The set of all $\tau \in \bdy\FS(F_n)$ which are attracting or repelling points for elements of $\Out(F_n)$ acting loxodromically on $\FS(F_n)$ corresponds bijectively with the set of filling laminations $\Lambda$ for elements of $\Out(F_n)$. This correspondence is equivariant in the following manner: for all corresponding pairs $\tau \leftrightarrow \Lambda$ and all $\theta \in \Out(F_n)$, $\tau^\theta \leftrightarrow \theta^\inv(\Lambda)$ is also a corresponding pair. 
\end{theorem}

Let $\S_\F$ denote those $T \in \S$ such that $\F(T)=\F$. The flag subcomplex of $\FS'(F_n)$ spanned by $\S_\F$ is denoted $\K_\F$ and is called the \emph{spine of the outer space of $F_n$ rel $\F$}; for example if $\F = \emptyset$ then $\K_\F$ is the ordinary spine of the outer space of~$F_n$. The action of any $\theta \in \Out(F_n)$ takes $\K_\F$ to $\K_{\theta^\inv(\F)}$, and hence we have an equation of stabilizer subgroups $\Stab(\F) = \Stab(\K_\F)$. The complex $\K_\F$ is contractible, being a restricted deformation space in the sense of \cite[Theorem 6.1]{GuirardelLevitt:DefSpaces}).

Given $S \in \cS_\F$ and a conjugacy class $c$ of $F_n$ which is not carried by~$\F$ and hence is loxodromic on $S$, we say that $c$ is \emph{$S$-intermediate} or that $c$ \emph{does not fill~$S$} if there exists an $F_n$-orbit of edges which misses the axes of~$c$. Note that an $S$-intermediate conjugacy class is intermediate  in the sense of Definition~\ref{DefIntermediate}. Note also that an $S$-intermediate conjugacy class exists for each $S \in \cS_\F$, as follows. Since $\F$ has co-edge number~$\ge 2$, there exists a finite arc $\alpha \in S$ whose endpoints $v,w \in S$ have nontrivial stabilizer, such that $\alpha$ misses the $F_n$-orbit of some edge $E \subset S$. Since edges of $S$ have trivial stabilizers, one easily produces $\gamma \in F_n$ acting loxodromically on $S$ with axis equal to a union of translates of $\alpha$ which therefore misses the orbit of $E$; the conjugacy class of $\gamma$ is therefore $S$-intermediate.

\subparagraph{The well function on the spine of relative outer space.} \emph{Henceforth,} as we did in Section~\ref{SectionWells}, we shall fix attracting neighborhoods of $\Lambda^-_\phi$ and of $\Lambda^+_\phi$ with respect to which the well functions $w_{\phi^\inv}$ and $w_\phi$ are defined.

Consider $S \in \S_\F$ and a pair of $S$-intermediate conjugacy classes $c_1,c_2$, so for $i=1,2$ the axis of $c_i$ in $S$ misses the orbit of the interior of some natural edge $E_i \subset S$. If $E_1,E_2$ have the same orbits then $c_1,c_2$ are supported by the same proper free factor system --- namely, the conjugacy classes of stabilizers of components of $S \setminus (F_n \cdot E_1) = S \setminus (F_n \cdot E_2)$ --- and so we may apply \cite[Lemma 4.14 (3)]{\FSLoxTag} to conclude that $\abs{w_{\phi^{-1}}(c_1) - w_{\phi^{-1}}(c_2)}$ is uniformly bounded independent of $c_1,c_2$. If on the other hand $E_1,E_2$ have different orbits then there are one-edge free splittings $S_1,S_2$ connected by an edge such that $c_i$ is elliptic in $S_i$ --- namely, $S_i$ is obtained from $S$ by collapsing to a point each component of $S \setminus F_n \cdot E_i$ --- and so we may apply \cite[Lemma 4.20]{\FSLoxTag} to get the same conclusion. Define a well function
$$W_\phi(S) = w_{\phi^{-1}}(c)
$$
for all $S \in \cS_\F$, by simply choosing $c$ to be any $S$-intermediate conjugacy class. Note that $W_\phi(S)$ is coarsely well-defined, meaning that changing the choice of $c$ changes the value of $W_\phi(S) \in \Z$ by at most a constant independent of $S$ and $c$. Note also that $W_\phi(S)$ is a Lipschitz function on the spine of relative outer space: if $S,T \in \cS_\F$ and if there is a collapse map $S \mapsto T$ then any $T$-intermediate conjugacy class $c$ is also $S$-intermediate and so $\abs{W_\phi(S)-W_\phi(T)}$ is uniformly bounded. We record this as:

\begin{lemma}\label{LemmaCoarseAndLipschitz} $W_\phi(S)=w_{\phi^\inv}(c) \from \S_\F \to \mathbb{Z}$ is coarsely well-defined and Lipschitz.~\qed
\end{lemma}
 
\begin{remark} Here we have defined $W_\phi(S)$ for free splittings $S \in \S_\F$ where $\F=\F(S)$ is a multi-edge free splitting, whereas in \cite[Definition 4.15]{\FSLoxTag} we instead defined $W_\phi(T)$ for one-edge free splittings $T$. While the domains of these two well functions are at opposite extremes in some sense, they are nonetheless related: if there is a collapse map $S \mapsto T$ then $\abs{W_\phi(T) - W_\phi(S)}$ is uniformly bounded. This is an easy consequence of \cite[Definition 4.15]{\FSLoxTag} and of the properties laid out in Section~\ref{SectionConjClassWell} regarding well functions on intermediate conjugacy classes, namely coarse well-definedness and the coarse additive inverse property; see also Remark~\ref{RemarkWellFunctions}. One could therefore unify the two definitions to obtain a coarsely well-defined, Lipschitz well function $W_\phi$ on the free splitting complex of $F_n$ relative to~$\F$ in the sense of~\cite{HandelMosher:RelComplexHyp}.
\end{remark}

If $c$ is $S$-intermediate  and $\xi \in \cH$ then $\xi^{-1}(c)$ is $S^\xi$-intermediate.  Thus 
$$W_\phi(S^\xi) = w_{\phi^{-1}}(\xi^{-1}(c))
$$ 
for any $S$-intermediate $c$.  In the special case that $\xi = \phi^m$ we have 
\begin{align*}
W_\phi(S^{\phi^m}) &= w_{\phi^{-1}}(\phi^{-m}(c)) = w_{\phi^{-1}}(c) + m \\ &= W_\phi(S) + m
\end{align*}
(This accounts for our defining $W_\phi$ in terms of $w_{\phi^{-1}}$ and not $w_\phi$; if we had done the latter, we would have $W_\phi(S^{\phi^m})   = W_\phi(S) - m$.)

\subsection{Setting up the proof.} 
\label{SectionThmESetup}
We may assume without loss of generality that $\phi$~and~$\phi^{-1}$ are rotationless. The action of $\phi$ on $\fscn$ is loxodromic by \cite[Theorem 1.1]{\FSLoxTag}. 

For proving the WWPD Construction Theorem by contradiction, we assume that of $\phi$ is not a WWPD element of the action $\cH \act \FS(F_n)$. Using that WWPD is equivalent to item~\pref{ItemBadRayNo} of Proposition~\ref{PropWWPDProps}, there exists a sequence $\theta_i \in \cH$ satisfying the following:

\bigskip

\noindent\textbf{Distinct Coset Property:} $\theta_i^{-1}\theta_j \not \in \Stab(\partial_\pm \phi)$ for $i \ne j$.

\medskip\noindent\textbf{Long Cylinder Property:} For each $S_0 \in \cS_\F$ there exists $R>0$ such that for any $K \ge 0$ there exists $I \ge 0$ such that if $0 \le k \le K$ and if $i \ge I$ then 
$$d(S_0 \cdot {\phi^{k}\theta_i}, S_0 \cdot {\phi^{k}}) < R 
$$
Setting $k=0$ and applying Lemma~\ref{LemmaCoarseAndLipschitz}, after further increasing $R$ the following inequality also holds for all $i$
$$\abs{W_\phi(S_0 \cdot {\theta_i})} < R
$$ 
Therefore, for any $K \ge 0$ there exists $I \ge 0$ such that if $0 \le k \le K$ and if $i \ge I$ then  
$$ \abs{W_\phi(S_0 \cdot {\phi^{k}\theta_i}) -k} < R 
$$
The latter inequality uses that $W_\phi(S_0 \cdot {\phi^{k}}) = W_\phi(S_0)+k$.

\bigskip\emph{Henceforth} we fix an initial choice of the sequence~$\Theta = (\theta_1,\theta_2,\ldots)$ satisfying the \emph{Distinct Coset Property} and the \emph{Long Cylinder Property}, although notice that these two properties each continue to hold after passing to an arbitrary subsequence of $\Theta$, which we will do as the proof proceeds. We also fix a choice of $S_0$ and a corresponding choice of $R$, although in later parts of the proof we may impose further constraints on the choice of~$S_0$ (the \emph{Long Cylinder Property} lets us change $S_0$ at the expense of increasing $R$).

\paragraph{Case analysis and outline of the proof:} Let $\C_\Theta$ be the set whose elements are infinite sequences $(c_1,c_2,\ldots)$ of intermediate conjugacy classes such that for some $S \in \cS_\F$ --- and hence for every $S \in \cS_\F$ (Lemma~\ref{LemmaEqQI}~\pref{ItemLengthBilip}) --- the length $L_S(c_i)$ of $c_i$ in $S$ is bounded independently of $i$, whereas the length $L_S(\theta_i^{-1}(c^{\vphantom{-1}}_i))$  is not  bounded independently of~$i$.

The proof of the WWPD Construction Theorem will break into Case~1 where $\C_\Theta \ne \emptyset$, and Case~2 where $\C_\Theta = \emptyset$; the geometric meaning of this case analysis is explained in Lemma~\ref{LemmaCTheta}. Case~2 will break into two further subcases, based on an analysis of the images of a generic leaf of $\Lambda^-$ under the outer automorphisms $\theta_i^\inv$.

In Cases 1 and 2(a), using neither the \emph{Distinct Coset Property}, nor the hypothesis on triviality of the restrictions of $\phi$, after passing to a subsequence of $(\theta_1,\theta_2,\ldots)$ we will obtain contradictions to regularity properties of well functions that were derived in Section~\ref{SectionWells}.  Having eliminated cases 1 and 2(a) we will be able to pass to a further subsequence of $(\theta_1,\theta_2,\ldots)$ satisfying very strong properties. In the final case 2(b) those properties will be combined with the hypothesis on triviality of restrictions of~$\phi$, allowing us to pass to one final subsequence for which the images $\theta_i^\inv(\Lambda^-)$ of the repelling lamination $\Lambda^-$ are all the same. By applying results of \FSLox\ it will follows that the images $\theta_i(\bdy_-\phi) \in \bdy\FS(F_n)$ of the repelling fixed point are all the same, as are the images $\theta_i(\bdy_+\phi) \in \bdy\FS(F_n)$ of the attracting fixed point, from which we will obtain a contradiction to the \emph{Distinct Coset Property}. 

\bigskip

For each $S \in \cS_\F$, define $\C_\Theta(S)$ to be set of all sequences $(c_1,c_2,\ldots)$ in $\C_\Theta$ such that each $c_i$ is $S$-intermediate. 

\begin{lemma} \label{LemmaCTheta}The following are equivalent.
\begin{enumerate}
\item\label{ItemCThetaEmpty}
$\C_\Theta = \emptyset$.
\item\label{ItemCThetaSEmpty}
$\C_\Theta(S) =  \emptyset$ for some $S \in \cS_\F$.
\item\label{ItemUnifLipEq}
For all $S \in \cS_\F$, the outer automorphisms $\theta_i^\inv$ are represented by $\theta_i^\inv$-twisted equivariant maps $\bar h_i : S \to S$ with uniformly bounded Lipschitz constants, uniformly bounded cancellation constants, and uniformly bounded quasi-isometry constants.  
\end{enumerate}
\end{lemma}

Underlying the proof of Lemma~\ref{LemmaCTheta} are Lemmas~\ref{LemmaEqQI} and~\ref{LemmaTopLegalCircuits}. The first of these includes results of Forester and of Guirardel-Levitt, together with a uniformity clause. 
 
\begin{lemma}\label{LemmaEqQI} For all $S,T \in \cS_\F$ with their natural geodesic metrics, we have:
\begin{enumerate}
\item\label{ItemQIExists}
\cite[Theorem 1.1]{Forester:Deformation}  There exists an $F_n$-equivariant quasi-isometry $S \mapsto T$.
\item\label{ItemLengthBilip}
\cite[Theorem 3.8 (7)]{GuirardelLevitt:DefSpaces}
There exists $k \ge 1$ depending only on $S,T$ such that for all conjugacy classes $c$ in $F_n$ we have $\frac{1}{k} \cdot L_T(c) \le L_S(c) \le k \cdot L_T(c)$.
\item\label{ItemEquivIsQI}
For every $\ell \ge 1$ there exists $\ell' \ge 1$, $c' \ge 0$ such that for every $\Phi \in \Aut(F_n)$, each $\ell$-Lipschitz $\Phi$-twisted equivariant map $f \from S \mapsto T$ is an $(\ell',c')$-quasi-isometry.
\end{enumerate}
\end{lemma}
 
\begin{proof} The twisted version of \pref{ItemEquivIsQI} follows from the untwisted version by precomposing the action homomorphism $F_n \to \Aut(S)$ with an appropriate automorphism. To prove the untwisted version, the given map $f$ factors as a product of an initial collapse map---which collapses to a point each edge whose $f$-image is a vertex of $T$---followed by a Stallings fold sequence. Since $S$ and $T$ have their natural geodesic metrics, the number of folds in the sequence is bounded by a constant depending only on the Lipschitz constant $k$ and the rank~$n$. It therefore suffices to consider separately the cases that $f$ is either a collapse map or a single fold, because a composition of sequence of quasi-isometries has QI constants depending only on the sequence of QI constants of the factors of the composition. 

If $f \from S \to T$ is a collapse map, let $\sigma \subset S$ be the union of collapsed natural edges. Each component of $\sigma$ contains at most one natural vertex in each orbit, because $S,T$ have the same vertex stabilizers. The diameters of the components of $\sigma$, in the natural metric on $S$, are therefore uniformly bounded by a constant $D$ depending only on the rank~$n$. Given $x,y \in S$, the segment $\overline{xy}$ decomposes into an alternating concatenation of segments of uncollapsed edges and collapsed segments of length at most~$D$, and so $d_T(f(x),f(y)) \ge \frac{1}{1+D} d(x,y) - 2D$; since $f$ is $1$-Lipschitz, this gives uniform quasi-isometry constants. 

If $f \from S \to T$ is a fold map, then there are oriented natural edges $E_0,E_1$ with the same initial vertex $v$ and initial segments $e_i \subset E_i$ with terminal points $w_i$ such that $f$ is a quotient map that identifies $\gamma \cdot e_0$ and $\gamma \cdot e_1$ for all $\gamma \in \gamma$. The points $w_0,w_1$ are in different orbits and at most one has nontrivial stabilizer, because $S,T$ have the same vertex stabilizers. It follows that if $x,y \in S$ satisfy $f(x)=f(y) \in T$ then the segment $\overline{xy}$ has length at most~$4$. We may assume that if $e_i \ne E_i$ then $e_i$ has length~$\frac{1}{2}$ and so on each edge of $S$ the map $f$ stretches length by a factor between~$1$ and~$2$. It follows that $f$ has uniform quasi-isometry constants independent of~rank. 
\end{proof}

Suppose that $G$ is a marked graph and that $H $ is a subgraph whose non-contractible components represent $\F$. In this context we refer to the edges of $G\setminus H$ as \emph{top} edges; since $\F$ has co-edge number $\ge 2$, there are at least two top edges. We obtain a free splitting $S \in \cS_\F$ from the universal cover $\wt G$ by collapsing to a point each component of the full pre-image $\wt H$ of $H$. We say that $S$ is \emph{determined by} or \emph{represented by the graph pair $(G,H)$}. (In \cite{\FSLoxTag} we required that $H$ be a natural subgraph of $G$ and have only non-contractible components.  In the current context it is more useful to drop these requirements.) Every $S \in \cS_\F$ is represented by at least one graph pair $(G,H)$. If $S$ is determined by $(G,H)$ then a conjugacy class is $S$-intermediate if and only its representative in $G$ crosses some but not all top edges. 
 
Let $\cS_\F^* \subset \cS_\F$ be the set of those $S \in \cS_\F$ such that one of two possibilities holds: either every vertex of $S$ has non-trivial stabilizer; or all vertex stabilizers of $S$ are trivial and there is only one orbit of vertices. Equivalently $S$ is realized by a graph pair $(G,H)$ such that: either $H$ is nonempty contains all the vertices of $G$ and has no contractible components; or $H$ is empty and $G$ is a rose. The first possibility applies when $\F$ is nonempty, the second when $\F$ is empty.

 \begin{lemma} \label{LemmaTopLegalCircuits} For each $S \in \cS_\F^*$, for each graph pair $(G,H)$ representing $S$ as above, and for each $\theta \in \cH$, there exists a homotopy equivalence $h : (G,H) \to (G,H)$ representing $\theta^{-1}$ so that for each top edge $E$ there is a closed path   $\tau_E$ forming a circuit  and satisfying the following properties.
 \begin{enumerate}
 \item \label{item:does not fill}$\tau_{E} $ crosses $E$ but not all top edges.
 \item \label{item:no cancellation}No top edges  are cancelled when  $h_i(\tau_{E} )$ is tightened to $(h_i)_\#(\tau_{E})$.
  \item \label{item:at most two} Each $\tau_{E}$ crosses at most two  top edges, counted with multiplicity.  \end{enumerate}  
 \end{lemma}

 \begin{proof}   As a first case, assume that $H$ contains every vertex in $G$ and that each component of $H$ is non-contractible.      Choose a homotopy equivalence $h :(G,H) \to (G,H)$ representing $\theta$ that restricts to an immersion on each edge and preserves each component of $H$.  Let $C_1$ and $C_2$ be the components of $H$ that contain the initial and terminal endpoints of $E$ respectively and let    $\mu \subset C_1$ and $\nu\subset C_2$ be the maximal initial and terminal subpaths of $(h)_\#(E)$ that are contained in $H$.      If $C_1 \ne C_2$ then we choose  $\tau_{E} = E\beta_2  \overline E \beta_1$ where $\beta_1 \subset C_1$ and $\beta_2 \subset C_2$ are any closed paths.  Top edges are not  cancelled when $h(  E\beta_2  \overline E \beta_1)$ is tightened to $(h)_\#( E\beta_2  \overline E \beta_1)$ because  $  \nu (h)_\#(\beta_2) \bar \nu$  and $  \bar \mu (h)_\#(\beta_1)   \mu$ do not tighten to  trivial paths.    If $C_1 = C_2$   then we  choose $\tau_{E} = E\beta_1$ where $\beta_1 \subset C_1$ is a closed path  such that $  \nu (h)_\#(\beta_1) \bar \mu$ does not tighten to the trivial path.  This completes the proof of \pref{item:no cancellation};    \pref{item:does not fill} and \pref{item:at most two} are clear from the construction because $G \setminus H$ has at least two edges.
 
 Suppose now that $H = \emptyset$ and that $G$ is  a rose.    Choose a homotopy equivalence $h : G \to G$ representing $\theta$ that restricts to an immersion on each edge, that fixes the unique vertex  and that has at least two gates at that vertex.    The easy case is that for each edge $E$, either the ends of   $E$ are contained in distinct gates  or the ends of $E$ are contained in a single gate  and there is an edge $E'$ with neither end in that gate. In the former case we take $\tau_{E_1} = E_1$ and in the latter      $\tau_{E_1} = E_1E'$. 
 
 The hard case is that there is a gate $\alpha$ such that every edge has at least one end in $\alpha$ and some edge has both ends in $\alpha$.   We induct on the sum of the lengths $|h(E)|$ of the paths $h(E)$ over all edges $E$ that have both ends in $\alpha$. Enumerate and orient the edges of $G$ as $E_1,\ldots, E_K, E_{K+1},\ldots, E_L \,\, (1 \le  K < L)$, so that $E_1,\ldots,E_K$ have both ends in $\alpha$ and $E_{K+1},\ldots, E_L$ have only their initial end in $\alpha$. 

Let $\beta$ be the complement of $\alpha$, consisting of the terminal ends of $E_{K+1}, \ldots, E_L$. There is an oriented edge $\eta$ so that the edge path $h(E_j)$ begins with $\eta$ [resp. ends with $\bar\eta$] if and only  $E_j$ has initial end in $\alpha$ [resp. terminal end in $\alpha$]. Post-compose $h$ with a map that drags the unique vertex of $G$ across $\eta$, producing a new map $h' : G \to G$ still representing $\theta^{-1}$. Then the edge path $h'(E_j)$ is obtained from $h(E_j)$ by first removing the initial $\eta$ and then either removing the terminal $\bar \eta$  if $j \le K$ or adding $ \eta$ to the terminal end if $j > K$.   

Replace $h : G \to G$ by $h': G' \to G'$. If one of the easy cases holds, we are done.  Otherwise there is a gate $\alpha'$ such that every edge has at an end in $\alpha'$ and some edge has both ends in $\alpha'$.      All ends in the set $\beta$ (the terminal ends of $E_{K+1},...,E_L)$ are mapped by $h'$ to the direction $\bar\eta$, and no other end is mapped by $h'$ to $\bar\eta$, so $\beta$ is a single gate of $h'$ and no edge has both ends in the gate $\beta$. In particular,  $\beta \ne \alpha'$. Furthermore, $\alpha' $ must contain all the initial ends of $E_{K+1}\ldots E_L$, so none of $E_{K+1}\ldots,E_L$ have both ends in $\alpha'$. Thus the edges with both ends in $\alpha'$ form a subset of $E_1,\ldots,E_K$. Since  $|h'(E_k)| = |h(E_k)| - 2$ for all $1 \le k \le K$,  we are done by induction.  
 \end{proof}

\begin{proof}[Proof of Lemma~\ref{LemmaCTheta}] $\pref{ItemUnifLipEq} \Longrightarrow \pref{ItemCThetaEmpty} \Longrightarrow \pref{ItemCThetaSEmpty}$ is obvious so it suffices to show that $\pref{ItemCThetaSEmpty} \Longrightarrow \pref{ItemUnifLipEq}$. Suppose then that $\C_\Theta(S) = \emptyset$. Once we produce the maps $\bar h_i$ with uniformly bounded Lipschtz constants, the uniform bounds on cancellation constants and quasi-isometry constants follow from \cite[Lemma 3.1]{\BookZeroTag} and Lemma~\ref{LemmaEqQI} respectively. 

Choose $S' \in \cS_\F^*$ that is obtained from $S$ by collapsing a forest. If a conjugacy class is $S'$-intermediate then it is  $S$-intermediate. Thus $\C_\Theta(S') = \emptyset$. Since the equivariant collapse map $S \mapsto S'$ is 1-Lipschitz, it is a quasi-isometry with uniform constants by Lemma~\ref{LemmaEqQI} and hence has an equivariant coarse inverse $S' \mapsto S$ with uniform constants. It therefore suffices to verify \pref{ItemUnifLipEq} for $S'$. 

Let $(G,H)$ be  a graph pair representing  $S'$ as in the definition of $\cS_\F^*$. For each $\theta_i$ in the sequence $\Theta$ and for each top edge $E$ of $G \setminus H$, let  $h_i : (G,H) \to (G,H)$ be the homotopy equivalence representing $\theta_i^{-1}$ and $\tau_{E,i} \subset G $ the circuit  obtained by applying Lemma~\ref{LemmaTopLegalCircuits}. Lifting $h_i$ to the universal cover and collapsing the total lift of $H$, the map $h_i$ induces a $\theta_i^\inv$-twisted equivariant map $\bar h_i \from S' \to S'$. The conjugacy class $c_i$ determined by   $\tau_{E,i}$  is $S'$-intermediate and satisfies $L_{S'}(c_i) \le 2$.  Since $(c_1,c_2\ldots)$ is not an element of $\C_\Theta(S')$ it must be that $L_{S'}(\theta_i^{-1}c_i) = L_{S'}([(h_i)_\#(\tau_{E,i})])$ is uniformly bounded. Lemma~\ref{LemmaTopLegalCircuits}~\pref{item:no cancellation}  implies that there is a uniform bound to the number of top edges crossed by $h_i(E)$ for each top edge $E$.  Equivalently, the  Lipschitz constants are uniformly bounded.   
\end{proof}

\subsection{The case analysis.} 
\label{SectionThmECaseAnalysis}
We begin with:

\smallskip\noindent\textbf{Case 1: $\C_\Theta \ne \emptyset$} 

\smallskip\noindent
It is convenient in this case to choose $S_0 \in \cS_\F^*$. Let $(G_0,H_0)$ be a graph pair representing $S_0$. For each $\theta_i$ in the sequence $\Theta$ and for each top edge $E$ of $G_0 \setminus H_0$, let $h_i \from (G_0,H_0) \to (G_0,H_0)$ be the homotopy equivalence representing $\theta_i^{-1}$ and $\tau_{E,i} \subset G_0 $ the circuit  obtained by applying Lemma~\ref{LemmaTopLegalCircuits}.

\begin{lemma} \label{two consequences} 
 Suppose that $(c_1,c_2,\ldots)$ is an element of $\C_\Theta$. Then there are arbitrarily large $i$ and arbitrarily large $M$ and  a top edge $E$ of $G_0$   such that $h_i(E)$ and the circuit representing $\theta_i^{-1}(c_i)$  have a common subpath that crosses $M$ top edges.
 \end{lemma}

\begin{proof}  Decompose the circuit $\sigma_i \subset G_0$ representing $c_i$ into a concatenation of a uniformly bounded number of subpaths, each of which is either contained in $H$ or is a top edge. The circuit $(h_i)_\#(\sigma_i)$ representing $\theta_i^{-1}(c_i)$ therefore decomposes into a concatenation of  a uniformly bounded number of subpaths, each of which is either contained in $H$ or contained in the $h_i$-image of some top edge. Since $(c_1,c_2\ldots) \in \C_\Theta$, the number of top edges crossed by the latter subpaths is unbounded.  
\end{proof}

For each $m \ge 0$, let $S_m =  S_0^{\phi^m}$.     By Lemma~\ref{LemmaCTheta} we can choose an element $(c_1,c_2,\ldots)$ in $\C(S_m)$.   From the fact that $c_i$ is  $S_m$-intermediate, it follows that  $\theta_i^{-1}(c_i)$ is $S_m^{\theta_i}$-intermediate.   Thus 
$$W_\phi(S_m^{\theta_i}) =  w_{\phi^{-1}}(\theta_i^{-1}(c_i))$$ 
for all $i$.  Similarly, Lemma~\ref{LemmaTopLegalCircuits}\pref{item:does not fill} implies that 
$$W_\phi(S_0^{\theta_i}) =  w_{\phi^{-1}}([(h_i)_\#(\tau_{E,i})])$$
where $[(h_i)_\#(\tau_{E,i})]$  is the conjugacy class of the circuit formed by $(h_i)_\#(\tau_{E,i})$.
By  Lemma~\ref{two consequences},     we can choose  $ i$   and a top edge $E$ of $G_0$ so that the circuit    representing $\theta_i^{-1}(c_i)$ in $G_0$  contains a subpath of $h_i(E)$ that crosses at least $m$ top edges.  Lemma~\ref{LemmaTopLegalCircuits}\pref{item:no cancellation} then implies that the circuit    representing $\theta_i^{-1}(c_i)$ in $G_0$ and the circuit $(h_i)_\#(\tau_{E,i})$ have a common subpath that crosses at least $m$ top edges.   Since there is a uniform bound for $w_{\phi^{-1}}([h_\#(\tau_{E,i})]) =W_\phi(S_0^{\theta_i})$,  Corollary~\ref{bounded intersection} implies that there is a uniform bound for $w_{\phi^{-1}}(\theta_i^{-1}(c_i))= W_\phi(S_m^{\theta_i}) $.  This contradicts the assumption that  $|W_\phi(S_m^{\theta_i}) -m| < R$ for all sufficiently large $i$ and so completes the proof in case 1.  

\bigskip\noindent\textbf{Case 2: $\C_\Theta =  \emptyset$}    

\smallskip

Let $\fG$ be a \ct\ representing $\phi^{-1}$  with top stratum $G_s$  and with $\F$ realized by a core filtration element $G_r$. Let $H^z_s$ be the union of $H_s$ with the zero strata, if any, that it envelops \cite[Definition 2.18]{\recognitionTag} and let $G_u = G- H^z_s$. Each zero stratum enveloped by $H_s$ is a contractible component of $G_{s-1}$. Since $\phi \restrict [G_r]$ is trivial and $\phi$ is irreducible rel $[G_r] \sqsubset [G_s]$, $f \restrict G_r$ is the identity and $G_u$ deformation retracts to~$G_r$. The edges of $G_u - G_r$ are non-fixed \neg\ and so are linear.    

We now constrain $S_0 \in \cS$ to be the free splitting determined by the graph pair $(G,G_{s-1})$. By Lemma~\ref{LemmaCTheta}, there exist $\theta_i^{-1}$ twisted equivariant maps $\bar h_i :S \to S$ with uniformly bounded Lipschitz constants, cancellation constants, and quasi-isometry constants. Letting $h_i : (G,G_{s-1}) \to (G,G_{s-1})$ be homotopy equivalences corresponding to $\bar h_i :S \to S$, it follows that:
\begin{enumerate}
\item \label{item:bounded  image}  
For each top edge $E_j$ the number of top edges in $h_i(E_j)$ is uniformly bounded. 
\item  \label{item:bcc} 
There is a uniform constant $C$ so that if $\sigma = \sigma_1 \sigma_2 \subset G$ is a decomposition into subpaths, then for all $i$ at most $C$ pairs of top edges are cancelled when $(h_i)_\#(\sigma_1) (h_i)_\#(\sigma_2)$ is tightened to $(h_i)_\#(\sigma)$.  
\item \label{item:LOneLTwo}
For all $L$ there exists $L'$ independent of $i$ so that if $\sigma \subset G$ crosses $\ge L'$ top edges then $(h_i)_\#(\sigma)$ crosses   $\ge L$ top edges.
\item \label{item:LTwoLOne}
For all $L'$ there exists $L$ independent of $i$ so that for any path $\sigma \subset G$, if $(h_i)_\#(\sigma)$ crosses $\ge L$ top edges then $\sigma$ crosses $\ge L'$ top edges.
\end{enumerate}

As a consequence of \pref{item:bcc} we have the following consequence for the path maps $(h_i)_{\shsh}$ (see Section~\ref{SectionBasicNotions}, or  Section 1.1.6 of \SubgroupsOne\ for a more comprehensive review):
\begin{enumeratecontinue}
\item \label{item:##} For all $i$ and all paths $\sigma \subset G$, $(h_i)_{\#\#}(\sigma)$ contains the path obtained from $(h_i)_\#(\sigma)$ by removing the maximal initial and terminal subpaths with exactly $C$ top edges (\SubgroupsOne, Lemma 1.6).
\end{enumeratecontinue}

After passing to a subsequence of the $\theta_i$'s (and the corresponding subsequences of $n_i$'s and $h_i$'s), we may assume by \pref{item:bounded image}:
 \begin{enumeratecontinue}
\item For each top edge $E_j$ the sequence of top edges crossed by $h_i(E_j)$ is independent of $i$.  
\end{enumeratecontinue}
Given a top edge $E_j$, write $h_i(E_j) = \alpha_0 \epsilon_1 \alpha_1 \ldots \epsilon_k \alpha_k$ where $\epsilon_1,\ldots,\epsilon_k$ are the top edges that are crossed. After passing to a further subsequence of the $\theta_i$'s we may assume that for all $0 \le l \le k$, the subpath $\alpha_i$ is either always trivial or never trivial, independent of $i$. We may therefore subdivide $E_j$ into subpaths called \emph{edgelets} and modify $h_i$, without changing $h_i(E_j)$, so that: 
\begin{enumeratecontinue}
\item For all $i$ and $j$, $h_i$ maps each edgelet in $E_j$ to either a top edge or to a non-trivial path in $G_{s-1}$. 
\end{enumeratecontinue}
  Recall that $\fG$ represents $\phi^{-1}$, not $\phi$, and that tiles are defined with respect to $f$. As $k$ goes to infinity, the number of top edges in a $k$-tile goes to infinity. We may therefore choose a fixed integer $K \ge 1$ so that for each $i,j$, the $K$-tile $\omega_j =  f^K_\#(E_j)$ has the property that its $(h_i)_\#$-image $\sigma_{ij} := (h_i)_\#(\omega_j)$ crosses at least $3C$ top edges.
 
 There may be some cancellation of top edges when $h_i(\omega_j)$ is tightened to $\sigma_j$. After passing to a further subsequence of the $\theta_i$'s, we may assume that this cancellation is independent of $i$; in particular, for each $K$-tile $\omega_j$, the sequence of top edges  crossed by $\sigma_{ij} $ is independent of $i$. Decompose $\sigma_{ij}$ into subpaths $\sigma_{ij} = \sigma^-_{ij} \sigma^+_{ij}$ where $\sigma^-_{ij}$ is the shortest  initial subpath of $\sigma_{ij}$ that completely contains  $C$ top edges and that terminates at the midpoint of a top edge.  Note that the endpoints of $\sigma^-_{ij}$ are independent of $i$. For each $K$-tile $\omega_j$ we define its \emph{prefix--suffix} decomposition $\omega_j = \omega^-_j \omega^+_j$ where $\omega^-_j$ is the shortest initial subpath  satisfying $(h_i)_\#(\omega_j^-) = \sigma_{ij}^-$. The common endpoint of $\omega^-_j$ and $\omega^+_j$, which we refer to as the \emph{midpoint} of $\omega_j$, is contained in the interior of an edgelet and maps to the common endpoint of $\sigma^-_{ij}$ and $\sigma^+_{ij}$.  Since the cancellation of top edges when $h_i(\omega_j)$ is tightened to $\sigma_j$ is independent of $i$  we have:

 \begin{enumeratecontinue}
\item \label{midpoint independence} For each $K$-tile $\omega_j$ its prefix--suffix decomposition $\omega_j = \omega^-_j \omega^+_j$, its midpoint and the $h_i$-image of its midpoint are independent of $i$.
\end{enumeratecontinue}

A generic leaf $\gamma^-$ of $\Lambda^-_{\phi}$ has a $0$-tiling; i.e.\ a splitting whose terms are either edges in $H_s$ or paths in $G_{s-1}$. After increasing $K$ if necessary, we may assume that $f^K(H_j) \subset G_u$ for each zero stratum $H_j$ enveloped by $H_s$. Applying $f_\#^K$ to the above $0$-splitting, we have a splitting of $\gamma^-$ whose terms are either $K$-tiles or paths in $G_u$. We may therefore write  $\gamma^-$ as an alternating concatenation of the form
$$\gamma^- = \cdots \omega_{j_{t-1}} \,\, \rho_t \,\, \omega_{j_t} \,\, \rho_{t+1} \,\, \omega_{j_{t+1}} \cdots
$$
where each $\omega_j$ is a $K$-tile and each $\rho_t$ is either trivial or contained in $G_{u}$. After taking the prefix--suffix decomposition of each $K$-tile and then collecting terms we obtain a new decomposition of $\gamma^-$ as follows:
\begin{align*}
(*) \qquad \gamma^- 
        &= \cdots \omega^-_{j_{t-1}} \,
                        \underbrace{\omega^+_{j_{t-1}} \, \rho_t \,  \,
                        \omega^-_{j_t}}_{\mu_t}  \,
                        \underbrace{\omega^+_{j_t} \, \rho_{t+1}  \, \,
                         \omega^-_{j_{t+1}}}_{\mu_{t+1}}  \,
                         \omega^+_{j_{t+1}} 
                        \cdots \\
         &= \cdots \mu_{t-1} \,\,\mu_t \,\, \mu_{t+1} \,\,\mu_{t+2} \cdots
\end{align*}
 Our choice of $C$ guarantees that the above decomposition is a ``universal $1$-splitting'' in the sense that 
$$(h_i)(\gamma^-) =  \,\,\,\ldots \,\,\,(h_i)_\#(\mu_{-1}) \,\,\, (h_i)_\#( \mu_{0}) \,\,\, (h_i)_\#(\mu_{1}) \,\,\,\ldots
$$  
is a decomposition into subpaths for all $h_i$. By \pref{midpoint independence}, 
\begin{enumeratecontinue}
\item \label{image independence} For each $\mu_t$, the $h_i$-image of the endpoints of $\mu_t$ are independent of $i$.
\end{enumeratecontinue} 
\vspace{.1in}

The proof now divides into two subcases.  

\medskip
\noindent{\bf Case 2a: \,  There exists $p<q$ so that for infinitely many values of $i$, the path $(h_i)_\#(\mu_{p} \ldots \mu_{q})$ is not a subpath of $\gamma^-$.}

\medskip\noindent
After passing to a subsequence of the $\theta_i$'s we may assume that the assumption of Case~2a holds for \emph{all} values of~$i$.  Choose $p' < p $ and $q' > q$ so that $(h_i)_\#(\mu_ {p'} \ldots \mu_{p-1})$ and $(h_i)_\#(\mu_ {q+1} \ldots \mu_{q'})$ cross  at least $C$ top edges.   By \pref{item:##}  
$$(h_i)_{\#\#}(\mu_{p'} \ldots \mu_{q'})  \supset (h_i)_\#(\mu_{p} \ldots \mu_{q})
$$
for all $i$. By \cite[Lemma 3.1.8]{\BookOneTag} there exists $K'$ such that 
$$f^{K'}_\#(E_j) \supset \mu_{p'} \ldots \mu_{q'}
$$
for each $j$. From Lemma 1.6-(3) of \SubgroupsOne\ it therefore follows that  
$$(h_i)_{\#\#}(f^{K'}_\#(E_j)) \supset (h_i)_{\#\#}(\mu_{p'} \ldots \mu_{q'})
$$
We conclude that:
\begin{itemize}
\item[$(\#)$] If a path $\sigma_0$ contains a $K'$-tile then $(h_i)_\#(\sigma_0)$ contains $(h_i)_\#(\mu_{p} \ldots \mu_{q})$, and hence $(h_i)_\#(\sigma_0)$ does not occur as a subpath of $\gamma^-$.
\end{itemize}   
 
Choose a circuit $\alpha \subset G$ that crosses some but not all edges of $G - G_{s-1}$. Then the circuit
$$\beta_{i,m} := (h_i)_\#f_\#^m(\alpha) \subset G
$$
determines a conjugacy class in $F_n$ that is represented by a circuit in $G \cdot \phi^m\theta_i$ which crosses some but not all top edges of $(G,G_{s-1}) \cdot \phi^m\theta_i$. It follows that 
$$w_{\phi^\inv}[\beta_{i,m}] = W(S_0 \cdot \phi^m\theta_i)
$$
and hence (after further increasing $R$) for each~$m$ we have $\abs{w_{\phi^\inv}[\beta_{i,m}] - m}<R$ if~$i$ is sufficiently large. It follows that:
\begin{itemize}
\item[$(\shsh)$] There is a sequence $L_m \to +\infinity$ such that for each $m$, if $i$ is sufficiently large then there exists a common subpath of $\gamma^-$ and $\beta_{i,m}$ that crosses $L_m$ top edges. 
\end{itemize}

For reasons as above we have
$$w_{\phi^\inv}[f^m_\#(\alpha)] = W(S_0^{\phi^m}) 
$$
Thus  $w_{\phi^\inv}[f^m_\#(\alpha)]$ goes to infinity with~$m$. By Proposition~\ref {PropKTiles}, for all sufficiently large $m$, $f_\#^m(\alpha)$ has a weak $K'$-tiling. By Lemma~\ref{endpoint condition}~\pref{k tile separation} there exists $L' > 0$ so that for all sufficiently large $m$, every subpath $\sigma \subset f_\#^m(\alpha)$ that contains at least $L'$ top edges contains a $K'$-tile. Choose $L> 0$ so that if $\sigma$ is any path in $G$ and if $\zeta$ is a subpath of $(h_i)_\#(\sigma)$ that crosses at least $L$ top edges then there is a subpath $\sigma_0$ of $\sigma$ that crosses at least $L'$ top edges such that $(h_i)_\#(\sigma_0)$ is a subpath of $\zeta$; this follows by applying item~\pref{item:LTwoLOne} above, with our current $\zeta$ in the role of the $\sigma$ of item~\pref{item:LTwoLOne}.
Thus any subpath of $\beta_{i,m}$ that crosses at least $L$ top edges contains a subpath of the form $(h_i)_\#(\sigma_0)$ where $\sigma_0$ is a subpath of $f^m_\#(\alpha)$ that contains at least $L'$ top edges, and hence $\sigma_0$ contains a $K'$-tile. Combining this with conclusion $(\#)$, it follows that no subpath of $\beta_{i,m}$ that crosses at least $L$ edges occurs as a subpath of $\gamma^{-}$.  This contradicts $(\shsh)$, completing the proof in Case~2a.

\medskip 
\noindent{\bf Case 2b:    For any  $p<q$, the path \ $(h_i)_\#(\mu_{p} \ldots \mu_{q})$ is  a subpath of $\gamma^-$ for all sufficiently large $i$.}

\medskip\noindent
Our first claim is that
\begin{description}
\item[$(*)$] For any given $p$, as $i$ varies, the expression $(h_i)_\#(\mu_{p})$ takes on only finitely many values. 
\end{description}
We can see this as follows. The number of top edges in $(h_i)_\#(\mu_{p})$ is uniformly bounded so the sequence of top edges in $(h_i)_\#(\mu_{p})$ takes on only finitely many values. We are therefore reduced to bounding the number of edges in a maximal subpath $\nu$ of $(h_i)_\#(\mu_{p})$ in $G_{s-1}$. Since $\gamma^-$ is birecurrent, there exists $q > p$ such that $\mu_p = \mu_q$. There is a copy $\nu'$ of $\nu$ in $(h_i)_\#(\mu_{q})$ that is separated from $\nu$ in $(h_i)_\#(\mu_{p} \ldots \mu_{q})$ by a uniformly bounded number of top edges. Assuming without loss that $(h_i)_\#(\mu_{p} \ldots \mu_{q})$ is a subpath of $\gamma^-$, we have that $\nu$ and $\nu'$ are separated in $\gamma^-$ by a uniformly bounded number, say $P$, of top edges. Choose $k$ so that every $k$-tile contains at least $P+1$ edges. Since $\gamma^{-}$ decomposes into $k$-tiles and subpaths in $G_{s-1}$, at least one of $\nu$ and $\nu'$ must be contained in a $k$-tile. This gives the uniform bound on the number of edges in $\nu$ and so completes the proof of the first claim.

\medskip\noindent
\textbf{The topological structure of $\Lambda^-_\phi$.} The fact that $G_u$ is the union of a filtration element of fixed edges $G_r$ and some \neg-linear edges $G_u \setminus G_r$ implies that $\Lambda^-_\phi$ has a particularly simple topological structure. If there were no linear edges, then $\Lambda^-_\phi$ would be minimal, meaning every leaf would be dense. With only linear edges, the structure is still quite simple: for example, the only nondense leaves are the finitely many periodic lines corresponding to circuits around which the linear edges of $G_u \setminus G_r$ twist. We now describe a decomposition of generic leaves of $\Lambda^-_\phi$ which reflects this structure.

The generic leaf $\gamma^-$ decomposes as a concatenation of subpaths in $H_s$ and maximal subpaths in $G_{s-1}$. We further divide the latter into subpaths in zero strata enveloped by $H_s$ called \emph{middle pieces} (referred to in \cite{\recognitionTag} as \lq taken connecting paths\rq) and subpaths in $G_u$ called \emph{bottom pieces}. There are only finitely many middle pieces and the $f_\#$-image of a middle piece is either another middle piece or a bottom piece. Each tile also divides into subpaths in $H$, bottom pieces and middle pieces. A bottom piece is \emph{primitive} if it occurs in a $1$-tile or is the $f_\#$-image of a middle piece. There are only finitely many primitive bottom pieces; we denote them with the symbol~$\delta$. Every bottom piece of a $k$-tile can be written as $f^j_\#(\delta)$ for some primitive bottom piece $\delta$ and some $j \ge 0$. This is obvious for $k=1$ and it follows by the obvious induction argument on $k$ using the fact that each $k$-tile is $s$-legal. 

\newcommand\under{u}

Recall that $f \restrict G_r$ is the identity and that each edge of $G_u \setminus G_r$ is \neg-linear. Since the initial vertices of edges of $G_u - G_r$ have valence one in $G_u$, paths in $G_u$ can only cross edges of $G_u - G_r$ in their first or last edges. A primitive bottom piece $\delta$ therefore has one of three types:
\begin{description}
\item[Constant Type:] $\delta$ is contained in $G_r$; 
\item[Once-Linear Type:] Either the initial or terminal edge of $\delta$, but not both, is contained in $G_u - G_r$; 
\item[Twice-Linear Type:] The initial and terminal edges of $\delta$ are contained in $G_u - G_r$. 
\end{description} 
If $\delta$ has constant type then $f_\#^j(\delta)$ is independent of $j$. If $\delta$ has once linear type then there is an edge $E$ of $G_u - G_r$, a closed path $\under$ in $G_r$ such that $f_\#^j(E) = E \under^j$, and a path $\beta$ in $G_r$, such that
\begin{align*}
\delta = E \, \beta, &\quad\text{and}\quad f_\#^j(\delta) =  E \, [\under^j\beta] \\
\text{or}\quad 
\delta = \beta \,  \overline E, & \quad\text{and}\quad f_\#^j(\delta) =  [\beta\bar \under^j] \,  \overline E
\end{align*}
If $\delta$ has twice linear type then there are edges $E,E'$ in $G_u - G_r$, closed paths $\under, \under'$ in $G_r$ such that $f^j_\#(E) =  E \under^j$ and $f^j_\#(E') =  E' \under'{}^j$, and a path $\beta$ in $G_r$, such that
$$\delta = E \, \beta \, E' \quad\text{and}\quad f_\#^j(\delta) =  E \, [\under^j\, \beta \, \bar \under'^j ] \, \overline E'
$$

Recall that for each term $\mu_t$ of the decomposition $\gamma^- =  \cdots \mu_{t-1} \,\,\mu_t \,\, \mu_{t+1} \,\,\mu_{t+2} \cdots$ there is a decomposition $\mu_t = \omega_{j_{t-1}}^+  \rho_t \, \omega_{j_t}^-$ where $\omega_{j_{t-1}}^+$ is a $K$-tile suffix, $\rho_t$ is either trivial or a maximal subpath in $G_u$, and $\omega_{j_t}^-$ is a $K$-tile prefix. Since every non-trivial subpath of $\gamma^-$ occurs in a tile,  each non-trivial $\rho_t$ can be written as  $f^j_\#(\delta)$ for some primitive bottom piece $\delta$ and some $j \ge 0$.  

To each $K$-tile suffix $\omega^+$, each primitive bottom piece $\delta$, and each $K$-tile prefix $\omega^-$ we associate the set of paths 
$$\cM(\omega^+\!,\delta,\omega^-) = \{\mu_t  \,\, \suchthat \,\, \omega^+ = \omega^+_{j_{t-1}} , \,\,\text{and}\,\,  f^j_\#(\delta) = \rho_t \,\, \text{for some $j \ge 0$, and}\,\, \omega^- = \omega^-_{j_t}\,\, \}
$$
Also, for each $K$-tile suffix $\omega^+$ and $K$-tile prefix $\omega^-$ we define 
$$\cM(\omega^+\!,\omega^-)  = \begin{cases} \quad
\{\omega^+\omega^- \} = \{\mu_t\} & \quad\text{if $\omega^+ = \omega^+_{j_{t-1}}$ and $\omega_- = \omega^-_{j_t}$ for some $t$} \\
\quad \emptyset &\quad\text{otherwise}
\end{cases}
$$
We say that the set $\cM(\omega^+\!,\delta,\omega^-)$ has constant, once-linear or twice-linear type if $\delta$ does. The set $\{\mu_t: t \in \Z\}$ is the (not necessarily disjoint) union of the sets $\cM(\omega^+\!,\delta,\omega^-)$ and the sets $\cM(\omega^+\!,\omega^-)$ as $\omega^+$, $\omega^-$, and $\delta$ vary. Note that there are finitely many of the sets $\cM(\omega^+\!,\delta,\omega^-)$ and $\cM(\omega^+\!,\omega^-)$. Note also that each set $\cM(\omega^+\!,\omega^-)$ has finite cardinality whereas the sets $\cM(\omega^+\!,\delta,\omega^-)$ may have infinite cardinality.

Our second claim is that:
\begin{description}
\item[$(**)$] Each $  \cM(\omega^+\!,\delta,\omega^-)$ contains a finite subset  $  \cM_0(\omega^+\!,\delta,\omega^-)$ so that for any~$i$ and~$j$, if $(h_i)_\#(\mu_t) =  (h_j)_\# (\mu_t)$ for all $\mu_t \in   \cM_0(\omega^+\!,\delta,\omega^-)$ then  $(h_i)_\#(\mu_t)= (h_j)_\# (\mu_t)$ for all $\mu_t \in   \cM(\omega^+\!,\delta,\omega^-)$.    
\end{description}
 
Before proving this claim we use it to complete the proof of the WWPD Construction Theorem, by arriving at a contradiction that settles Case 2b. From claim $(*)$ it follows that for any finite subset $M \subset \{\mu_t: t \in \Z\}$ we can pass to a subsequence so that $(h_i)_\#(\mu_{t}) = (h_j)_\#(\mu_{t})$ for all $i,j$ and for all $\mu_t \in M$. We may therefore assume that $(h_i)_\#(\mu_{t}) = (h_j)_\#(\mu_{t})$ for all $i,j$ and for all $\mu_t$ in each $ \cM(\omega^-,\omega^+)$ and in each $ \cM_0(\omega^+,\delta,\omega^-)$. From claim~$(**)$ it then follows that $(h_i)_\#(\mu_{t})= (h_j)_\#(\mu_{t})$ for all  $i,j$ and all $\mu_t$. This in turn implies that $(h_i)_\#(\gamma^-)= (h_j)_\#(\gamma^-)$ for all $i,j$, and hence taking weak closures that $\theta_i^\inv(\Lambda^-) = \theta_j^\inv(\Lambda^-)$. It follows by \cite[Theorem~1.2]{\FSLoxTag} that $\theta_i(\bdy_-\phi) = \theta_j(\bdy_-\phi)$ and that $\theta_i(\bdy_+ \phi) = \theta_j(\bdy_+\phi)$, and so $\theta_i^\inv\theta_j(\bdy_\pm\phi) = \bdy_\pm\phi$, in contradiction to the \emph{Distinct Coset Property} saying that $\theta_i^{-1}\theta_j \not \in \Stab(\partial_\pm\phi)$.

\bigskip

It remains to verify claim $(**)$. This is obvious if $\cM(\omega^+\!,\delta,\omega^-)$ is finite, in particular if it has constant type and therefore has cardinality at most one. For notational simplicity we let $M = \cM(\omega^+\!,\delta,\omega^-)$, we assume $M$ is infinite, and we denote the desired finite subset by $M_0 = \cM_0(\omega^+\!,\delta,\omega^-)$. There exists an infinite sequence of integers $0 \le j_1 < j_2 < \dots$ so that 
$$M \, = \, \{\tau_l \, := \, \omega^+ \, f^{j_l}_\#(\delta) \,\, \omega^- \, \suchthat \, l \ge 1\}
$$ 
For all $l \ge 1$, the paths $\tau_l$ have the same initial vertex $v_-$ and the same terminal vertex $v_+$. The images $v'_- := h_i(v_-)$ and $v'_+ := h_i(v_+)$ are independent of $i$ by \pref{image independence}. We include $\tau_1 = \omega^+ \, f^{j_1}_\#(\delta) \, \omega^-$ in $M_0$, and then applying $(*)$ we pass to a subsequence of the $\theta_i$'s so that $h_i(\tau_1)$ is independent of~$i$. 

For each $l \ge 2$ consider the closed edge path $\tau_1 \bar \tau_l$ based at $v_-$, with nontrivial path homotopy class denoted $[\tau_1 \bar\tau_l] \in \pi_1(G,v_-)$. Let $\gamma_l$ denote the closed path obtained by straightening $\tau_1 \bar \tau_l$, and so $\gamma_l$ is the unique closed path based at $v_-$ that represents $[\tau_1 \bar\tau_l]$. Let $D = \{[\tau_1 \bar\tau_l] \suchthat l \ge 2\} \subset \pi_1(G,v_-)$, an infinite set of nontrivial elements of $\pi_1(G,v_-)$.

Let $\Upsilon_i \from \pi_1(G,v_-) \to \pi_1(G,v'_-)$ be the isomorphism induced by $(h_i)_\#$. Under the identification $\pi_1(G,v_-) \approx F_n$, we obtain automorphisms 
$$\Upsilon_i^{-1} \Upsilon^{\vphantom{-1}}_1 \in \Aut(\pi_1(G,v_-)) \approx \Aut(F_n)
$$ 
Denote their fixed subgroups by 
$$B_i = \Fix(\Upsilon_i^{-1} \Upsilon^{\vphantom{-1}}_1) = \Fix(\Upsilon^\inv_1 \Upsilon_i^{\vphantom{{-1}}} ) < F_n
$$
For each $l \ge 2$, it follows that the path $(h_i)_\#(\tau_l)$ is independent of $i$ if and only if $\Upsilon_i [\tau_1 \bar\tau_l] \in \pi_1(G,v'_-)$ is independent of $i$ if and only if $[\tau_1 \bar\tau_l] \in B_i$ for all $i$. 

Each $B_i$ has rank~$\le n$ by the solution to the Scott conjecture \cite[Theorem 6.1]{\BHTag}, and each $B_i$ is primitive meaning that if $B_i$ contains a nonzero power of an element then it contains that element. 

It therefore suffices to show that there exists a finite subset $D_0 \subset D$ such that for any subgroup $B < F_n$ that is a primitive and has rank $\le n$, if $D_0 \subset B$ then $D \subset B$. Once that is shown then we finish claim~$(**)$ by defining $M_0 = \{ \tau_l \suchthat [\tau_1 \bar\tau_l] \in D_0\}$.

For the rest of the proof we fix $B < F_n$ to be primitive and of rank~$\le n$; we shall define the subset $D_0 \subset D$ independent of~$B$.

If the set $M$ is of once-linear type then either $\delta = E \beta$ or $\delta = \beta \bar E$, and $f^j_\#(E) = E \under^j$. The two cases are symmetric so we assume that $\delta = E \beta$. Thus 
\begin{align*}
\tau_l &= \omega^+ \, E \, [\under^{j_l} \, \beta] \, \omega^- \\
\gamma_l &= (\omega^+E) \, \under^{j_1-j_l} \, (\omega^+ E)^{-1}
\end{align*}
It follows that as $l$ varies, the $[\tau_1 \bar\tau_l]$'s are all contained in a cyclic subgroup of $\pi_1(G,v_-)$. Any primitive subgroup that contains a single nontrivial element of this cyclic subgroup contains all of $D$, so taking $D_0$ to be any nonempty, finite subset of $D$ we are done in the once-linear case.
 
We are now reduced to the case that $M$ is of twice-linear type. In this case we have $\delta = E \, \beta \, \overline E'$, $f^j_\#(E) = E \under^j$, and $f^j_\#(E') = E' \under'^j$, and so 
\begin{align*}
\tau_{l} &= \omega^+ E \, [\under^{j_l} \, \beta \, \bar \under'^{j_l} ] \, \overline E' \, \omega^- \\
\gamma_l &= (\omega^+ E) \, [\under^{j_1} \, \beta \, ( \under')^{j_l-j_1} \, \bar \beta \, \under^{-j_l}] \, (\omega^+ E)^\inv 
\end{align*}
If the number of edges cancelled when $\under^{j_1} \beta \, ( \under')^{j_l-j_1}  \bar \beta \, \under^{-j_l}$ is tightened is not bounded independently of $l$ then there is a closed path $\alpha$ such that $\under',\beta$ and $\under$ are all iterates of~$\alpha$. The path $\under^{j_1}\beta ( \under')^{j_l-j_1} \bar\beta \under^{-j_l}$ therefore tightens to an iterate of $\alpha$ and so each $\gamma_l$ has the form $(\omega^+ E) \, \alpha^* \, (\omega^+ E)^\inv$. Each $[\tau_1 \bar\tau_l]$ is therefore contained in the cyclic subgroup of $\pi_1(G,v_-)$ generated by $(\omega^+E) \, \alpha \, (\omega^+ E)^{-1}$, and the proof concludes as in the once linear case with $D_0$ any nonempty, finite subset of~$D$. 

It remains to consider the case that the cancellation of $\under^{j_1} \beta \, ( \under')^{j_l-j_1}  \bar \beta \, \under^{-j_l}$ is uniformly bounded. In this case, there are paths $\sigma_1,\sigma_2, \sigma_3$, root-free closed paths $\alpha_1,\alpha_2$, an integer $l_0 \ge 2$, and increasing sequences of positive integers $\{s(l)\}$ and $\{t(l)\}$ defined for $l \ge l_0$, such that if $l \ge l_0$ then $\gamma_l$ can be written in the form
$$\gamma_l = \sigma_1\alpha_1^{s(l)} \sigma_2 \, \alpha_2^{-t(l)}\sigma_3
$$
and such that the following hold: 
\begin{itemize}
\item $\under'$ is an iterate of $\alpha_1$, and $\under$ is an iterate of $\alpha_2$; 
\item The path $\sigma_1$ is obtained by straightening the concatenation of $\omega^+ E \, \under^{j_1} \, \beta$ followed by some initial subpath of $(\under')^*$; 
\item The path $\sigma_2$ is obtained by straightening the concatenation of some terminal subpath of $(\under')^*$, followed by the path $\bar\beta$, followed by some initial subpath of~$(\under^\inv)^*$; 
\item The path $\sigma_3$ is obtained by straightening the concatenation of some terminal subpath of $(\under^\inv)^*$ followed by~$\overline E \bar \omega^+$.
\item The maximal common initial subpath of $\sigma_2$ and $\alpha_1$ is a proper initial subpath of~$\alpha_1$; we denote this subpath $\sigma'_2$.
\end{itemize}
Let 
$$D_0 = \{[\tau_1 \bar\tau_l] \suchthat 2 \le l \le l_0 + 2n-1\} \subset \pi_1(G,v_-)
$$
Assuming that $D_0 \subset B$ we complete the proof by showing that the path $\sigma_1 \alpha_1^s \sigma_2 \alpha_2^t \sigma_3$ represents an element of $B$ for all $s, t \in \Z$. 
 
Since $B \subgroup \pi_1(G,v_-)$ has rank~$\le n$, there is a finite core graph $X$ of rank~$\le n$ equipped with a basepoint $x$ and an immersion $Q : (X,x) \to (G,v_-)$ such that a closed path in $G$ based at $v_-$ lifts to a closed path in $X$ based at $x$ if and only if it represents an element of $B$. Equivalently, the induced homomorphism $Q_* \from \pi_1(X,x) \to \pi_1(G,v_-)$ is an injection with image~$B$. The path $\gamma_{l_0+2n-1}$ representing $[\tau_1 \bar\tau_{l_0+2n-1}] \in D_0$ lifts to a closed path $\tilde\gamma_{l_0+2n-1}$ in $X$ based at $x$. We have the following forms for a decomposition of $\gamma_{l_0+2n-1}$ and the corresponding decomposition of $\tilde\gamma_{l_0+2n-1}$:
\begin{align*}
\gamma_{l_0+2n-1} &= \sigma_1 \,\, \underbrace{\alpha_1 \,\, \alpha_1 \,\, \cdots \,\, \alpha_1}_{\text{$s(l_0+2n-1)$ times}} \,\, \sigma_2 \,\, \underbrace{\alpha^\inv_2 \,\, \alpha^\inv_2 \,\, \cdots \,\, \alpha^\inv_2}_{\text{$t(l_0+2n-1)$ times}} \,\, \sigma_3 \\
\tilde\gamma_{l_0+2n-1} &= \tilde\sigma_1 \,\, \tilde\alpha_{1,1} \,\, \tilde\alpha_{1,2} \,\,  \ldots \,\, \tilde \alpha_{1,s(l_0+2n-1)} \,\, \tilde\sigma_2 \,\, \tilde\alpha^{-1}_{2,1} \,\, \tilde\alpha^{-1}_{2,2} \,\, \ldots \,\, \tilde\alpha^{-1}_{2,t(l_0+2n-1)} \,\, \tilde\sigma_3
\end{align*}
In $X$ denote the following vertices, where $1 \le i < s(l_0+2n-1)-1$:
\begin{align*}
y_0 &= \text{the terminal vertex of $\tilde\sigma_1$} =  \text{the initial vertex of $\tilde\alpha_{1,1}$} \\
y_i &= \text{the terminal vertex of $\tilde\alpha_{1,i}$} = \text{the initial vertex of $\tilde\alpha_{1,i+1}$} \\
\tilde\sigma'_{2,i} &= \text{the unique lift of $\sigma'_2$ with initial vertex $y_i$} \\
z_i &= \text{the terminal vertex of $\ti\sigma'_{2,i}$}
\end{align*}
Also, for $l_0 \le l \le l_0 + 2n-2$, the initial and terminal vertices of $\tilde\sigma'_{2,s(l)}$ are denoted $\eta_l = y_{s(l)}$ and let $\zeta_l = z_{s(l)}$. 

The vertex $\zeta_{l_0}$ has valence $\ge 3$ in $X$, because both of the paths $\gamma_{l_0}$ and $\gamma_{l_0+1}$ lift to paths in $X$ starting at~$x$, these two lifts have maximal common initial subpath $\tilde\sigma_1 \, \tilde\alpha_{1,1} \, \ldots \, \ti\alpha_{1,s(l_0)} \, \tilde\sigma'_{2,s(l_0)}$, this subpath is proper in each of $\gamma_{l_0}$ and $\gamma_{l_0+1}$, and the terminal endpoint of this subpath is $\zeta_{l_0}$; this follows by maximality of $\sigma'_2$. Similarly each of the vertices $\zeta_{l_0+1},\ldots,\zeta_{l_0+2n-2}$ has valence $\ge 3$ in $X$. 

Since $X$ has at most $2n-2$ distinct vertices of valence $\ge 3$, it must be that there is repetition amongst the vertices $\zeta_{l_0},\ldots,\zeta_{l_0+2n-2}$. Let $a<b$ be indices such that $\zeta_{l_a} = \zeta_{l_b}$. It follows that $y_{s(l_a)} = \eta_{l_a} = \eta_{l_b} = y_{s(l_b)}$. Since there is at most one lift of $\alpha_1$ to $X$ that terminates at each $y_i$, it must be that $y_0 = y_{s(l_b)-s(l_a)}$. This implies that some iterate of $\alpha_1$ lifts to a closed path in $X$ based at $y_0$ and hence that some straightened iterate of $\sigma_1 \alpha_1 \bar \sigma_1$ lifts to a closed path in $X$ based at $x$. Since $B$ is primitive, the path $\sigma_1 \alpha_1 \bar \sigma_1$ itself lifts to a closed path based at $x$ and so determines an element of~$B$. 
 
The symmetric argument shows that $\bar \sigma_3 \alpha_2 \sigma_3$ determines an element of $B$ and hence that each path 
$$\sigma_1\alpha_1^s \sigma_2 \alpha_2^t\sigma_3 = [(\sigma_1 \alpha_1^{s- s_2}\bar \sigma_1)(\sigma_1\alpha_1^{s_2 } \sigma_2\alpha_2^{-t_2}\sigma_3)(\bar \sigma_3 \alpha_2^{t+t_2} \sigma_3)]
$$ 
determines an element of $B$ for each $s,t \in \Z$ as needed to complete the proof.
\qed

\bibliographystyle{amsalpha} 
\bibliography{mosher} 

\end{document}